\documentclass[10pt]{amsart}

\usepackage{amsmath,amssymb,amsxtra,anysize, tikz,hyperref,tikz-cd,graphicx,combelow}

\usepackage{comment}
\usepackage{enumerate}
\usepackage{enumitem}
\usepackage{tikz}
\usepackage{tikz-cd}
\usetikzlibrary{decorations.markings}
\usepackage{hyperref}
\usepackage{amsthm,amsfonts,amsmath,amscd,amssymb}
\usepackage{xcolor,float}
\usepackage[all]{xy}
\usepackage{mathrsfs}
\usepackage{graphics,graphicx}
\usepackage{caption}
\usepackage{subcaption}
\usepackage{mathtools}

\let\oldmarginpar\marginpar
\renewcommand\marginpar[1]{\-\oldmarginpar[\raggedleft\footnotesize #1]%
	{\raggedright\footnotesize #1}}
\theoremstyle{plain}
\newtheorem{thm}{Theorem}[section]
\newtheorem{lemma}[thm]{Lemma}
\newtheorem{example}[thm]{Example}
\newtheorem{prop}[thm]{Proposition}

\newtheorem{cor}[thm]{Corollary}

\newtheorem{conj}[thm]{Conjecture}
\theoremstyle{definition}

\newtheorem{definition}[thm]{Definition}
\newtheorem{remark}[thm]{Remark}

\theoremstyle{remark}

\numberwithin{equation}{section}

\newcommand{\PP}{\mathbb{P}}

\renewcommand{\P}{\mathbb{P}}

\newcommand{\N}{\mathbb{N}}
\newcommand{\Z}{\mathbb{Z}}

\newcommand{\R}{\mathbb{R}}
\newcommand{\C}{\mathbb{C}}
\newcommand{\SA}{\mathcal{A}}

\newcommand{\SB}{\mathscr{B}}

\newcommand{\fM}{\mathfrak{M}}
\newcommand{\La}{\Lambda}
\newcommand{\la}{\lambda}

\newcommand{\sse}{\subseteq}
\newcommand{\lr}{\longrightarrow}

\newcommand{\x}{\times}

\newcommand{\Cont}{\operatorname{Cont}}
\newcommand{\GL}{\operatorname{GL}}

\newcommand{\Aug}{\operatorname{Aug}}

\newcommand{\st}{\text{st}}
\newcommand{\std}{\text{st}}

\newcounter{daggerfootnote}

\newcommand{\BM}{\mathrm{BM}}
\newcommand{\SBim}{\mathrm{SBim}}

\newcommand{\bC}{\mathbb{C}}
\renewcommand{\C}{\mathbb{C}}

\newcommand{\bQ}{\mathbb{Q}}
\newcommand{\bR}{\mathbb{R}}

\newcommand{\cB}{\mathcal{B}}

\newcommand{\cF}{\mathscr{F}}

\newcommand{\cL}{\mathcal{L}}

\newcommand{\HHH}{\mathrm{HHH}}

\newcommand{\Hom}{\mathrm{Hom}}

\newcommand{\Fl}{\mathscr{F}\ell}
\newcommand{\brick}{\operatorname{brick}}

\newcommand{\Br}{\mathrm{Br}}

\newcommand{\schub}{\buildrel\circ \over X}
\newcommand{\rich}[2]{\mathcal{R}^{\circ}\!(#2,#1)}

\DeclareMathOperator{\ant}{ant}
\DeclareMathOperator{\inv}{inv}
\DeclareMathOperator{\BS}{BS} 
\DeclareMathOperator{\OBS}{OBS} 
\DeclareMathOperator{\Gr}{Gr}

\newcommand{\s}{\sigma}
\newcommand{\Ja}{J^{(1)}}
\newcommand{\Jb}{J^{(2)}}

\newcommand{\wire}[1]{\mathcal{W}(#1)}

\newcommand{\JS}[1]{{\color{cyan}{{\bf Jos\'e}: #1}}}
\newcommand{\MG}[1]{{\color{teal}{{\bf Misha}: #1}}}

 \usepackage{epigraph}
 
   \DeclareFontFamily{U}{wncy}{}
\DeclareFontShape{U}{wncy}{m}{n}{<->wncyr10}{}
\DeclareSymbolFont{mcy}{U}{wncy}{m}{n}
\DeclareMathSymbol{\Sh}{\mathord}{mcy}{"58} 

\DeclareRobustCommand{\ateb}{\text{\reflectbox{$\beta$}}}
\DeclareRobustCommand{\um}{\text{\reflectbox{$\mu$}}}
\DeclareRobustCommand{\ate}{\text{\reflectbox{$\eta$}}}
\def\Le{\reflectbox{$\mathrm{L}$}}

\newcommand{\bdot}[2]{
\draw [black, fill = black] (#1, #2) circle [radius = 0.1];
}

\newcommand{\wdot}[2]{
\draw [black, fill = white] (#1, #2) circle [radius = 0.1];
}

\newcommand{\sdot}[2]{
\draw [black, fill=black] (#1, #2) circle [radius = 0.05]
}

\newcommand{\edge}[3]{
\draw (#1,#2) -- (#1, #2 + #3);
\bdot{#1}{#2}
\wdot{#1}{#2 + #3}
}

\newcommand{\dcrsg}[2]{
\begin{scope}[xshift = #1 cm, yshift = #2 cm]
\draw [dashed, rounded corners] (0,0) -- (0.2,0) -- (0.8,1)-- (1,1);
\draw [dashed, rounded corners] (0,1) -- (0.2,1) -- (0.8,0) -- (1,0);\
\end{scope}
}

\newcommand{\crsg}[2]{
\begin{scope}[xshift = #1 cm, yshift = #2 cm]
\draw [rounded corners] (0,0) -- (0.2,0) -- (0.8,1)-- (1,1);
\draw [rounded corners] (0,1) -- (0.2,1) -- (0.8,0) -- (1,0);\
\end{scope}
}

\newcommand{\rcrsg}[2]{
\begin{scope}[xshift = #1 cm, yshift = #2 cm]
\draw [rounded corners] (1,0) -- (0.8,0) -- (0.5,0.5);
\draw [rounded corners] (1,1) -- (0.8,1) -- (0.5,0.5);
\draw [dashed,rounded corners] (0.5,0.5) -- (0.2,0) -- (0, 0);
\draw [dashed,rounded corners] (0.5,0.5) -- (0.2,1) -- (0, 1);\
\end{scope}
}

\newcommand{\brcrsg}[2]{
\begin{scope}[xshift = #1 cm, yshift = #2 cm]
\draw [rounded corners] (0,1) -- (0.2,1) -- (0.8,0)-- (1,0);
\draw [rounded corners] (1,1) -- (0.8,1) -- (0.5,0.5);
\draw [dashed,rounded corners] (0.5,0.5) -- (0.2,0) -- (0, 0);\
\end{scope}
}

\newcommand{\trcrsg}[2]{
\begin{scope}[xshift = #1 cm, yshift = #2 cm]
\draw [rounded corners] (1,0) -- (0.8,0) -- (0.5,0.5);
\draw [rounded corners] (0,0) -- (0.2,0) -- (0.8,1) -- (1,1);
\draw [dashed,rounded corners] (0.5,0.5) -- (0.2,1) -- (0, 1);\
\end{scope}
}

\newcommand{\trcrsgbr}[2]{
\begin{scope}[xshift = #1 cm, yshift = #2 cm]
\draw [dashed, rounded corners] (1,0) -- (0.8,0) -- (0.5,0.5);
\draw [rounded corners] (0,0) -- (0.2,0) -- (0.8,1) -- (1,1);
\draw [dashed,rounded corners] (0.5,0.5) -- (0.2,1) -- (0, 1);\
\end{scope}
}

\newcommand{\crsgr}[2]{
\begin{scope}[xshift = #1 cm, yshift = #2 cm]
\draw [rounded corners] (0,0) -- (0.2,0) -- (0.5,0.5);
\draw [rounded corners] (0,1) -- (0.2,1) -- (0.5,0.5);
\draw [dashed,rounded corners] (0.5,0.5) -- (0.8,0) -- (1, 0);
\draw [dashed,rounded corners] (0.5,0.5) -- (0.8,1) -- (1, 1);\
\end{scope}
}

\newcommand{\crsgbr}[2]{
\begin{scope}[xshift = #1 cm, yshift = #2 cm]
\draw [rounded corners] (0,0) -- (0.2,0) -- (0.8,1)-- (1,1);
\draw [rounded corners] (0,1) -- (0.2,1) -- (0.5,0.5);
\draw [dashed,rounded corners] (0.5,0.5) -- (0.8,0) -- (1, 0);\
\end{scope}
}

\newcommand{\btm}[2]{
\begin{scope}[xshift = #1 cm, yshift = #2 cm]
\draw (0,0) -- (1,0);
\end{scope}
}

\newcommand{\dbtm}[2]{
\begin{scope}[xshift = #1 cm, yshift = #2 cm]
\draw[dashed] (0,0) -- (1,0);
\end{scope}
}

\newcommand{\tp}[2]{
\begin{scope}[xshift = #1 cm, yshift = #2 cm]
\draw (0,1) -- (1,1);
\end{scope}
}

\newcommand{\dtp}[2]{
\begin{scope}[xshift = #1 cm, yshift = #2 cm]
\draw[dashed] (0,1) -- (1,1);
\end{scope}
}

\newcommand{\dmr}[2]{
\begin{scope}[xshift = #1 cm, yshift = #2 cm]
\draw (0,0) -- (1,0);
\draw (0,1) -- (1,1);
\edge{0.5}{0}{1};
\end{scope}
}

\newcommand{\hdmr}[2]{
\begin{scope}[xshift = #1 cm, yshift = #2 cm]
\draw (0.5,0) -- (1,0);
\draw (0.5,1) -- (1,1);
\edge{0.5}{0}{1};
\end{scope}
}

\newcommand{\tdmr}[2]{
\begin{scope}[xshift = #1 cm, yshift = #2 cm]
\draw (0.5,0) -- (1,0);
\draw (0,1) -- (1,1);
\edge{0.5}{0}{1};
\end{scope}
}

\newcommand{\bdmr}[2]{
\begin{scope}[xshift = #1 cm, yshift = #2 cm]
\draw (0,0) -- (1,0);
\draw (0.5,1) -- (1,1);
\edge{0.5}{0}{1};
\end{scope}
}

\newcommand{\bcrsg}[2]{
\begin{scope}[xshift = #1 cm, yshift = #2 cm]
\draw [rounded corners] (0,0) -- (0.2,0) -- (0.8,1)-- (1,1);
\end{scope}
}

\title{Positroid Links and Braid varieties}

\author[R. Casals]{Roger Casals}
\address{Dept. of Mathematics\\ University of California, Davis\\ One Shields Avenue, Davis, CA, USA}
\email{casals@math.ucdavis.edu}
\author[E. Gorsky]{Eugene Gorsky}
\address{Dept. of Mathematics\\ University of California, Davis\\ One Shields Avenue, Davis, CA, USA}
\email{egorskiy@math.ucdavis.edu}
\author[M. Gorsky]{Mikhail Gorsky}
\address{Department of Mathematics,
University of Vienna,
Oscar-Morgenstern Platz~1,
1090 Vienna,
Austria
\\
\newline and Universit\"at Hamburg, Fachbereich Mathematik, Bundesstraße 55, 20146 Hamburg, Germany\\ 
\newline and Institut Camille Jordan UMR 5208, Université Jean Monnet, CNRS, Centrale Lyon, INSA Lyon, Université Claude Bernard Lyon 1, 20, rue Annino, 42023, Saint-Étienne, France.}
\email{mikhail.gorskii@univie.ac.at}
\author[J. Simental]{Jos\'e Simental}
\address{Instituto de Matem\'aticas, Universidad Nacional Aut\'onoma de M\'exico. Mexico City, Mexico.}
\email{simental@im.unam.mx}

\usepackage{amsfonts}
\input{cyracc.def}

\font\tencyrit=wncyi10
\font\tencyr=wncyr10
\def\cyi{\tencyrit\cyracc}
\def\cyr{\tencyr\cyracc}

\begin{document}

\begin{abstract}
We study braid varieties and their relation to open positroid varieties. We discuss four different types of braids associated to open positroid strata and show that their associated Legendrian links are all Legendrian isotopic. In particular, we prove that each open positroid stratum can be presented as the augmentation variety for four different Legendrian fronts described in terms of either permutations, juggling patterns, cyclic rank matrices or Le diagrams. We also relate braid varieties to open Richardson varieties and brick manifolds, showing that the latter provide projective compactifications of braid varieties, with normal crossing divisors at infinity.
\end{abstract}

\maketitle
\vspace{-0.7cm}


\epigraph{
\cyi Ona lyubila Richardsona\\ Ne potomu, chtoby prochla
}{\cyr A.~S.~Pushkin, \cyi Evgeni\u\i \, Onegin\footnotemark}\footnotetext{{\it The reason she loved Richardson
  was not that she had read him} --- A.S. Pushkin, {\it Eugene Onegin} \, (tr. V. Nabokov).}


\setcounter{tocdepth}{1}
\tableofcontents

\section{Introduction}\label{sec:intro}
This article studies braid varieties \cite{CGGS,Mellit} and their relation to open positroid varieties \cite{KLS}. In a nutshell, we study four braids associated to any open positroid variety, and develop new techniques to algebraically study their braid varieties. 
In addition, this paper brings to bear insight from contact and symplectic topology to explicitly study these braid varieties, with a focus on Legendrian links and their relation to open positroid varieties in Grassmannians.

\noindent An open positroid variety $\Pi$ of the Grassmannian $\Gr(k,n)$ can be indexed by either of the following four pieces of data. First, a pair of permutations $u,w\in S_n$ such that $u\leq w$ in the Bruhat order and $w$ is a $k$-Grassmannian permutation. Second, a $k$-bounded affine permutation $f:\Z\lr\Z$ of size $n$. 
Third, a cyclic rank matrix $r$ and, fourth, a Le diagram. The bijections between these objects and the description of their associated positroid varieties are provided in \cite{KLS, Post}. In this article, we study four braids, one associated to each of these four pieces of data, and introduce and study their associated Legendrian links.

The first result of this manuscript, within the realm of algebraic combinatorics, is showing that these four braids close up to links in $\R^3$ that are smoothly isotopic, up to trivially adding unlinks. For that, we develop new results  using the positroid data above: $k$-Grassmannian permutations, $k$-bounded affine permutations,  cyclic rank matrices and Le diagrams. In particular, this requires addressing the dissonance in the number of strands between these braids, which we address by introducing a Markov-type destabilization move that suits the algebraic combinatorics associated to positroids.

\noindent The main contact geometric result of this manuscript is then showing that the four associated Legendrian links are all Legendrian isotopic, up to trivially unlinked 
unknots. In particular, we show that our results and  constructions in the smooth case, related to the algebraic combinatorics of positroids, can all be realized by contact isotopies. To our knowledge, the conceptual insight that certain Legendrian links, not just smooth links, underlie each of these four presentations of a positroid variety is also new. It has the important consequence of allowing the description and study of positroid strata in terms of contact topology, which has already been initiated in other works with fruitful consequences, cf.~ \cite{asplund2023lagrangian,cgglss,CGGS,CGLSBW,CasalsWeng22}. In particular, \cite[Theorem 1.1]{CGGS} established the relationship between braid varieties and augmentation varieties. The contact isotopies between the four Legendrian links above imply that the four braid varieties associated to these Legendrians are isomorphic to the corresponding positroid variety, up to trivially adding frozens.

\noindent Finally, the article includes new results relating braid varieties to projective brick manifolds and open Richardson varieties. In particular, we show that brick varieties are good projective compactifications of our affine braid varieties. This also allows us to relate their homology to the top $a$-degree Khovanov-Rozansky homology of the underlying smooth link and establish the curious Lefschetz property for open Richardson varieties. The article concludes with a brief discussion on conjectural matters regarding cluster structures and Legendrian links.
\subsection{Scientific Context}\label{ssec:scientific_context}

Positroid varieties first appeared in the study of total positivity \cite{Lusztig1994, Lusztig1998, Post, Rietsch} and in the context of Poisson geometry \cite{BGY}. Let ${\Pi}_{u,w}$ be the open positroid variety of the Grassmannian $\Gr(k,n)$ indexed by a pair of permutations $u,w\in S_n$, where $u\leq w$ in Bruhat order and $w$ is $k$-Grassmannian. We consider the bijections between such pairs $(u,w)$, bounded affine permutations $f:\Z\lr\Z$, cyclic rank matrices $r$,  and Le-diagrams $\Le$ established in \cite{KLS, Post}.

\noindent For instance, the bounded affine permutation $f(u,w):\Z\lr\Z$ corresponding to a pair $(u,w)$ is $f(u,w):=u^{-1}t_kw$, where $t_k$ is the translation by the $k$th fundamental weight; conversely, $f$ recovers $(u,w)$. Here $f$ is interpreted as a bijection $f:\Z\to \Z$ such that $f(i+n)=f(i)+n$ and  $i\le f(i)\le i+n$ for all $i\in\Z$. The four pieces of data $(u,w)$, $f$, $r$ and $\Le$ are said to {\it represent the same positroid type} if they correspond to each other under these bijections. Each piece of data, $(u,w)$, $f$, $r$, and $\Le$, yields an open stratum ${\Pi}_{u,v}$, ${\Pi}_{f}$, ${\Pi}_{r}$, and ${\Pi}_{\Le}$ in $\Gr(k,n)$, and ${\Pi}_{u,v} = {\Pi}_{f} = {\Pi}_{r} = {\Pi}_{\Le}$ if $(u,w)$, $f$, $r$, and $\Le$ represent the same positroid type, cf.~\cite{KLS}.

In Section \ref{sec:RichardsonJuggling} we explain how each of these pieces of data, $(u,w)$, $f$, $r$ and $\Le$, also yields a {\it braid word}. In consequence, we can associate braids and links to each such four types of positroid data. These four braids, which we correspondingly denote $R_n(u,v)$, $J_k(f)$, $M_k(r)$ and $D_{k}(\Le)$, are studied in detail in this article. Either of these four braids will be referred to as a {\it positroid braid}. We also connect the results in Section \ref{sec:RichardsonJuggling} to previous works in the literature, including \cite{GL,KLS,STWZ}.

In Section \ref{sec:Legendrian} we associate a {\it Legendrian link} to each such positroid braid. This has an important consequence: we can construct the corresponding positroid strata in a contact geometric manner. Namely, for the Legendrian links we construct from positroid braids, the corresponding positroid stratum is recovered as a Legendrian invariant. Specifically, as the spectrum of the $0$th homology of the Legendrian contact dg-algebra. Therefore,  it becomes a central question whether these Legendrian links are Legendrian isotopic if they are obtained from data representing the same positroid type. This is the content of Section \ref{sec:Legendrian}, where we develop the necessary results to show that this is the case. Note that these Legendrian links and their connection to positroid strata and their cluster algebras have already featured in the recent preprints \cite{asplund2023lagrangian,cgglss,CGLSBW,glsbs,glsb}.

In Section \ref{sec:brick} we study the braid varieties associated to these positroid braids. Braid varieties have featured prominently in the series of articles \cite{CGGS,cgglss,glsbs,glsb}, where their cluster structures are studied. The present manuscript is the second part of a trilogy: the first part is \cite{CGGS}, where braid varieties were studied through weaves, and third part is \cite{cgglss}, which establishes the general existence of cluster algebra structures on braid varieties. The current article  studies some relevant algebraic geometric aspects of braid varieties  associated to positroid data in Section \ref{sec:brick}. These include  their relation to open Richardson varieties, the construction of smooth projective compactifications with normal crossing divisors, and the computation of their torus-equivariant homology, among others.

\subsection{Main Results} By definition, two positive braid words $\beta_1,\beta_2$ are said to be equivalent if they represent the same element in the braid group $\Br_m$. By \cite{Garside69}, two equivalent $\beta_1,\beta_2$ represent the same element in the braid monoid $\Br_m^+$.

\noindent The first result, proven in Section \ref{sec:RichardsonJuggling},  establishes the relation between the four types of positroid braids $R_n(u,v)$, $J_k(f)$, $M_k(r)$ and $D_{k}(\Le)$.

\begin{thm}\label{thm:main1}
Let $u,w\in S_n$ be such that $u\leq w$ in the Bruhat order and $w$ is a $k$-Grassmannian permutation, $f$ a bounded affine permutation, $r$ a cyclic rank matrix, and $\Le$ a Le-diagram. Suppose that these four pieces of data represent the same positroid type. Then

\begin{itemize}
	\item[(i)] The $n$-stranded braid $\Delta_nR_n(u,w)$ and the $k$-stranded braid $\Delta_kJ_k(f)$ are equivalent, up to positive Markov stabilizations and adding unlinked disjoint strands. \\

    \item[(ii)] The $k$-stranded braid $J_k(f)$ and the $k$-stranded braid $\Delta_kD_k(\Le)$ are equivalent.\\
 
	\item[(iii)] The $k$-stranded positive braids $\Delta_kJ_k(f)$ and $M_k(r)$ are equivalent.\hfill$\Box$
\end{itemize}
\end{thm}

\noindent Note that Theorem \ref{thm:main1}.(i) relates two braids, $R_n(u,w)$ and $J_k(f)$, on 
a different number of strands. Section \ref{ssec:destabilization} develops a Markov-type destabilization  which is well-suited for  comparing the different types of algebraic combinatorics related to positroids.

The second result, established in Section \ref{sec:Legendrian}, is a contact geometric counterpart of Theorem \ref{thm:main1}. Section \ref{ssec:Leglinks} introduces four Legendrian links $\La(u,w),\La(f),\La(r)$ and $\La(\Le)$ in $(\R^3,\xi_\st)$, each one associated to a different type of positroid data. 

\begin{thm}\label{thm:main2}
Let $u,w\in S_n$ be such that $u\leq w$ in the Bruhat order and $w$ is a $k$-Grassmannian permutation, $f$ a bounded affine permutation, $r$ a cyclic rank matrix, and $\Le$ a Le-diagram. Suppose that these four pieces of data represent the same positroid type.

\noindent Then the four Legendrian positroid links $\La(u,w),\La(f),\La(r),\La(\Le)\sse(\R^3,\xi_\st)$ are Legendrian isotopic, up to unlinked max-tb Legendrian unknots.\hfill$\Box$
\end{thm}

\noindent An important consequence of Theorem \ref{thm:main2} is that the Legendrian contact dg-algebras associated to each of these Legendrian links are stable tame isomorphic. In particular, the spectra of their $0$th  homology algebras are isomorphic up to torus factors and they coincide with the corresponding positroid stratum, also up to torus factors. This provides an intrinsic and geometric way to recover positroids from these Legendrian links.

The third result studies braid varieties associated to positroid braids. In particular,  it relates braid varieties to open Richardson varieties. Braid varieties $X(\beta;w)$ associated to a braid (word) $\beta$ and a permutation $w$ were introduced in \cite{CGGS}, their definition is recalled in Section \ref{ssec:braid_prelim} below. The proof of the following result uses Theorem \ref{thm:main2} together with \cite{KLS} to show that any positroid variety $\Pi_{u, w}$ in the Grassmannian $\Gr(k, n)$ can be expressed in terms of braid varieties, either using the $n$-stranded braid $R_{n}(u, w)$ or the $k$-stranded braid $J_{k}(f)$.

\begin{thm}\label{thm:rich vs juggling intro}
Let $u,w\in S_{n}$ with  $u\leq w$ in Bruhat order, $w$ a $k$-Grassmannian permutation, and $f := ut_{k}w^{-1}$ the corresponding $k$-bounded affine permutation. Then we have algebraic isomorphisms 
\begin{itemize}
    \item[(i)] $\Pi_{u, w} 
\cong X(\beta(u^{-1}w_{0,n})\beta(w_{0,n}ww_{0,n});w_{0,n})$;\\

\item[(ii)] $X(\beta(u^{-1}w_{0,n})\beta(w_{0,n}ww_{0,n});w_{0,n})\cong  X(J_{k}(f);w_{0,k}) \times (\bC^{*})^{n-k-\varphi}$
\end{itemize} 
of affine algebraic varieties, where $\varphi := \#\{i \in [1,n] : f(i) = i\}$ is the number of fixed points of $f$.\hfill$\Box$ 
\end{thm}

\noindent Theorem \ref{thm:rich vs juggling intro} is proven in Section \ref{ssec:Richardson_braidvar}. Section \ref{sec:brick} offers a fourth result as well. Section \ref{ssec:compactify} shows that the brick manifolds introduced in \cite[Definition 3.2]{Escobar} provide smooth projective compactifications of braid varieties. 
We denote the brick manifold of $\beta$ by $\brick(\beta)$ and its maximal open stratum by $\brick^{\circ}(\beta)$, cf.~\cite{Escobar}. The precise relation we establish between braid varieties and brick manifolds 
is the following:

\begin{thm}	\label{thm: intro brick} Let $\beta = \sigma_{i_{1}}\cdots\sigma_{i_{\ell}}\in\cB_n$ be a positive braid word, $\ateb\in\cB_n$ its opposite, $\delta(\beta)$ the Demazure product of $\beta$, and consider the truncations $\beta_{j} := \sigma_{i_{1}}\cdots \sigma_{i_{j}}$, $j\in[1,\ell]$. The following holds:
	
	\begin{itemize}
		\item[$(i)$] The algebraic map $\Theta:\bC^{\ell}\lr\Fl_n^{\ell+1}$, $(z_1, \dots, z_{\ell}) \mapsto (\cF^{\std}, \cF^{1}, \dots, \cF^{\ell}),$
		where $\cF^{j}$ is the flag associated to the matrix $B_{\ateb_{j}}^{-1}(z_{\ell - j + 1}, \dots, z_{\ell})$, restricts to an algebraic isomorphism
		$$\Theta:	X(\ateb; \delta(\beta)) \stackrel{\cong}{\lr} \brick^{\circ}(\beta).$$
				
		\item[$(ii)$] 
		The complement to $X(\ateb; \delta(\beta))$ in $\brick(\beta)$ is a normal crossing divisor. Its components correspond to all possible ways to remove a letter from $\ateb$ while preserving its Demazure product.
	\end{itemize}
In particular, $\brick(\beta)$ is a smooth projective good compactification of the affine variety $X(\ateb;\delta(\beta))$.
\end{thm}

\noindent Note that $\brick(\beta)$ depends on the choice of braid word $\beta\in\cB^+$, and not just the braid element $[\beta]\in\Br^+$, whereas $X(\beta;w_0)$ only depends on the positive braid $[\beta]$. Therefore, Theorem \ref{thm: intro brick}.(ii) can be used to construct several smooth projective SNC compactifications of the same braid variety.

\noindent In addition, Theorem \ref{thm: intro brick},  in combination with \cite{Escobar}, clarifies the connection between braid varieties and the combinatorics of subword complexes. This allows us to translate properties of spherical subword complexes via brick manifolds to braid varieties, and vice versa.

\noindent Finally, Section \ref{ssec:equiv_hom} explains how to compute the torus-equivariant homology of braid varieties associated to positroids. The article concludes in Section \ref{sec:cluster} with a brief discussion on conjectural matters.\\

\noindent {\bf Acknowledgments}: We are grateful to Taras Panov for his questions on possible compactifications of braid varieties and to Laura Escobar for her interest in our previous work \cite{CGGS}; these led to the conception of Section 4. We also thank Pavel Galashin, Thomas Lam, Anton Mellit and Minh-Tam Trinh for useful discussions, and Etienne M\'enard for his comments. E. G. would like to thank Melody Chan, Yifan Guo and Ethan Partida for correcting a mistake in an earlier version of Lemma \ref{lem: homology brick}. Finally, we are grateful to the referee for their helpful comments, including catching an error in the first version: the manuscript has improved thanks to them.

R.~Casals is supported by the National Science Foundation under grants DMS-2505760 and DMS-1942363, a UC Davis College of L\&S Dean's Fellowship, and the Alfred P. Sloan Foundation.
E.~Gorsky is supported by the NSF FRG grant DMS-1760329.
M.~Gorsky was partially supported by the French ANR grant CHARMS (ANR-19-CE40-0017), and by the Deutsche Forschungsgemeinschaft SFB 1624 ``Higher structures, moduli spaces and integrability'' (506632645), and received funding from the European Research Council (ERC) under the European Union’s Horizon 2020 research and innovation programme (grant agreement No. 101001159). Parts of this work were done during his stay at the University of Stuttgart, and he is very grateful to Steffen Koenig for the hospitality. J.~Simental was partially supported by SECIHTI project CF-2023-G-106 and UNAM'S PAPIIT Grant IA102124, and is grateful for the financial support and hospitality of the Max Planck Institute for Mathematics, where parts of this work were carried out.\hfill$\Box$\\

\noindent {\bf Notational Conventions}: Here are our notational conventions and a comparison to those in the existing literature. As usual, $S_n$ is the symmetric group and $s_i \in S_n$ is the simple transposition that just swaps $i$ and $i+1$. We multiply permutations as we compose functions. For example, $s_2s_1\in S_3$ is the permutation $s_2s_1(1) = 3, s_2s_1(2) = 1, s_2s_1(3) = 2$. Our notion of a $k$-Grassmannian permutation coincides with that of \cite{KLS} but differs from the one in \cite{GL}. A reason to choose this convention is Lemma \ref{lem: projection schubert} below.

Let $\Br_n$ be the braid group in $n$ strands with Artin generators $\sigma_1, \dots, \sigma_{n-1}$. Let $\Br^+_{n} \sse \Br_n$ be the monoid of positive braids. Braids are multiplied so that the map $\sigma_{i} \mapsto s_i$, $\Br_n \to S_n$ is a group homomorphism. When drawing braid diagrams, the convention is that the strands are enumerated by $1, \dots, n$ from top to bottom. Due to the convention above, we draw the braid diagram of a braid word as follows: we read the crossings (generators) of the braid word right-to-left and we draw them in the braid diagram left-to-right.
Thus, the following is a picture of the braid diagram for the braid word $\sigma_2\sigma_1 \in \Br_3$:
\begin{center}
    	\begin{tikzpicture}[scale = 0.6]
  \draw node at (5.5,1.5) {$3$};
  \draw node at (5.5,2.5) {$2$};
  \draw node at (5.5,3.5)  {$1$};

  \draw[out=0, in=180] (6, 3.5) to (8, 2.5);
  \draw[out=0, in=180] (6, 2.5) to (8, 3.5);
  \draw[out=0, in=180] (6, 1.5) to (8, 1.5);

  \draw[out=0, in=180] (8, 3.5) to (10, 3.5);
  \draw[out=0, in=180] (8, 2.5) to (10, 1.5);
  \draw[out=0, in=180] (8, 1.5) to (10, 2.5);

  \draw node at (7, 0.5) {$\sigma_1$};
  \draw node at (9, 0.5) {$\sigma_2$};

  \draw node at (3.5, 2.5) {\Large{$\sigma_2\sigma_1 =$}};
  \end{tikzpicture}
  \end{center}

\noindent Note that the underlying permutation is indeed $s_2s_1$. We denote by $\cB_n$ the set of braid {\it words} in the $n$ Artin generators $\sigma_1,\ldots,\sigma_{n-1}$, of $\Br_n$, and $\cB^+_n$ the set of {\it positive} braid words. Two braids are said to be \emph{equivalent} if they represent the same element in the braid group. Equivalently, if they are represented by braid words which are related by a sequence of braid moves. The set of braids equivalent to the braid represented by a braid word $\beta\in\cB_n$ is denoted by $[\beta]$. Given a braid word $\beta$, read left to right, its {\it opposite} $\ateb$ is defined to be the braid word $\beta$ read right to left, i.e. in reverse order. The half-twist word we use will be denoted by
$$\Delta_n:=(\sigma_1)(\sigma_2\sigma_1)(\sigma_3\sigma_2\sigma_1)\cdot\ldots\cdot(\sigma_{n-1}\sigma_{n-2}\ldots\sigma_2\sigma_1)\in\cB^+_n,$$
and its Coxeter projection is denoted by $w_{0,n}\in S_n$. If $u,w\in S_n$, we denote $u\leq w$ if $u$ is less than $w$ in the Bruhat order. The tautological braid lift of a reduced expression of a permutaion $w\in S_n$ to a braid word in $\Br$ is denoted by $\beta(w)\in \Br$ or simply $w\in \Br$. 
For a permutation $w \in S_n$, we denote by $w^*$ its conjugation by the longest element: $w^* := w_{0, n} w w_{0, n}$. In particular, we have $s_i^* = s_{n-i}$.
Finally, the braid matrices we use in Section \ref{sec:brick} coincide with those used in \cite{CGGS, CasalsNg} but differ from those used in \cite{cgglss}. The two conventions differ by taking inverse matrices. 

Finally, sometimes (in particular, in Section \ref{ssec:JugglingBraid}) we will consider  permutation braids where the strands are labeled in a non-standard way. In this case, we clearly state the labels and their order {\it both} on the left and on the right, and how the left labels are connected to the right ones. This determines a permutation braid uniquely up to braid moves. 

In particular, we define {\it interval} braids as follows.
An interval braid on $k$ strands is a braid of the form $\sigma_i\sigma_{i-1}\cdots \sigma_{j}$ for some $1 \leq j \leq i \leq k - 1$. When depicting such braids diagrammatically, we label strands on the left and on the right by a pair of $k$-element subsets of $[1,n]$ differing in precisely one element, in the decreasing order from top to bottom on both sides. 
Explicitly, assume that $a_1>\ldots>a_{k-1}$, $a_{j-1}>b>a_j$ for some $1 \leq j \leq k-1$ or $b > a_1$, and $a_i>c>a_{i+1}$ for some $1 \leq j \leq k-2$, $c >a_1$ or $a_{k-1} > c$. 
The interval braid has labels $a_1,\ldots,a_{j-1},b,a_j,\ldots,a_{k-1}$ on the left and $a_1,\ldots,a_i,c,a_{i+1},\ldots,a_{k-1}$ on the right. The labels $a_1,\ldots,a_{k-1}$ are connected with the namesake labels on the right, while $b$ on the left is connected to $c$ on the right. The resulting permutation braid has $i-j$ crossings. See Figure \ref{fig: interval braid} for a visual description. 
\hfill$\Box$

\begin{figure}
    \centering
   \begin{center}
    	\begin{tikzpicture}[scale = 0.6]

     \draw node at (5, -3.5) {$a_{k-1}$};
     \draw node at (5, -2.5) {$\vdots$};
     \draw node at (5, -1.5) {$a_{i+1}$};
  \draw node at (5, -0.5) {$a_{i}$};
  \draw node at (5,0.5) {$a_{i-1}$};
  \draw node at (5, 1.5) {$\vdots$};
  \draw node at (5,2.5) {$a_{j}$};
  \draw node at (5,3.5)  {$b$};
  \draw node at (5, 4.5) {$a_{j-1}$};
  \draw node at (5, 5.5) {$\vdots$};
  \draw node at (5, 6.5) {$a_1$};

 \draw node at (9, -3.5) {$a_{k-1}$};
     \draw node at (9, -2.5) {$\vdots$};
     \draw node at (9, -1.5) {$a_{i+1}$};
  \draw node at (9, -0.5) {$c$};
  \draw node at (9,0.5) {$a_i$};
  \draw node at (9, 1.5) {$\vdots$};
  \draw node at (9,2.5) {$a_{j+1}$};
  \draw node at (9,3.5)  {$a_{j}$};
  \draw node at (9, 4.5) {$a_{j-1}$};
  \draw node at (9, 5.5) {$\vdots$};
  \draw node at (9, 6.5) {$a_1$};

  \draw (6, 3.5) to[out=0, in=180] (8, -0.5);
  \draw (6, 2.5) to[out=0, in=180] (8, 3.5);
  \draw (6, 1.5) to[out=0, in=180] (8, 2.5);
  \draw (6, 0.5) to[out=0, in=180] (8, 1.5);
  \draw (6, -0.5) to[out=0, in=180] (8, 0.5);
  \draw (6, 4.5) to (8, 4.5);
  \draw (6, 6.5) to (8, 6.5);
  \draw (6, -1.5) to (8, -1.5);
  \draw (6, -3.5) to (8, -3.5);
  \end{tikzpicture}
\end{center}
    \caption{
    The interval braid $\sigma_i\sigma_{i-1}\cdots \sigma_{j}$.}
    \label{fig: interval braid}
\end{figure}

\section{Positroid Braids and Equivalences}\label{sec:RichardsonJuggling}

In this section we introduce positroid braids and start discussing equivalences between them. After setting up the necessary combinatorics in Subsection \ref{ssec:Comb}, we study positroid braids as follows:

\begin{itemize}
    \item[(i)] The Richardson braid is presented in Subsection \ref{ssec:RichardsonBraid} and the Juggling braid is presented in Subsection \ref{ssec:JugglingBraid}. The former is an $n$-stranded braid and the latter is a $k$-stranded braid.\\

    \item[(ii)] Subsection \ref{ssec:destabilization} shows how to relate the Richardson braid and the Juggling braid via a sequence of generalized destabilization moves, which are also discussed in that subsection.\\

    \item[(iii)] The Le braid is introduced in Subsection \ref{ssec:LeBraid}, which also establishes its relation with the Juggling braid, and the matrix braid is discussed in Subsection \ref{ssec:MatrixBraids}.
\end{itemize}


\subsection{Combinatorial Data}\label{ssec:Comb} Let us introduce the combinatorial data used to describe positroid braids. Fix two natural numbers $k,n\in\N$ such that $k\leq n$. There are four equivalent families of combinatorial objects indexing open positroid strata that we employ: certain pairs of permutations $u,w\in S_n$, certain bijections $f:\Z\lr\Z$, Le diagrams, and cyclic rank matrices. These are schematically depicted in Figure \ref{fig:CombinatorialData}. The object of this first subsection is to define part of these pieces of combinatorial data and review the bijections we will need. 

\begin{center}
	\begin{figure}[h!]
		\centering
		\includegraphics[width=\textwidth]{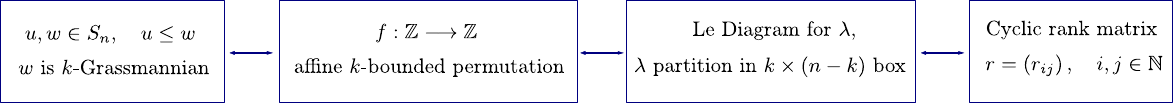}
		\caption{The four types of combinatorial data indexing positroid braids.}
		\label{fig:CombinatorialData}
	\end{figure}
\end{center}

\subsubsection{$k$-Grassmannian permutations and positroid pairs}
\label{subsec: w}
By definition, a permutation $w \in S_n$ is said to be \emph{$k$-Grassmannian} if
\[
w(1) < w(2) < \cdots < w(k), \mbox{ and } w(k+1) < \cdots < w(n).
\]
Similarly, $\Sh \in S_n$ is said to be a \emph{$k$-shuffle} if $\Sh^{-1}$ is $k$-Grassmannian, that is, if 
\[
\Sh^{-1}(1) < \Sh^{-1}(2) < \cdots < \Sh^{-1}(k), \mbox{ and } \Sh^{-1}(k+1) < \cdots < \Sh^{-1}(n).
\]


The set of $k$-Grassmannian permutations in $S_{n}$ (equivalently, the set of $k$-shuffles) is in bijection with the set of partitions $\lambda$ whose Young diagram fits inside a $k \times (n-k)$-rectangle, cf.~ \cite[Section 19]{Post}. We do not distinguish between a partition $\lambda$ and its Young diagram; we draw the latter in French notation. Thus, if the Young diagram of $\lambda$ fits inside a $k \times (n-k)$-rectangle, we write $\lambda \subseteq (n-k)^k$. Such $\lambda$ can be written as
\[
\lambda = (\lambda_{1}, \dots, \lambda_{k})
\]
where $n-k \geq \lambda_1 \geq \cdots \geq \lambda_{k} \geq 0$. Note that it is possible that $\la$ has a zero part, e.g.~ $\lambda_{k} = 0$ is allowed. Similarly, the transposed partition can be written as
\[
\lambda^{t} = (\lambda_{1}^{t}, \dots, \lambda_{n-k}^{t})
\]
where $k \geq \lambda_1^{t} \geq \cdots \geq \lambda_{n-k}^{t} \geq 0$. Again, it is allowed that $\lambda_{n-k}^{t} = 0$; this happens if and only if $n-k > \lambda_{1}$.

For $\lambda \subseteq (n-k)^{k}$, we denote by $w_{\lambda} \in S_{n}$ its associated $k$-Grassmannian permutation. By using one-line notation, we can write
\begin{equation}
\label{eq: w from lambda}
w_{\lambda} = [1 + \lambda_k, 2 + \lambda_{k-1}, \dots, k+\lambda_1, k+1-\lambda_{1}^{t}, k+2 -\lambda_2^{t}, \dots, n - \lambda_{n-k}^{t}].
\end{equation}
Note that the length of $w_{\lambda}$ is $\ell(w_{\lambda}) = |\lambda| := \la_1 + \ldots + \la_{k}$. In fact, we can obtain reduced decompositions for $w_{\lambda}$ as follows:
\begin{equation}\label{eqn:w lambda}
\begin{array}{rl}
w_{\lambda} = &  
(s_{\lambda_{k}}\cdots s_{1})\cdots (s_{k+\lambda_{2}-2}\cdots s_{k-1})(s_{k+\lambda_{1}-1}s_{k+\lambda_{1}-2}\cdots s_{k})
\\ = &  
(s_{n-\lambda_{n-k}^{t}}\cdots s_{n-1})\cdots (s_{k+2 - \lambda_2^{t}}\cdots s_{k}s_{k+1})(s_{k+1-\lambda_{1}^{t}}\cdots s_{k-1}s_{k}).
\end{array}
\end{equation}
 This expression can be read pictorially: we draw the Young diagram $\lambda$ and fill the box in row $i$ and column $j$ with the number $k + j - i$. The first reduced expression in Equation \ref{eqn:w lambda} above is obtained by reading this diagram by rows. The second reduced expression is obtained by reading it by columns. Throughout the paper, we use the convention that empty rows (with $\lambda_i=0$) do not contribute to the first product and the empty columns (with $\lambda^t_i=0$) do not contribute to the second product.

\begin{example}\label{ex:3x4} Let us consider the values $k=3$ and $n=7$ and the Young diagram $\lambda = (4, 3, 1)$. By filling the $(i,j)$-box of $\la$ with $3+(j-i)$ we obtain:
	\begin{center}
		\begin{tikzpicture}[scale=0.45]
		\draw (0,0)--(4,0)--(4,1)--(3,1)--(3,2)--(1,2)--(1,3)--(0,3)--(0,0);
		\draw (0,1)--(3,1);
		\draw (0,2)--(1,2);
		\draw (1,0)--(1,2);
		\draw (2,0)--(2,2);
		\draw (3,0)--(3,1);
		\draw (0.5,0.5) node {$3$};
		\draw (0.5,1.5) node {$2$};
		\draw (0.5,2.5) node {$1$};
		\draw (1.5,0.5) node {$4$};
		\draw (1.5,1.5) node {$3$};
		\draw (2.5, 0.5) node {$5$};
		\draw (2.5, 1.5) node {$4$};
		\draw (3.5, 0.5) node {$6$};
		\end{tikzpicture}
	\end{center}
The associated $3$-Grassmannian permutation is
$$w_{\lambda} = (s_1)(s_4s_3s_2)(s_6s_5s_4s_3)
= (s_6)(s_4s_5)(s_3s_4)(s_1s_2s_3),$$
and note that the length $\ell(w_\lambda)$ is indeed $|\la|=8$.\hfill$\Box$
\end{example}

We can also read the permutation $w_{\lambda}$ as follows. First, we identify a partition $\lambda \subseteq (n-k)^{k}$ with a sequence of $n$ vertical and horizontal steps, that start from the northwest corner of the rectangle and follow the shape of the partition until they reach the southeast corner. For example, the partition in Example \ref{ex:3x4} corresponds to the sequence $(H, V, H, H, V, H, V)$\footnote{Note that the sequence of steps depends on the size of the box $k \times (n-k)$ and not just on the partition $\la$. For example, the partition in Example \ref{ex:3x4} considered inside a $4 \times 4$ box yields the sequence $(V, H, V, H, H, V, H, V)$.}. Enumerate these steps consecutively along the border
of the partition. We will refer to these as the {\em right labels} of $\lambda$. We also enumerate the left border of the rectangle with the numbers $1, \dots, k$, reading top-to-bottom, and the bottom border with the numbers $k+1, \dots, n$, reading left-to-right. We refer to these as the {\em left labels} of $\lambda$.

For each vertical (resp. horizontal) step in the border of $\la$, draw a horizontal (resp. vertical) ray to the left (resp. down). The resulting diagram is known as the \emph{wiring diagram} of $\lambda$, and it gives the permutation $w_{\lambda}$ by mapping the left labels to the right labels
along the rays. 
See Figure \ref{fig:3x4}.

\begin{center}
    \begin{figure}[ht]
\begin{tikzpicture}[scale=0.6]
		\draw (0,0)--(4,0)--(4,1)--(3,1)--(3,2)--(1,2)--(1,3)--(0,3)--(0,0);
		\draw (0,1)--(3,1);
		\draw (0,2)--(1,2);
		\draw (1,0)--(1,2);
		\draw (2,0)--(2,2);
		\draw (3,0)--(3,1);
        \draw (4.2, 0.5) node { $\mathbf{7}$};
        \draw (3.5, 1.3) node {$\mathbf{6}$};
        \draw (3.2, 1.5) node {$\mathbf{5}$};
        \draw (2.5, 2.3) node {$\mathbf{4}$};
        \draw (1.5, 2.3) node {$\mathbf{3}$};
        \draw (1.2, 2.5) node {$\mathbf{2}$};
        \draw (0.5, 3.3) node {$\mathbf{1}$};
        
        \draw (3.5, -0.3) node{$7$}; 
        \draw (2.5, -0.3) node{$6$};
        \draw (1.5, -0.3) node{$5$};
        \draw (0.5, -0.3) node {$4$};
        \draw (-0.2, 0.5) node {$3$};
        \draw (-0.2, 1.5) node {$2$};
        \draw (-0.2, 2.5) node {$1$};

        \draw[<-] (0.5,3)--(0.5,0);
        \draw[<-] (1.5, 2)--(1.5,0);
        \draw[<-] (2.5, 2)--(2.5, 0);
        \draw[<-] (3.5, 1)--(3.5,0);

        \draw[<-] (1, 2.5)--(0,2.5);
        \draw[<-] (3, 1.5)--(0, 1.5);
        \draw[<-] (4, 0.5)--(0,0.5);
		\end{tikzpicture}
  \caption{
  In one-line notation, $w_{\lambda} = [2,5,7,1,3,4,6]$. Right labels shown in bold.}
  \label{fig:3x4}
    \end{figure}
\end{center}

\begin{definition}
A pair of permutations $(u, w)$ with $u, w \in S_{n}$ is said to be a \emph{$k$-positroid pair} if $u \leq w$ in the Bruhat order and $w$ is $k$-Grassmannian. A $k$-positroid pair will be simply referred to as a \emph{positroid pair} if $k$ is understood from context.
\hfill$\Box$
\end{definition}

\subsubsection{Le diagrams} In order to additionally record the data of $u$ in a positroid pair $(u, w)$, $u, w \in S_{n}$, one enhances the Young diagram $\la$ for $w$ into a {\it Le diagram}. Let us recall that the hook of a box in a Young diagram $\lambda$ consists of all the boxes above it (in the same column) as well as all the boxes to its right (in the same row). By definition, a Le diagram $\Le = (\lambda, P)$ consists of a partition $\lambda$ together with a collection $P$ of boxes of $\lambda$, such that any box belonging to the hooks of two different boxes in $P$ must also be in $P$. Pictorially, we depict a Le diagram $\Le$ as a partition $\lambda$ together with dots in some of its boxes, precisely in those that belong to $P$. We thus refer to boxes in $P$ as boxes \emph{with a dot}.

Fixing a partition $\la \subseteq (n-k)^{k}$, \cite[Theorem 19.1]{Post} shows that there is a bijection between Le diagrams with underlying partition $\la$ and elements $u \in S_{n}$ with $u \leq w_{\lambda}$. To a Le diagram $\Le = (\lambda, P)$, the bijection associates the permutation $u$ that is obtained by deleting the simple transpositions corresponding to boxes in $P$ in either of the reduced decompositions \eqref{eqn:w lambda} of $w_{\lambda}$. See Example \ref{ex: markov 44}. Thus, for each $k \in [n]$, there is a bijection between $k$-positroid pairs $(u, w)$ in $S_{n}$ and Le diagrams whose underlying partition fits into a $k \times (n-k)$-rectangle.  

\begin{example}\label{ex: markov 44} 
	Let us consider the values $(k,n)=(4,6)$ and the Young diagram $\lambda=(2,2,2,2)$. The associated permutation is $w_\la=(s_2s_3s_4s_5)(s_1s_2s_3s_4)$.
 Choose the permutation $u=(s_4)(s_2s_3) \in S_6$,
  which satisfies $u\leq w$. 
  The Le diagram associated to this pair $(u,w)$ is drawn on Figure \ref{figure:ex:markov 44}.(B). 
 In one-line notation we have 
 $$
 w^{-1}=[5,6,1,2,3,4],\ w=[3,4,5,6,1,2],
 $$
 $$
 u^{-1}=[1,4,2,5,3,6],\ u=[1,3,5,2,4,6]
 $$

\noindent Note that $w^{-1}$ is a $4$-shuffle in $S_6$.\hfill$\Box$
\begin{center}
\begin{figure}[ht] 
\begin{subfigure}[t]{0.45\textwidth}
\centering
\begin{tikzpicture}[scale = 0.6]
			\draw (0,0)--(0,4)--(2,4)--(2,0)--(0,0);
			\draw (0,1)--(2,1);
			\draw (0,2)--(2,2);
			\draw (0,3)--(2,3);
			\draw (1,0)--(1,4);
			\draw (0.5,0.5) node {$4$};
			\draw (0.5,1.5) node {$3$};
			\draw (0.5,2.5) node {$2$};
			\draw (0.5,3.5) node {$1$};
			\draw (1.5,0.5) node {$5$};
			\draw (1.5,1.5) node {$4$};
			\draw (1.5,2.5) node {$3$};
			\draw (1.5,3.5) node {$2$};
			\end{tikzpicture}
			\caption{The Young diagram $\lambda=(2,2,2,2)$ associated to $w_\la= (s_2s_1)(s_3s_2)(s_4s_3)(s_5s_4) = (s_2s_3s_4s_5)(s_1s_2s_3s_4).$}
			\end{subfigure}
			\hspace{0.3 in}
			\begin{subfigure}[t]{0.4\textwidth}
			\centering
				\begin{tikzpicture}[scale = 0.6]
		\draw (0,0)--(0,4)--(2,4)--(2,0)--(0,0);
		\draw (0,1)--(2,1);
		\draw (0,2)--(2,2);
		\draw (0,3)--(2,3);
		\draw (1,0)--(1,4);
		\draw (0.5,0.5) node {$\bullet$};
		
		\draw (0.5,3.5) node {$\bullet$};
		\draw (1.5,0.5) node {$\bullet$};
		
		\draw (1.5,2.5) node {$\bullet$};
		\draw (1.5,3.5) node {$\bullet$};
		\draw (0.5,4)--(0.5,0.5)--(2,0.5);
		\draw (0.5,3.5)--(2,3.5);
		\draw (1.5,0.5)--(1.5,4);
		\draw (1.5,2.5)--(2,2.5);
		\end{tikzpicture}
			\caption{The Le diagram associated to the pair $(u,w_\la),$ for $u = (s_2)(s_4s_3) = (s_4)(s_2s_3)$.}
		\end{subfigure}
		\caption{Constructing a Le diagram from a pair $(u, w).$} 
		\label{figure:ex:markov 44}
		\end{figure}
		\end{center}
 \end{example}

\noindent Note that from a Le diagram $\Le = (\lambda, P)$, we can read the permutation $u$ in a similar way to how we read the permutation $w$ from the Young diagram $\lambda$. Simply make the following change to any box with a dot:
\begin{center}
    \begin{tikzpicture}[scale=0.4]
\draw (0,0)--(2,0)--(2,2)--(0,2)--cycle;
\draw[<-] (2, 1)--(0,1);
\draw[<-] (1, 2) -- (1,0);
\draw[black, fill] (1,1) circle (2ex);

\draw[->, thick] (3,1)--(5,1);

\draw (6,0)--(8,0)--(8,2)--(6,2)--cycle;
\draw[black, fill] (7,1) circle (2ex);
\draw[<-] (8,1) to[out=180, in=90] (7,0);
\draw[<-] (7,2) to[out=270, in=0] (6,1);
    \end{tikzpicture}
\end{center}

\noindent In Example \ref{ex: markov 44} we have:
\begin{center}
	\begin{tikzpicture}[scale = 0.6]
		\draw (0,0)--(0,4)--(2,4)--(2,0)--(0,0);
		\draw (0,1)--(2,1);
		\draw (0,2)--(2,2);
		\draw (0,3)--(2,3);
		\draw (1,0)--(1,4);
		\draw (0.5,0.5) node {$\bullet$};
		
		\draw (0.5,3.5) node {$\bullet$};
		\draw (1.5,0.5) node {$\bullet$};
		
		\draw (1.5,2.5) node {$\bullet$};
		\draw (1.5,3.5) node {$\bullet$};

            \node at (-0.2, 3.5) {1};
            \node at (-0.2, 2.5) {2};
            \node at (-0.2, 1.5) {3};
            \node at (-0.2, 0.5) {4};
            \node at (0.5, -0.4) {5};
            \node at (1.5, -0.4) {6};

            \node at (0.5, 4.3) {$\mathbf{1}$};
            \node at (1.5, 4.3) {$\mathbf{2}$};
            \node at (2.2, 3.5) {$\mathbf{3}$};
            \node at (2.2, 2.5) {$\mathbf{4}$};
            \node at (2.2, 1.5) {$\mathbf{5}$};
            \node at (2.2, 0.5) {$\mathbf{6}$};

            \draw[->] (0, 3.5) to[out=0, in=270] (0.5, 4);
            \draw (0, 2.5) to (1, 2.5);
            \draw (1, 2.5) to[out=0, in=270] (1.5, 3);
            \draw[->] (1.5, 3) to[out=90, in=180] (2, 3.5);
            \draw[->] (0, 1.5) to (2, 1.5);
            \draw (0, 0.5) to[out=0, in=270] (0.5, 1) to (0.5, 3) to[out=90, in=180] (1, 3.5);
            \draw[->] (1, 3.5) to[out=0, in=270] (1.5, 4);
            \draw (0.5, 0) to[out=90, in=180] (1, 0.5) to[out=0, in=270] (1.5, 1) to (1.5, 2);
            \draw[->] (1.5, 2) to[out=90, in=180] (2, 2.5);
            \draw[->] (1.5, 0) to[out=80, in=180] (2, 0.5);
		\end{tikzpicture}
\end{center}
which coincides with $u = [1, 3, 5, 2, 4, 6]$. 

\subsubsection{Bounded affine permutations}\label{ssec:boundedaffinepermutation}Finally, let us discuss  $k$-bounded affine permutations of size $n$, following \cite{KLS}. By definition, an affine permutation $f:\Z\to \Z$ of size $n$ is a bijection such that $f(i+n)=f(i)+n$ for all $i\in\Z$; we often denote affine permutations in one-line window notation $f=[f(1)\ldots f(n)]$. By definition, an affine permutation is said to be {\it $k$-bounded} if the following conditions are satisfied:
$$i\le f(i)\le i+n,\quad i\in\Z\quad\mbox{ and } \quad \sum_{i = 1}^{n}(f(i) - i) = nk.$$
\noindent 
By \cite[Prop.~3.15]{KLS}, a $k$-bounded affine permutation $f$ admits a unique decomposition of the form
\begin{equation}\label{eq: dec bounded affine permutation}
f = u_ft_{k}w_f^{-1},\qquad \mbox{where }t_k:=[1 + n, 2+n, \dots, k+n, k+1, k+2, \dots, n],
\end{equation}

\noindent with $(u_f,w_f)\in S_n$ a positroid pair. This is to say, any $k$-bounded affine composition admits a unique decomposition of the form
\[
f = \mbox{\cyr{U}}^{-1}t_k\Sh
\]
where $\Sh$ is a $k$-shuffle permutation and {\cyr{U}} $\leq \Sh$.

Let us provide an explicit description of the permutations $u_f, w_f$ appearing in \eqref{eq: dec bounded affine permutation}. For that, we note that there exist exactly $k$ values $i_{1} < i_{2} < \ldots < i_{k}$ of $i_r\in[1,n]$ such that $n<f(i_r)$, and exactly $(n-k)$ values $j_{1} < j_{2} < \cdots < j_{n-k}$ of $j_t\in[1,n]$ such that $f(j_t) \leq n$. 
The permutations $u_f,w_f$ are then described as follows:
\begin{equation}
\label{eq: w from f}
w_f := [i_{1}, i_{2}, \dots, i_{k}, j_{1}, j_{2}, \dots, j_{n-k}],
\end{equation}
\begin{equation}
\label{eq: u from f}
u_f := [f(i_{1}) - n, \dots, f(i_{k}) - n, f(j_{1}), \dots, f(j_{n-k})].
\end{equation}
\noindent Note that $w_f$ is a $k$-Grassmannian permutation and $u \in S_{n}$, since $n < f(i_{r}) \leq 2n$ for every $r\in[1,k]$. The permutations $u_f,w_f$ 
coincide with the permutations  in \cite[Proposition 3.15]{KLS}. \footnote{Our notation coincides with that of \cite{KLS}. Note that what \cite{GL} calls a $k$-Grassmannian permutation is what we call a $k$-shuffle, so the decomposition in \cite[Proposition 4.2]{GL} is in fact the decomposition $f = \mbox{\cyr{U}}^{-1}t_k\Sh$.}

\begin{example}
First, for the trivial $n$-translation $f=t_{k} = [1+n, \dots, k+n, k+1, \dots, n]$, we have $(i_{1}, \dots, i_{k}) = (1, \dots, k)$ and thus $w_f = [1, 2, \dots, n]$. Similarly, $u_f = [1, 2, \dots, n]$ is also the identity. 

Second, for the $k$-bounded permutation $f$ defined by $f(i)=i+k$, we obtain that $(i_{1}, \dots, i_{k}) = (n-k+1, \dots, n)$ and hence $w_f = [n-k+1, \dots, n, 1, \dots, n-k]$ 
is the maximal $k$-Grassmannian permutation. In this second case, the permutation $u_f = [1, 2, \dots, n]$ is still the identity.\hfill$\Box$
\end{example}

\noindent Next, we record some facts translating the notations between Le diagrams and bounded affine permutations. 

\begin{lemma}
\label{lem: Le and f}
Suppose that $u\le w$ correspond both to the Le diagram of shape $\lambda$ and to the bounded affine permutation $f$, so that $f = ut_{k}w^{-1}$. Let $i_1,\ldots,i_k$ and $j_1,\ldots,j_{n-k}$ be as above. Then:

a) $i_1=1+\lambda_k, i_2 = 2 + \lambda_{k-1}, \ldots,i_k=k+\lambda_1$ and  $j_1=k+1-\lambda_1^t, j_2 = k+2-\lambda_{2}^{t}\ldots, j_{n-k}=n-\lambda_{n-k}^t$.

b) If the Le diagram for $u\leq w$ has no dots in the rightmost column then 
$$f(j_{n-k})=u(n)=w(n)=j_{n-k}.$$ Otherwise $f(j_{n-k})=u(n)=n-d+1$, with $d$ the row number of the lowest dot in the last column.
\end{lemma}

\begin{proof}
Part (a) follows by comparing the two formulas  \eqref{eq: w from lambda} and \eqref{eq: w from f} for $w$ in one-line notation. The first part of Part (b) holds by construction.  To prove the second part, observe that 
$$
w=(s_{n-\lambda_{n-k}^t}\cdots s_{n-1})\cdots (s_{k+1-\lambda_1^{t}}\cdots s_{k})
$$
and that we have
$
u=(as_{n-d+1}\cdots s_{n-1})b
$,
where $b$ does not contain $s_{n-1}$ and $a$ is a subword of $(s_{n-\lambda_{n-k}^t}\cdots s_{n-d-1})$.
Therefore $u(n)=n-d+1$.
\end{proof}

Lemma \ref{lem: Le and f} above gives us a direct way for moving from the Le diagram to the corresponding bounded affine permutation. Let us fix a Le diagram $\Le$ with the underlying partition $\lambda$. Recall the left and right labels for $\lambda$ from Section \ref{subsec: w}.
It follows from identity \eqref{eq: w from f} that $i_{1}, \dots, i_{k}$ are the right labels corresponding to the \emph{vertical} steps of $\la$, while $j_1, \dots, j_{n-k}$ are the right labels that correspond to the horizontal steps. 

Now we consider the wiring diagram of $\Le$. 
The following two facts are a consequence of \eqref{eq: u from f}.

\begin{itemize}
\item[(a)] Consider the wire starting at the bottom of the $s$-th column of $\Le$ with the left label $s+k$. Then, $f(j_{s})$ is the right label of the endpoint of this wire.
\item[(b)] Consider the wire starting on the left of the $\ell$-th row of $\Le$ with the left label $\ell$. Then, $f(i_{\ell}) = n + t$, where $t$ is the right label of the endpoint of this wire. 
\end{itemize}

\noindent Note that (a) and (b) imply that the dotless columns of $\Le$ correspond to the fixed points of $f$; while the dotless rows correspond to $x \in \{1, \dots, n\}$ satisfying $f(x) = x+n$.  See Figure \ref{fig: wiring diagram} for an example.
\begin{figure}[h!]
\centering
				\begin{tikzpicture}[scale = 0.8]
		\draw (0,0)--(0,4)--(2,4)--(2,0)--(0,0);
		\draw (0,1)--(2,1);
		\draw (0,2)--(2,2);
		\draw (0,3)--(2,3);
		\draw (1,0)--(1,4);
		\draw (0.5,0.5) node {$\bullet$};
		
		\draw (0.5,3.5) node {$\bullet$};
		\draw (1.5,0.5) node {$\bullet$};
		
		\draw (1.5,2.5) node {$\bullet$};
		\draw (1.5,3.5) node {$\bullet$};

  \draw node at (0.5, 4.3) {$1$};
  \draw node at (1.5, 4.3) {$2$};
  \draw node at (2.2, 3.5) {$3$};
  \draw node at (2.2, 2.5) {$4$};
  \draw node at (2.2, 1.5) {$5$};
  \draw node at (2.2, 0.5) {$6$};

  \draw[->, color=blue] (0, 3.5) to[out=0, in=270] (0.5, 4);
  \draw[->, color=blue] (0, 2.5) to (1, 2.5) to[out=0,in=270] (1.5, 3) to [out=90, in=180] (2, 3.5);
  \draw[->, color=blue] (0, 1.5) to (2, 1.5);
  \draw[->, color=blue] (0,0.5) to[out=0,in=270] (0.5, 1) to (0.5, 3) to[out=90, in=180] (1, 3.5) to[out=0, in=270] (1.5, 4);

  \draw[->, color=red] (0.5, 0) to[out=90, in=180] (1, 0.5) to[out=0, in=270] (1.5, 1) to (1.5, 2) to[out=90, in=180] (2, 2.5);
  \draw[->, color=red] (1.5, 0) to[out=90, in=180] (2, 0.5);	
		\end{tikzpicture}

  \caption{The wiring diagram for the Le diagram of Example \ref{ex: markov 44}. If $f$ is the corresponding bounded affine permutation, then $f(1), f(2) < n$, while $f(3), f(4), f(5), f(6) > n$. 
  Moreover, to find $f(1)$ we follow the red strand starting opposite to $1$, and see that $f(1) = 4$. Similarly, $f(2) = 6$. Likewise, to find $f(3)$ we follow the blue strand starting opposite to $3$, and see that $f(3) = 1 + 6$. Similarly, $f(4) = 3+6$, $f(5) = 5+6$, $f(6) = 2+6$.
  Note that the only dotless row has the right label $5$: this corresponds to $5$ being the only $x \in \{1, \dots, 6\}$ satisfying $f(x) = x+6$.}
    \label{fig: wiring diagram}
\end{figure}

For $1 \le s \le n-k$ we define
\begin{equation}
\inv(s):=\sharp\{t>s: f(j_t)<f(j_s)\}.
\end{equation}

\begin{lemma}
\label{lem: inv}
For all $s\in [1,n-k]$, we have the inequalities
$$
j_s\le f(j_s)-\inv(s)\le k+s
$$
The first inequality is sharp unless $f(j_s)=j_s$.
\end{lemma}

\begin{proof}
If $t>s$ then $j_t>j_s$. If $f(j_t)<f(j_s)$ then
$$
j_s<j_t\le f(j_t)<f(j_s),
$$
and since $f$ is injective all such $f(j_t)$ are distinct. Therefore either $f(j_s)=j_s$ and $\inv(s)=0$, or $\inv(s)\le f(j_s)-j_s-1$, so $f(j_s)-\inv(s)\ge j_s+1$.

For the second inequality, observe that 
\begin{equation}\label{eq: noninversions}
\sharp\{t>s: f(j_t)>f(j_s)\}=n-k-s-\inv(s),
\end{equation}
so we have at least $n-k-s-\inv(s)$ distinct integers between $f(j_s)+1$ and $n$. Therefore
$f(j_s)+n-k-s-\inv(s)\le n$ and $f(j_s)-\inv(s)\le k+s$. 
\end{proof}

\subsection{Richardson Braids}\label{ssec:RichardsonBraid}

Given a permutation $v\in S_n$, we will denote by $\beta(v)\in\Br_n^+$ its reduced positive braid lift to the $n$-stranded braid group. A positive braid word for the braid $\beta(v)\in\Br_n^+$, which we also denote by $\beta(v)\in\cB_n^+$, 
 is obtained by considering a reduced expression $v=s_{i_1}\ldots s_{i_{\ell(v)}}$ for $v\in S_n$ in terms of a product of the simple transpositions $s_1,\ldots,s_{n-1}\in S_n$ generating the symmetric group $S_n$ and substituting each $s_i$ by the Artin braid generator $\sigma_i$, $i\in[1,n]$, i.e. $\beta(v)=\s_{i_1}\ldots \s_{i_{\ell(v)}}$. The braid word $\beta(v)$ depends on the choice of a reduced expression of $v$, but all such words for a given $v$ are related by braid relations $\sigma_i\sigma_{i+1}\sigma_i=\sigma_{i+1}\sigma_i\sigma_{i+1}$, a.k.a.~Reidemester III moves, and commutation relations, and thus represent the same braid. Recall the notation $v^* := w_0 v w_0$.


\begin{definition}[Richardson braid]\label{def:Richardson}
Let $u,w\in S_n$ be two permutations such that $u\leq w$ in the Bruhat order. The Richardson braid word $R_n(u,w)\in\cB_n$ associated to the pair $(u,w)$ is
$$R_n(u,w):=\beta(u^*)^{-1}\beta(w^*).$$
By definition, the smooth Richardson link is the smooth link $\La(u,w)\sse\R^3$ given by the $0$-framed closure of the $n$-stranded braid word $R_{n}(u, w)$.\hfill$\Box$
\end{definition}

\noindent Definition \ref{def:Richardson} is inspired by \cite[Section 1.5]{GL}. The reason we have to conjugate by $w_0$ is Corollary \ref{cor:richardson}. See also Remark \ref{rem:different isomorphisms}. Let us also define a version of the Richardson braid with only positive Artin generators, together with its associated smooth link:

\begin{definition}[Positive Richardson braid]\label{def:pos Richardson}
Let $u, w \in S_n$ be two permutations such that $u \leq w$ in the Bruhat order. The positive Richardson braid word $R^{+}_{n}(u, w)\in\cB_n$ associated to the pair $(u, w)$ is
\[
R^{+}_{n}(u, w) := \beta(u^{-1}w_0)\beta(w^*).
\]
By definition, the positive Richardson link is the smooth link $\Lambda^{+}(u, w) \subseteq \R^{3}$ associated to the $0$-framed closure of the $n$-stranded braid word $\Delta^{-1}_{n}R^{+}_{n}(u, w)$.\hfill$\Box$
\end{definition}

\begin{remark}
We emphasize that the braid words in Definitions \ref{def:Richardson} and \ref{def:pos Richardson} depend on the choice of reduced expressions of $u$ and $w$, but the braids $R_n(u, w), R^{+}_{n}(u, w)$, and $\Delta^{-1}_{n}R^{+}_{n}(u, w)$ depend only on the pair $(u, w)$. The smooth links $\La(u,w)$ and $\La^{+}(u, w)$ thus also depend only on the pair $(u, w)$.

By construction, the smooth links $\La(u,w)$ and $\La^{+}(u, w)$ are smoothly isotopic. Indeed, note that we can find a reduced expression 
$\Delta_n = \beta(u^{-1}w_0)\beta(u^*)$, so that
$$
\Delta_n^{-1}=\beta(u^*)^{-1}\beta(u^{-1}w_0)^{-1}
$$
and it follows that, in the braid group $\Br_n$, the braids $\Delta_n^{-1}R^{+}_n(u,w)$ and $R_n(u,w)$ are equivalent. 
\hfill$\Box$
\end{remark}

\subsection{Juggling Braid}\label{ssec:JugglingBraid}
Let $f$ be a $k$-bounded affine permutation. In this section, we construct a braid $J_k(f)$ on $k$-strands, called the \emph{juggling braid} of $f$, and provide an explicit braid word for it.

Given a $k$-bounded affine permutation $f: \Z \lr \Z$, we picture it as follows. We consider the plane $\R^2$ with Cartesian coordinates $(x,y)$. For each number $i \in \Z$, we join $i$ to $f(i)$ using the upper-circumference arc $A_{f(i)}(f)$\footnote{For convenience, cf.~Definition \ref{def:juggling}, we choose to label the upper-circumference arc $A_{f(i)}(f)$ by its rightmost point.}:
\[A_{f(i)}(f)=\left\{(x,y)\in\R^2: \left(x - \frac{i+f(i)}{2}\right)^{2}+y^2=\frac{(f(i)-i)^2}{4}\right\}\cap\{y\geq0\}\sse\R^2.\]
If $f(i) = i$, then $A_{i}(f) = A_{f(i)}(f)$ is an arc of radius zero that we represent by a dot at $(i, 0)$. The union of the arcs $A_{i}$, $i \in \Z$ is referred to as the \emph{affine juggling diagram} of $f$, cf. \cite{KLS}.

\begin{example} \label{ex:affine juggling}
Let us choose $k = 3, n = 7$ and $f = [3,4,9,6,7,12,8]$. Then the affine juggling diagram has the following form. 
\begin{center}
		\begin{tikzpicture}[scale=0.7]
            \draw (-4.1, -0.2) node {\scriptsize{-4}};
            \draw (-3.1,-0.2) node {\scriptsize{-3}};
		\draw (-2.1,-0.2) node {\scriptsize{-2}};
		\draw (-1.1,-0.2) node {\scriptsize{-1}};
		\draw (0,-0.2) node {\scriptsize{0}};
		\draw (1,-0.2) node {\scriptsize{1}};
		\draw (2,-0.2) node {\scriptsize{2}};
		\draw (3,-0.2) node {\scriptsize{3}};
		\draw (4,-0.2) node {\scriptsize{4}};
		\draw (5,-0.2) node {\scriptsize{5}};
		\draw (6,-0.2) node {\scriptsize{6}};
		\draw (7,-0.2) node {\scriptsize{7}};
		\draw (8,-0.2) node {\scriptsize{8}};
		\draw (9,-0.2) node {\scriptsize{9}};
		\draw (10,-0.2) node {\scriptsize{10}};
		\draw (11,-0.2) node {\scriptsize{11}};
		\draw (12,-0.2) node {\scriptsize{12}};
		\draw (13,-0.2) node {\scriptsize{13}};
		\draw (14,-0.2) node {\scriptsize{14}};
		\draw [->] (-4.5,0)--(14.5,0);
		\draw (3,0) arc (0:180:1);
		\draw (4,0) arc (0:180:1);
		\draw (8,0) arc (0:180:0.5);
		\draw (6,0) arc (0:180:1);
		\draw (7,0) arc (0:180:1);
		\draw (9,0) arc (0:180:3);
		\draw (12,0) arc (0:180:3);
            \draw (1,0) arc (0:180:0.5);
            \draw (5,0) arc (0:180:3);
            \draw (0,0) arc (0:180:1);
            \draw (-1,0) arc (0:180:1);
            \draw (10,0) arc (0:180:1);
            \draw (11,0) arc (0:180:1);
            \draw (2,0) arc (0:180:3);

            \node at (15, 0) {$\dots$};
		\end{tikzpicture}
	\end{center}
\end{example}

Since the function $f$ is $n$-periodic, so is its affine juggling diagram. Thus, to recover the affine juggling diagram it is enough to consider the arcs $A_{j}, j = 1, \dots, n$, whose \emph{rightmost} points are at $(j,0)$ with $j = 1, \dots, n$. The union of all such arcs will be referred to as simply the \emph{juggling diagram} of $f$. Moreover, we orient the arcs in a \emph{counterclockwise direction}. 

\begin{example}\label{ex: juggling}
The juggling diagram of the function $f$ from Example \ref{ex:affine juggling} is

\begin{center}
		\begin{tikzpicture}[scale=0.7]

            \draw (-4.1, -0.2) node {\scriptsize{-4}};
            \draw (-3.1,-0.2) node {\scriptsize{-3}};
		\draw (-2.1,-0.2) node {\scriptsize{-2}};
		\draw (-1.1,-0.2) node {\scriptsize{-1}};
		\draw (0,-0.2) node {\scriptsize{0}};
		\draw (1,-0.2) node {\scriptsize{1}};
		\draw (2,-0.2) node {\scriptsize{2}};
		\draw (3,-0.2) node {\scriptsize{3}};
		\draw (4,-0.2) node {\scriptsize{4}};
		\draw (5,-0.2) node {\scriptsize{5}};
		\draw (6,-0.2) node {\scriptsize{6}};
		\draw (7,-0.2) node {\scriptsize{7}};
		\draw (8,-0.2) node {\scriptsize{8}};
		\draw [->] (-4.5,0)--(8.5,0);
		\draw[->] (3,0) arc (0:180:1);
		\draw[->] (4,0) arc (0:180:1);
		\draw[->] (6,0) arc (0:180:1);
		\draw[->] (7,0) arc (0:180:1);
            \draw[->] (1,0) arc (0:180:0.5);
            \draw[->] (5,0) arc (0:180:3);
            \draw[->] (2,0) arc (0:180:3);

		\end{tikzpicture}
	\end{center}
\end{example}



 By virtue of $f$ being a $k$-bounded affine permutation, there exist exactly $k$ values $1 \leq i_1 < \ldots < i_k\leq n$  such that $n< f(i_s)\le 2n$. Equivalently, $0 < f(i_s) - n \leq n$, and $f^{-1}(f(i_s) - n) = i_s - n \leq 0$. Thus, we obtain $k$ arcs $A_{f(i_1) - n}, \dots, A_{f(i_k) - n}$ in the juggling diagram of $f$ whose leftmost point is at $(i, 0)$ with $i \leq 0$. 
 We call these upper-circumferences \emph{special}. The juggling braid is defined via a tangle diagram obtained from the juggling diagram, as follows:

\begin{definition}[Juggling Braid]\label{def:juggling}  Let $f: \Z\lr\Z$ be a $k$-bounded affine permutation of size $n$, and consider the juggling diagram for $f$ as defined above, which is the union of all the arcs of non-zero radius among $A_{1}, \dots, A_{n}$ oriented in the counterclockwise direction. By definition, the juggling braid $J_k(f)$ is the braid defined by this tangle, declaring all the crossings between these arcs to be positive and smoothing the intersections with the $x$-axis according to the local models depicted in Figure \ref{fig:JugglingBraid}.
\end{definition}




\begin{center}
	\begin{figure}[h!]
		\centering
		\includegraphics[scale=0.6]{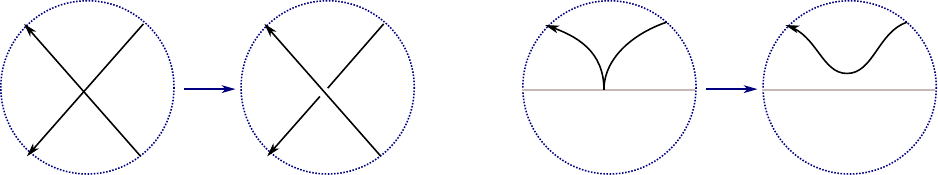}
		\caption{Local models constructing the juggling braid from the juggling diagram.}
		\label{fig:JugglingBraid}
	\end{figure}
\end{center}

\begin{remark}
Let us comment on the number of strands of $J_k(f)$, as well as their labeling. Note that the elements $i_1, \dots, i_k$ satisfying $f(i_k) > n$ are precisely those points which are \emph{not} the leftmost point of an arc in the juggling diagram of $f$. We thus have $k$ strands and with initial points $i_1, \dots, i_k$.
We label the strands so that the $j$-th strand is precisely the strand whose initial point is $i_{k-j+1}$. 
\end{remark}

To see that $J_k(f)$ has precisely $k$ strands we can alternatively think of a juggler with $k$ balls and think of the arcs in $J_k(f)$ as being the trajectories for these balls while being juggled. Here $x$ denotes the time coordinate and $y$ denotes the height of a ball. This implies the following simple fact:

\begin{prop}
\label{prop: vertical line test}
Each vertical line with non-integer coordinate intersects $J_k(f)$ in at most $k$ arcs, which correspond to distinct strands of the juggling braid.\hfill$\Box$
\end{prop}

\begin{example}\label{ex: running 1}
The following will be our running example. Consider the $4$-bounded affine permutation of rank $6$, $f = [4,6,7,9,11,8]$. The juggling diagram of $f$ is as follows: 

\begin{center}
		\begin{tikzpicture}[scale=0.7]

             \draw (-3.1, -0.2) node {\scriptsize{-3}};
             \draw (-2.1, -0.2) node {\scriptsize{-2}};
             \draw (-1.1, -0.2) node {\scriptsize{-1}};
            \draw (0, -0.2) node {\scriptsize{0}};
		\draw (1,-0.2) node {\scriptsize{1}};
		\draw (2,-0.2) node {\scriptsize{2}};
		\draw (3,-0.2) node {\scriptsize{3}};
		\draw (4,-0.2) node {\scriptsize{4}};
		\draw (5,-0.2) node {\scriptsize{5}};
		\draw (6,-0.2) node {\scriptsize{6}};
		\draw (7,-0.2) node {\scriptsize{7}};
		\draw (8,-0.2) node {\scriptsize{8}};
		\draw [->] (-4,0)--(8.5,0);
            \draw[color=red] (4,0) arc (0:180:1.5);
            \draw[color=red] (1,0) arc (0:180:2);

            \draw[color=blue] (6,0) arc (0:180:2);
            \draw[color=blue] (2,0) arc (0:180:1);

            \draw[color=cyan] (3,0) arc (0:180:2.5);

            \draw[color=magenta] (5,0) arc (0:180:3);

		\end{tikzpicture}
	\end{center}

 And thus the juggling braid is as follows:

 \begin{center}
    	\begin{tikzpicture}[scale = 0.6]


  \draw[color=blue] (6, 3.5) to[out=0, in=180] (8, 0.5);
  \draw[color=magenta] (6, 2.5) to[out=0, in=180] (8, 3.5);
  \draw[color=red] (6, 1.5) to[out=0, in=180] (8, 2.5);
  \draw[color=cyan] (6, 0.5) to[out=0, in=180] (8, 1.5);

    \draw[color=red] (8,2.5) to[out=0, in=180] (10,0.5);
    \draw[color=magenta] (8,3.5) to[out=0, in=180] (10, 3.5);
    \draw[color=cyan] (8, 1.5) to[out=0, in=180] (10,2.5);
    \draw[color=blue] (8, 0.5) to [out=0, in=180] (10,1.5);

    \draw[color=red] (10,0.5) to[out=0, in=180] (12, 3.5);
    \draw[color=blue] (10,1.5) to[out=0, in=180] (12, 0.5);
    \draw[color=cyan] (10, 2.5) to [out=0, in=180] (12, 2.5);
    \draw[color=magenta] (10, 3.5) to [out=-20, in=180] (12, 1.5);

  \end{tikzpicture}
  \end{center}
\end{example}

\begin{remark}
One can use parabolas, which correspond to actual juggling trajectories, or other curves instead of circles in the definition of $J_k(f)$. As long as these curves are convex and each pair of curves intersects at most once, the resulting braids are all related by Reidemeister III moves.
\hfill$\Box$
\end{remark}

Thanks to the previous remark, we may assume that all the crossings between special arcs come \emph{after} all other crossings. This means that we have a decomposition
\[
J_k(f) = \Jb_k(f)\Ja_k(f)
\]
where $\Jb_k(f)$ records the crossings between special arcs, and $\Ja_k(f)$ records all other crossings. Note that by definition, $\Jb_k(f)$ is a reduced $k$-stranded braid. In our running Example \ref{ex: running 1} we have:

 \begin{center}
    	\begin{tikzpicture}[scale = 0.6]


  \draw[color=blue] (6, 3.5) to[out=0, in=180] (8, 0.5);
  \draw[color=magenta] (6, 2.5) to[out=0, in=180] (8, 3.5);
  \draw[color=red] (6, 1.5) to[out=0, in=180] (8, 2.5);
  \draw[color=cyan] (6, 0.5) to[out=0, in=180] (8, 1.5);

    \draw[color=red] (8,2.5) to[out=0, in=180] (10,0.5);
    \draw[color=magenta] (8,3.5) to[out=0, in=180] (10, 3.5);
    \draw[color=cyan] (8, 1.5) to[out=0, in=180] (10,2.5);
    \draw[color=blue] (8, 0.5) to [out=0, in=180] (10,1.5);

    \draw[color=red] (10,0.5) to[out=0, in=180] (12, 3.5);
    \draw[color=blue] (10,1.5) to[out=0, in=180] (12, 0.5);
    \draw[color=cyan] (10, 2.5) to [out=0, in=180] (12, 2.5);
    \draw[color=magenta] (10, 3.5) to [out=-20, in=180] (12, 1.5);

    \draw[dashed] (9.5,-1) to (9.5,4);

    \node at (7.5, -1) {$J^{(1)}(f)$};
    \node at (11, -1) {$J^{(2)}(f)$};

  \end{tikzpicture}
  \end{center}

Finally, we define the link associated to the juggling diagram.

\begin{definition}
Let $f$ be a $k$-bounded affine permutation. The juggling link $\La(f) \sse \R^3$ is the smooth $0$-framed closure of $\Delta_{k}^{-1}J_k(f)$.
\end{definition}


\subsubsection{The braid $J_{k}^{(1)}(f)$} In this section, we express the braid $\Ja_{k}(f)$ as a product of $(n-k)$ interval braids, some of which may be empty, that naturally correspond to the columns of the Le diagram $\Le$ corresponding to $f$. 
As defined above, a crossing in $J_{k}(f)$ belongs to the braid $\Ja_{k}(f)$ if and only if at least one of the arcs involved in the crossing is not special. Here is the construction.

\begin{definition}
Suppose that $s<t<u<f(s)<f(t)<f(u)$, so that the arcs $A_{f(s)},A_{f(t)}$ and $A_{f(u)}$ pairwise intersect. We say that $A_{f(s)},A_{f(t)}$ and $A_{f(u)}$ are in {\em good position} if $A_{f(s)}\cap A_{f(u)}$ is inside $A_{f(t)}$. Equivalently, $A_{f(s)}\cap A_{f(t)}$ is outside $A_{f(u)}$ and $A_{f(t)}\cap A_{f(u)}$ is outside $A_{f(s)}$ (see Figure \ref{fig: inversion triple}). Otherwise we say that these three arcs form an \emph{inversion triple}. We say that a juggling diagram is in good position if all triples of arcs as above are in good position.\hfill$\Box$
\end{definition}

\begin{figure}[ht!]
\begin{tikzpicture}
\draw (4,0) arc (0:180:1.5); 
\draw (6,0) arc (0:180:2); 
\draw (7,0) arc (0:180:2); 

\draw (13,0) arc (0:180:1.5); 
\draw (13.5,0) arc (0:180:1.5); 
\draw (14.7,0) arc (0:180:2); 

\draw (1,-0.2) node {$s$};
\draw (2,-0.2) node {$t$};
\draw (3,-0.2) node {$u$};
\draw (4,-0.2) node {$f(s)$};
\draw (6,-0.2) node {$f(t)$};
\draw (7,-0.2) node {$f(u)$};

\draw (10,-0.2) node {$s$};
\draw (10.5,-0.2) node {$t$};
\draw (10.8,-0.2) node {$u$};
\draw (12.9,-0.2) node {$f(s)$};
\draw (13.6,-0.2) node {$f(t)$};
\draw (14.7,-0.2) node {$f(u)$};
\end{tikzpicture}
\caption{Good position (left) and inversion triple (right).}
\label{fig: inversion triple}
\end{figure}

\begin{lemma}
Up to braid moves, we can draw a juggling diagram for $J_k^{(1)}$ in good position.
\end{lemma}

\begin{proof}
First, consider a reduced braid with strands labeled on the left. Following \cite{Elias} (see also \cite{ManinS}), we denote by $(s|t)$ a crossing between the strands labeled by $s$ and $t$. Three strands labeled by $s, t$, and $u$ form an inversion triple if $s<t<u$ and the crossings between the respective strands appear from left to right in the order $(t|u),(s|u),(s|t)$ (called the {\em antilexicograhic order} in \cite{Elias}). 
Now \cite[Proposition 3.15, Corollary 3.16]{Elias} state that any reduced braid can be transformed via braid moves to a unique braid word without inversion triples. Furthermore, this can be achieved by braid moves which always reduce the number of inversion triples.

Our braid $J_k^{(1)}$ is not reduced, but any two arcs intersect at most once.
Therefore any collection of arcs where each arc belongs to a different strand of $J_k^{(1)}$ is reduced, and we can repeatedly apply the above result and Proposition \ref{prop: vertical line test} as we scan the braid from right to left. This would ensure that the number of inversion triples to the right of a given vertical line  can be decreased to zero, and eventually we will eliminate all inversion triples.
\end{proof}

To describe the braid $J_k^{(1)}$, we use the Le diagram $\Le$ corresponding to $f$ and the notations $i_1,\ldots,i_k, j_1,\ldots,j_{n-k}$ as above. In particular, $f(i_s)>n$ and $A_{f(i_s)-n}$ are special arcs, while $f(j_t)\le n$ and $A_{f(j_t)}$ are non-special arcs. Moreover, $A_{i_1}, \dots, A_{i_k}$ are the initial arcs of the $k$ strands of $J_{k}^{(1)}$. 

\begin{lemma}
\label{lem: right end}
For $t>s \geq 0$, we have $f(j_t)\in \{i_1, \dots, i_k, j_{n-k}, \dots, j_{s+1}\}.$
\end{lemma}

\begin{proof}
Note that $w=[i_1,\ldots,i_k,j_{1},\ldots,j_{n-k}]$ is a permutation of $[1,\ldots,n]$, hence 
$$
\{1,\ldots,n\}=\{j_1,\ldots,j_s\}\sqcup \{i_1, \dots, i_k, j_{n-k}, \dots, j_{s+1}\}.
$$ 
For $t>s$, we have $f(j_t)\ge j_t>j_s$, so 
$f(j_t)\notin \{j_1,\ldots,j_s\}$.
\end{proof}

We define the sets
\begin{equation}
\label{eq: def tuple}
T_{n-k-s} := \{i_1, \dots, i_k, j_{n-k}, \dots, j_{s+1}\} \setminus \{f(j_{n-k}), \dots, f(j_{s+1})\},\ s=0,\ldots,n-k.
\end{equation}
By Lemma \ref{lem: right end} these have exactly $k$ elements for each $s$.
Note that if $j_s$ is a fixed point of $f$, then $j_s \not\in T_{i}$ for all $i$. Also note that
\[
T_{0}=\{i_1,\ldots,i_k\},
\]
\[T_{n-k} = \{i_1, \dots, i_k, j_{n-k}, \dots, j_1\}\setminus \{f(j_{n-k}), \dots, f(j_1)\} = \{f(i_1) - n, \dots, f(i_k) - n\},
\] 
and
\[
T_{n-k-s+1}=\left(T_{n-k-s}\cup\{j_s\}\right)\setminus \{f(j_s)\}
\]
so that, if $j_s$ is a fixed point of $f$, we get $T_{n-k-s+1} = T_{n-k-s}$. 


\begin{lemma}
\label{lem: j1 from intervals} In the notation above, the following hold:

\begin{itemize}
    \item[(a)] Assume that $J_k^{(1)}$ is in good position. Then it can be written as a product of braids 
$$
J_k^{(1)}=I_1\cdots I_{n-k},
$$
where $I_s$ is given by the crossings of the arc $A_{f(j_s)}$ with arcs $A_{f(j_t)}$, $t<s$, 
as well as all the crossings of $A_{f(j_s)}$ with the special arcs.

\item[(b)] If $f(j_s)=j_s$ then the braid $I_s$ is trivial.

\item[(c)] If $f(j_s)>j_s$, label the strands by the set $T_{n-k-s}$ on the left, and by $T_{n-k-s+1}$ on the right, in decreasing order from top to bottom.
Then the braid $I_s$ is the interval braid which connects  $f(j_s)$ on the left to $j_s$ on the right, and all other elements of $T_{n-k-s+1}$ to themselves.
\end{itemize}
\end{lemma}




\begin{proof}
For Part (a), the definition of good position implies that all crossings between $A_{j_t}$ and $A_{j_{t'}}$ with $t,t'>s$ (resp. $t,t'<s$) are located to the right (resp. left) of $A_{j_s}$. Furthermore, if $t<s<t'$ then the crossing $A_{j_t}\cap A_{j_s}$ is to the left of the crossing $A_{j_s}\cap A_{j_{t'}}$, so we can indeed sort the crossings as desired. Part (b) is immediate since $f(j_s)=j_s$ is a dot in the juggling diagram which we ignore in the juggling braid.

For Part (c) first note that, since $i_t - n < 0 < j_s \leq f(j_s)$, the arcs $A_{f(j_s)}$ and $A_{f(i_t) - n}$ cross if and only if $j_s < f(i_t) - n < f(j_s)$. It follows that the braid $I_s$ records the crossings of $A_{f(j_s)}$ only with arcs whose right endpoint is to the left of $f(j_s)$, and thus it must be an interval braid. 
Let $\widetilde{I_s}$ denote the interval braid connecting 
$T_{n-k-s+1}$ and $T_{n-k-s}$ as above, we need to verify that $I_s = \widetilde{I}_s$. Note that the crossings in $\widetilde{I_s}$ are labeled by $x\in  T_{n-k-s+1}$ such that $j_s<x<f(j_s)$. If $x = f(i_t) - n$ for some $t$, then we have $i_t - n \leq 0 < j_s < f(i_t - n) < f(j_s)$ so the arcs $A_{f(j_s)}$ and $A_{f(i_t) -n}$ intersect once. Now assume that $x$ is \emph{not} of the form $f(i_t) - n$, so it must be of the form $f(j_t)$ for some $t$. Note that necessarily $t < s$, for otherwise $f(j_t) \not\in T_{n-k-s}$. So we have $j_t < j_s < f(j_t) < f(j_s)$ and the arcs $A_{f(j_s)}$ and $A_{f(j_t)}$ intersect once. It follows that both $I_s$ and $\widetilde{I}_s$ involve crossings between the same strands, and since both are interval braids the result follows.
\end{proof}

\begin{example}\label{ex: j1 intervals}
Let us take the bounded affine permutation of Example \ref{ex: running 1}. We have $f = [4,6,7,9,11,8]$ and, as we have seen, the juggling diagram of $f$ is:

\begin{center}
		\begin{tikzpicture}[scale=0.7]

             \draw (-3.1, -0.2) node {\scriptsize{-3}};
             \draw (-2.1, -0.2) node {\scriptsize{-2}};
             \draw (-1.1, -0.2) node {\scriptsize{-1}};
            \draw (0, -0.2) node {\scriptsize{0}};
		\draw (1,-0.2) node {\scriptsize{1}};
		\draw (2,-0.2) node {\scriptsize{2}};
		\draw (3,-0.2) node {\scriptsize{3}};
		\draw (4,-0.2) node {\scriptsize{4}};
		\draw (5,-0.2) node {\scriptsize{5}};
		\draw (6,-0.2) node {\scriptsize{6}};
		\draw (7,-0.2) node {\scriptsize{7}};
		\draw (8,-0.2) node {\scriptsize{8}};
		\draw [->] (-4,0)--(8.5,0);
            \draw[color=red] (4,0) arc (0:180:1.5);
            \draw[color=red] (1,0) arc (0:180:2);

            \draw[color=blue] (6,0) arc (0:180:2);
            \draw[color=blue] (2,0) arc (0:180:1);

            \draw[color=cyan] (3,0) arc (0:180:2.5);

            \draw[color=magenta] (5,0) arc (0:180:3);

            \node at (-1, 2.0) {$\circ$};
            \node at (0.35, 2.5) {$\circ$};
            \node at (-0.3, 1.85) {$\circ$};
            \node at (0.75, 0.95) {$\circ$};

		\end{tikzpicture}
	\end{center}

%
%
%
%
	%

The crossings marked with $\circ$ represent crossings between two special arcs, and thus they will not be represented in the braid $\Ja_k(f)$. Note that $i_1 = 3, i_2 = 4, i_3, = 5, i_4 = 6$, while $j_1 = 1$ and $j_2 = 2$. Thus, we start with the set $T_0 = \{3, 4, 5, 6\}$. Note that $T_1 = (T_0 \cup \{2\})\setminus\{f(2)\} = \{2, 3, 4, 5\}$, and we find the interval $I_2$ as follows: 

\begin{center}
    	\begin{tikzpicture}[scale = 0.6]
  \draw node at (6, 0.5) {$3$};
  \draw node at (6,1.5) {$4$};
  \draw node at (6,2.5) {$5$};
  \draw node at (6,3.5)  {$6$};
  \draw node at (6, -0.2) {$T_0$};

  \draw node at (8,0.5) {$2$};
  \draw node at (8,1.5)  {$3$};
  \draw node at (8,2.5)  {$4$};
  \draw node at (8,3.5)  {$5$};
  \draw node at (8, -0.2) {$T_1$};

  \draw[color=blue] (6, 3.5) to[out=0, in=180] (8, 0.5);
  \draw[color=magenta] (6, 2.5) to[out=0, in=180] (8, 3.5);
  \draw[color=red] (6, 1.5) to[out=0, in=180] (8, 2.5);
  \draw[color=cyan] (6, 0.5) to[out=0, in=180] (8, 1.5);

  \draw node at (7, -0.2) {\Large{$I_{2}$}};
  \end{tikzpicture}
\end{center}

\noindent Now, $T_2 = (T_1 \cup \{1\})\setminus\{f(1)\} = (T_0 \cup \{2, 1\})\setminus\{f(2), f(1)\} = \{1, 2, 3, 5\}$, and we find the interval $I_1$ that is concatenated to $I_2$ on the right:

\begin{center}
    	\begin{tikzpicture}[scale = 0.6]
  \draw node at (6, 0.5) {$3$};
  \draw node at (6,1.5) {$4$};
  \draw node at (6,2.5) {$5$};
  \draw node at (6,3.5)  {$6$};
  \draw node at (6, -0.2) {$T_0$};

  \draw node at (8,0.5) {$2$};
  \draw node at (8,1.5)  {$3$};
  \draw node at (8,2.5)  {$4$};
  \draw node at (8,3.5)  {$5$};
  \draw node at (8, -0.2) {$T_1$};

  \draw node at (10,0.5) {$1$};
  \draw node at (10,1.5)  {$2$};
  \draw node at (10,2.5)  {$3$};
  \draw node at (10,3.5)  {$5$};
  \draw node at (10, -0.2) {$T_2$};

  \draw[color=blue] (6, 3.5) to[out=0, in=180] (8, 0.5);
  \draw[color=magenta] (6, 2.5) to[out=0, in=180] (8, 3.5);
  \draw[color=red] (6, 1.5) to[out=0, in=180] (8, 2.5);
  \draw[color=cyan] (6, 0.5) to[out=0, in=180] (8, 1.5);

  \draw[color=magenta] (8,3.5) to[out=0, in=180] (10, 3.5);
  \draw[color=red] (8, 2.5) to[out=0, in=180] (10, 0.5);
  \draw[color=cyan] (8, 1.5) to[out=0, in=180] (10, 2.5);
  \draw[color=blue] (8, 0.5) to[out=0, in=180] (10, 1.5);

  \draw node at (7, -0.2) {\Large{$I_{2}$}};
  \draw node at (9, -0.2) {\Large{$I_1$}};
  \end{tikzpicture}
\end{center}

\noindent Comparing to Example \ref{ex: running 1}, this braid coincides with $\Ja_k(f)$.\hfill$\Box$
\end{example}

In order to find an explicit description of the interval $I_s$, we need to find the relative positions of $f(j_s)$ in $T_{n-k-s}$ and of $j_s$ in 
$T_{n-k-s+1}$. This is achieved in the following lemma. 

\begin{lemma}
\label{lem:tuple}
Assume $f(j_s)>j_s$. Then:

(a) The set 
$
\{x < f(j_s) \mid x \not \in T_{n-k-s}\}
$
has precisely $\inv(s) + s$ elements.

(b) The set $
\{x < f(j_s) \mid x \in T_{n-k-s}\}
$
has precisely $f(j_s)-\inv(s)-s-1$ elements.

(c) The set 
$
\{x < j_s \mid x \in T_{n-k-s+1}\}
$
has precisely $j_s-s=k-\lambda_s^t$ elements.
\end{lemma}

\begin{proof}
For Part (a), Lemma \ref{lem: right end} implies
$$
\{1,\ldots,n\}\setminus T_{n-k-s}=\{j_1,\ldots,j_s\}\sqcup 
\{f(j_{s+1}),\ldots,f(j_{n-k})\}.
$$
Assume that $x\notin T_{n-k-s}$ and $x<f(j_s)$.
We have the following two cases:

1) $x\in \{j_1,\ldots,j_s\}$, then $x\le j_s< f(j_s)$ and $x\notin T_{n-k-s}$. This accounts for $s$ elements. 

2) $x=f(j_t)$ for $t>s$ and $f(j_t)<f(j_s)$. This accounts for $\inv(s)$ elements. 

\noindent Part (b) follows from Part (a). For Part (c), among $j_s-1$ elements $\{1,\ldots,j_s-1\}$ there are $s-1$ elements $j_1,\ldots,j_{s-1}$ and $(j_s-s)$ elements $i_t<j_s$. None of them belongs to $\{f(j_s),\ldots,f(j_{n-k})\}.$
The former are all not in $T_{n-k-s+1}$ but the latter all are.
The last equation follows from Lemma \ref{lem: Le and f}.
\end{proof}

\begin{cor}\label{cor: intervals juggling}
The interval braid $I_{s}$ is given by
\[
I_{s} = s_{\lambda_s^{t}-1}\cdots s_{k-f(j_s) +\inv(s) + s +1}.
\]
and 
$$
J^{(1)}_k= I_{1}\cdots I_{n-k}.
$$
\end{cor}

\begin{proof}
This follows from Lemma \ref{lem: j1 from intervals} and Lemma \ref{lem:tuple}.
\end{proof}

\subsubsection{The braid $J^{(2)}(f)$} The braid $J^{(2)}(f)$ takes care of all the crossings between special arcs and it is reduced. Recall that these arcs are $A_{f(i_{1})-n}, \dots, A_{f(i_{k})-n}$, where the special arc $A_{f(i_{\ell})-n}$ joins $i_{\ell}-n$ to $f(i_{\ell} - n) = f(i_{\ell}) - n$. If $\ell < s$, then we have $i_{\ell} - n < i_{s} - n \leq 0 < f(i_{\ell} - n)$, so
the special arcs $A_{f(i_{\ell}) - n}$ and $A_{f(i_{s}) - n}$ will cross precisely when $f(i_{\ell} - n) < f(i_{s} - n)$, that is, when $f(i_{\ell}) < f(i_s)$. This implies the following result: 

\begin{prop}\label{prop: J2}
Let $f$ be a $k$-bounded affine permutation. Then the braid $J^{(2)}(f)$ is a positive reduced lift of the inverse of the permutation that sorts $(f(i_1), \dots, f(i_k))$ in decreasing order.\hfill$\Box$ 
\end{prop}

In Example \ref{ex: running 1} we have that $i_1 = 3, i_2 = 4, i_3 = 5, i_4 = 6$, while $f(i_1) = 7, f(i_2) = 9$, $f(i_3) = 11$, $f(i_4) = 8$ so that $J^{(2)}(f)$ is

\begin{center}
    	\begin{tikzpicture}[scale = 0.6]
  \draw node at (5.8, 0.5) {$7$};
  \draw node at (5.8,1.5) {$8$};
  \draw node at (5.8,2.5) {$9$};
  \draw node at (5.8,3.5)  {$11$};

  \draw node at (8.2,0.5) {$i_4$};
  \draw node at (8.2,1.5)  {$i_3$};
  \draw node at (8.2,2.5)  {$i_2$};
  \draw node at (8.2,3.5)  {$i_1$};

  \draw (6, 0.5) to[out=0, in=180] (8, 3.5);
  \draw (6, 1.5) to[out=0, in=180] (8, 0.5);
  \draw (6, 2.5) to[out=0, in=180] (8, 2.5);
  \draw (6, 3.5) to[out=0, in=180] (8, 1.5);

  \end{tikzpicture}
\end{center}

\noindent This concludes our presentation of the braid $J_k(f)$ and its properties. Let us now compare the links we obtain from $J_k(f)$ to those obtained from Richardson braids.

\subsection{Generalized destabilizations and comparison of $\La(u,w)$ and $\La(f)$}\label{ssec:destabilization} The main result of this subsection is to establish that, possibly up to unlinked unknots, the smooth links $\La(u,w)$ and $\La(f)$ are smoothly isotopic.
\begin{thm}\label{thm:Richardson_to_juggling}
Let $(u,w)$ be a positroid pair and $f = ut_{k}w^{-1}$ its associated bounded affine permutation. Then the smooth link $\La^{+}(u,w)\sse\R^3$ is smoothly isotopic to the smooth link $\La(f)\sse\R^3$, up to a possibly empty collection of unlinked unknots.
\end{thm}

Theorem \ref{thm:Richardson_to_juggling} is proven by appropriately using the following lemma iteratively.

\begin{lemma}\label{lem:Markovsmooth} Let $\beta_1\in\Br_n$ be a positive braid on $n$ strands in the generators $\sigma_1,\ldots,\sigma_{n-1}$ and $\beta_2\in\Br_{n+1}$ the lift of a permutation $w\in S_{n+1}$ on $(n+1)$ strands. Set $j:=w^{-1}(n+1)$, let $\tilde{\beta}_1 \in \Br_{n+1}$ be the braid obtained from $\beta_1$
by shifting the indices of the Artin generators up by $1$, and let $\beta'_2 \in \Br_n$ be the braid obtained from $\beta_2$ by removing the strand that ends at the bottom on the right of $\beta_2$.


\begin{itemize}
    \item[$(1)$] Consider the braid words $\eta_1 := \Delta_{n+1}\beta_2 (\sigma_{\ell}\cdot\ldots\cdot\sigma_{j+1}\sigma_j\cdot\ldots\cdot\sigma_2\sigma_1)\tilde{\beta}_1 \in \Br_{n+1}$, an $(n+1)$-stranded positive braid, and $\eta_2:=\Delta_n\beta'_2(\sigma_{\ell - 1}\cdot\ldots\cdot\sigma_{j+1})\beta_1 \in \Br_n$, an $n$-stranded positive braid.
    \\

\noindent  Then the 0-framed smooth closure of $\eta_1$ is smoothly isotopic to that of $\eta_2$.\\

\item[$(2)$] Consider the braid words $\eta_1:=\Delta_{n+1}\beta_2(\sigma_{j}\sigma_{j-1}\cdot\ldots\cdot\sigma_1)\tilde{\beta}_1 \in \Br_{n+1}$, an $(n+1)$-stranded positive braid, and $\eta_2:=\Delta_n\beta'_2\beta_1\in \Br_n$, an $n$-stranded positive braid.
\\

\noindent  Then the 0-framed smooth closure of $\eta_1$ is smoothly isotopic to the unlinked union of the 0-framed smooth closure of  $\eta_2$ and an unknot.
\end{itemize}
\hfill$\Box$
\end{lemma}

\begin{figure}
    \centering
    \includegraphics[scale=0.6]{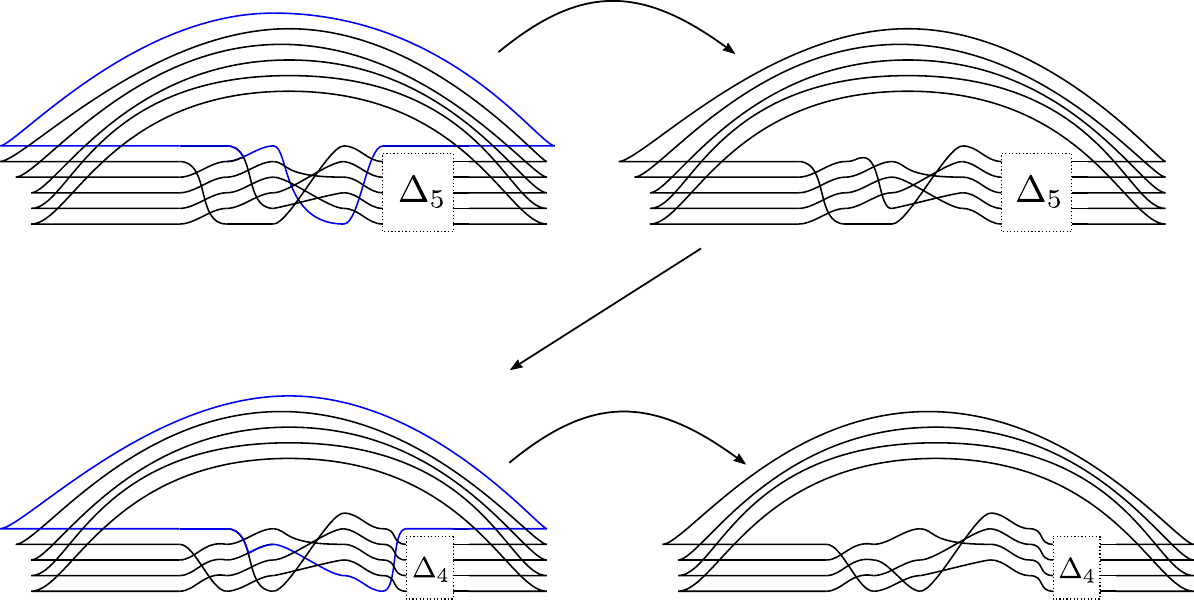}
    \caption{The sequence of moves starting with $R(u, w)$, on the upper left corner, and finishing with $J_k(f)$ on the lower right corner. We have marked in blue the occasions where we use Lemma \ref{lem:Markovsmooth}, and we have chosen to draw the link diagrams as fronts, as in Lemma \ref{lem:Markov}, for easier reference. (So all crossings in these link diagrams are over-crossings.)}
    \label{fig: running example}
\end{figure}

For now, Lemma \ref{lem:Markovsmooth} will be left unproven and we just directly use it. In the next section we will independently prove Lemma \ref{lem:Markov}, which is a stronger version that implies Lemma \ref{lem:Markovsmooth}.

\begin{example}
Before proving Theorem \ref{thm:Richardson_to_juggling}, we verify it in Example \ref{ex: running 1} as follows. Recall that $f = [4, 6, 7, 9, 11, 8]$, so that $w = [3,4,5,6,1,2]$ and $u = [1, 3, 5, 2, 4, 6]$. From Example \ref{ex: markov 44} we obtain:
    \[
    \beta(w) = (\s_2\s_3\s_4\s_5)(\s_1\s_2\s_3\s_4) \Rightarrow \beta(w^*) = (\s_4\s_3\s_2\s_1)(\s_5\s_4\s_3\s_2),
    \]
    and $u^{-1}w_0 = [1,4,2,5,3,6]w_0 = [6,3,5,2,4,1]$. Figure \ref{fig: running example} then shows the sequence of moves, using Lemma \ref{lem:Markovsmooth},  that starts with the Richardson link and ends with the Juggling link.\hfill$\Box$ 
\end{example}

Now, Subsection \ref{ssec:RichardsonBraid} defines the Richardson link $\Lambda(u,w)$ to be smooth 0-framed closure of the braid 
$\Delta_n^{-1}\beta(u^{-1}w_0)\beta(w^*)$
, where $(u,w)$ is a positroid pair. It will be more convenient to write this braid word as 
$\Delta_n^{-2}\Delta_n\beta(u^{-1}w_0)\beta(w^*)$
and focus on the piece 
$\Delta_n\beta(u^{-1}w_0)\beta(w^*)$, where all our changes will take place, see e.g. Figure \ref{fig: running example}. Let $f=f_{u,w} = ut_kw^{-1}$ be the associated bounded affine permutation, as in Subsection \ref{ssec:boundedaffinepermutation}. By the identity \eqref{eq: u from f},
$$
u=[f(i_1)-n,\ldots,f(i_k)-n,f(j_1),\ldots,f(j_{n-k})]
$$
and therefore 
$$
uw_0=[f(j_{n-k}),\ldots,f(j_1),f(i_k)-n,\ldots,f(i_1)-n].
$$
Now note that 
\begin{equation}\label{eq: u-1w_0}
u^{-1}w_0 = (w_0(uw_0)w_0)^{-1}.
\end{equation}
In particular, up to conjugation by the longest element $w_0$, the inverse of the restriction of the permutation $u^{-1}w_0$ to $\{n-k+1,\ldots,n\}$ is the permutation $J^{(2)}(f)$ which sorts $f(i_1),\ldots,f(i_k)$ in decreasing order, see Proposition \ref{prop: J2}. 

The iterative applications of Lemma \ref{lem:Markovsmooth} that will lead to a destabilization procedure from $R_n(u,w)$ to $J_k(f)$ are indexed by the columns of the corresponding Le diagram. The two basic steps we need are as follows:

\begin{itemize}
    \item[(1)] First, suppose that the Le diagram for $(u,w)$ has dots in the rightmost column, with the lowest dot in row $d$. Then we schematically draw the braid 
    as follows:
\begin{center}
\begin{tikzpicture}
\draw (0,0)--(0,3)--(2,3)--(2,0)--(0,0);
\draw (0,4)--(2,4)--(3,1);
\draw [line width=3,color=blue] (0,4)--(2,4)--(2.5,2.5)--(3,3);
\draw (2,3)--(3,4);
\draw (2,2)--(3,3);
\draw (2,1)--(3,2);
\draw (3.7,4) node {$1$};
\draw (3.7,3.5) node {$\vdots$};
\draw (3.7,3) node {\small $d$\normalsize};
\draw (3.7,2) node {$\vdots$};
\draw (3.7,1) node {\small$\lambda_{n-k}^t+1$\normalsize};
\draw[<->,dashed] (-0.2,-0.2)--(3.2,-0.2);
\draw (1.5,-0.5) node {$w^*$};
\draw [line width=3,color=blue] (4.5,3)--(6,0);
\draw (6.2,0) node {$n$};
\draw[<->,dashed] (3.9,-0.2)--(6.2,-0.2);
\draw (4.7,-0.5) node {$u^{-1}w_0$};
\draw[<->,dashed] (6.5,-0.2)--(9.2,-0.2);
\draw (7.5,-0.5) node {$\Delta_n$};
\draw [line width=3,color=blue] (6.5,0)--(9,4);
\draw (1,1.5) node {$\widetilde{w}$};

\draw  (-0.2,-0.2)--(-0.2,-0.8);
\draw (-0.2,-1) node {$x=0$};

\draw  (3.7,-0.2)--(3.7,-0.8);
\draw (3.7,-1) node {$x=1$};

\draw  (6.2,-0.2)--(6.2,-0.8);
\draw (6.2,-1) node {$x=2$};
\end{tikzpicture}
\end{center}

Write $w = (s_{n-\lambda_{n-k}^{t}}\cdots s_{n-1})w'$, so that $w^* = (s_{\lambda_{k-t}^{t}}\cdots s_{1})\widetilde{w}$, and note that, by \eqref{eq: u-1w_0}, $u^{-1}w_0$ is a permutation braid that connects $n$ on the right with $w_0u(n)= w_0(f(j_{n-k})) = w_0(n-d+1) = d$ on the left, where the equation $f(j_{n-k}) = n-d+1$ follows from Lemma \ref{lem: Le and f}. Then, the thick blue strand in the figure above can be removed by Lemma \ref{lem:Markovsmooth}, leaving a braid which can be described as follows. On the side of $w_0ww_0=w^*$, we get
$$
\left(s_{\lambda_{n-k}^{t}}\cdots s_{d}\right)\cdot\widetilde{w}= \widetilde{w}\cdot\left(s_{\lambda_{n-k}^{t}+n-k-1}\cdots s_{d+n-k-1} \right).
$$
On the side of $u^{-1}w_0$, we just remove the strand connecting $d$ to $n$ and leave the rest unchanged. On the side of $\Delta_n$, we remove the strand connecting $n$ to $1$, and get $\Delta_{n-1}$.\\


\item[(2)] Second, suppose that the last column of the Le diagram is empty. Then by Lemma \ref{lem: Le and f} $u(n) = f(j_{n-k}) = n-\lambda_{n-k}^{t}$, so by \eqref{eq: u-1w_0} we get $u^{-1}w_0(\lambda_{n-k}^{t} + 1) = n$, and thus the diagram has the following form:

\begin{center}
\begin{tikzpicture}
\draw (0,0)--(0,3)--(2,3)--(2,0)--(0,0);
\draw (0,4)--(2,4)--(3,1);
\draw [line width=3,color=blue] (0,4)--(2,4)--(3,1);
\draw (2,3)--(3,4);
\draw (2,2)--(3,3);
\draw (2,1)--(3,2);
\draw (3.7,4) node {$1$};
\draw (3.7,3.5) node {$\vdots$};
\draw (3.7,3) node {$\vdots$};
\draw (3.7,2) node {$\vdots$};
\draw (3.7,1) node {\small $\lambda_{n-k}^t+1$\normalsize};
\draw[<->,dashed] (-0.2,-0.2)--(3.2,-0.2);
\draw (1.5,-0.5) node {$w^*$};
\draw [line width=3,color=blue] (4.5,1)--(6,0);
\draw (6.2,0) node {$n$};
\draw[<->,dashed] (3.9,-0.2)--(6.2,-0.2);
\draw (4.7,-0.5) node {$u^{-1}w_0$};
\draw[<->,dashed] (6.5,-0.2)--(9.2,-0.2);
\draw (7.5,-0.5) node {$\Delta_n$};
\draw [line width=3,color=blue] (6.5,0)--(9,4);
\draw (1,1.5) node {$\widetilde{w}$};

\draw  (-0.2,-0.2)--(-0.2,-0.8);
\draw (-0.2,-1) node {$x=0$};

\draw  (3.7,-0.2)--(3.7,-0.8);
\draw (3.7,-1) node {$x=1$};

\draw  (6.2,-0.2)--(6.2,-0.8);
\draw (6.2,-1) node {$x=2$};
\end{tikzpicture}
\end{center}

\noindent In this case the thick line closes up to an unknot. By Lemma \ref{lem:Markovsmooth} this unknot is unlinked from the rest of the link diagram. After moving it away, we are left with $\widetilde{w}$ instead of $w$.
\end{itemize}
\noindent We refer to either of the two operations above as a (generalized) destabilization. Let us now prove Theorem \ref{thm:Richardson_to_juggling} by appropriately applying generalized destabilizations $(n-k)$ times, as follows.

\begin{lemma}
\label{lem: inductive destabilization}
At the $m$-th step, $1\le m\le n-k+1$, of the above destabilization procedure we get an $(n+1-m)$-strand braid  that can be decomposed into the following three parts:
\begin{itemize}
    \item[(a)] The first part $\beta^{(1)}_{m}$ is the product of interval braids
\[
\prod_{h = m}^{n-k}\left(s_{\lambda^{t}_{n-k+1-h}+h-1}\cdots s_h\right) \times
\prod_{\ell = 0}^{m-2}\left(s_{\lambda^{t}_{n-k-m-2+\ell}+n-k-1}\cdots s_{2n-k-m-\ell+3+\inv(n-k-m+2-\ell)-f(j_{n-k-m+2-\ell})}\right),
\]    
where for $m=1$ the second product is assumed to equal 1 and for $m = n-k+1$ the first product is assumed to equal 1.

\vspace{0.1cm}

\item[(b)] The second part $\beta^{(2)}_{m}$ is a permutation braid obtained as a part of the wiring diagram for $u^{-1}w_0$ connecting $1, \dots, n-m+1$ on the right with $u^{-1}(n),\ldots,u^{-1}(m)$ on the left.\\

\item[(c)] The third part is $\Delta_{n-m+1}$.   \end{itemize}
\end{lemma}

\begin{proof}
We prove this by induction in $m$. For notational convenience, we refer to the left end of $\beta^{(1)}_{m}$ as $x=0$, to the right end of $\beta^{(1)}_{m}$ which coincides the left end of $\beta^{(2)}_{m}$ as $x=1$ and to the right end of $\beta^{(2)}_{m}$ as $x=2$ (see figures above). For $m=1$ we get
\begin{equation}\label{eq: w0ww0}
\beta^{(1)}_{1}=w^*= \left(s_{\lambda^{t}_{n-k}}\cdots s_1\right)\cdots\left(s_{\lambda_{1}^{t} + n - k -1}\cdots s_{n-k}\right). 
\end{equation}
and $\beta^{(2)}_{1}=u^{-1}w_0$. 

Assume that the statement holds for all $1 \leq j \leq m$.
To prove the step of induction, we mark in blue the bottom-most strand at $x=2$. It is labeled by $n-m+1$ at $x=2$ and by $(u^{-1}w_0)^{-1}(n-m+1)  = w_0u(n-m+1) = w_0(f(j_{n-k-m+1})) = n+1-f(j_{n-k-m+1})$ at $x=1$. For brevity, denote $m' := n-k-m+1$.

By assumption, the rightmost interval braid in $\beta^{(1)}_{m}$ is
\begin{equation}\label{eq: m step}
s_{m-1+\lambda^{t}_{n-k-m+1}}\cdots s_{m}=s_{m-1+\lambda^{t}_{m'}}\cdots s_{m}.
\end{equation}

We need to verify that we can apply Lemma \ref{lem:Markovsmooth}. Let us count the number of strands \emph{above} the blue strand $n+1-f(j_{m'})$ we have removed at $x=1$. Note that such strands correspond to $1 \leq \ell < m$ such that $n+1-f(j_{\ell'}) < n+1-f(j_{m'})$ where $\ell' :=  n-k-\ell+1$, that is, they correspond to elements $n - k \geq \ell' > m'$ such that $f(j_{\ell'}) > f(j_{m'})$. Thanks to \eqref{eq: noninversions}, there are precisely $n-k-m'-\inv(m')$ such strands. We conclude that the blue strand occupies the position
\[
n + 1 - f(j_{m'}) - (n-k-m' - \inv(m')) = k+1+m'+\inv(m') - f(j_{m'})
\]
and, in order to verify we can apply Lemma \ref{lem:Markovsmooth}, we need to verify that 
\[
k+m'+\inv(m') - f(j_{m'}) \leq \lambda^{t}_{m'} = k+m'-j_{m'} \Leftrightarrow j_{m'} \leq f(j_{m'}) - \inv(m')
\]
which holds thanks to Lemma \ref{lem: inv}. So we can indeed use Lemma \ref{lem:Markovsmooth} in order to destabilize. We have cases:

\begin{itemize}
\item[-] If $j_{m'} = f(j_{m'}) - \inv(m')$, then we remove a disjoint blue unknot and erase the interval braid completely.
\item[-] If $j_{m'} = f(j_{m'}) - \inv(m') - 1$, then we also erase the interval braid completely.
\item[-] Else, we delete the part $(s_{m+k+m'+\inv(m')-f(j_{m'})}\cdots s_{m})$ intersecting the blue strands out of the interval braid \eqref{eq: m step} and are left with 
\[
s_{m-1+\lambda^{t}_{m'}}\cdots s_{m+k+m'+\inv(m')-f(j_{m'})+1} = s_{m-1+\lambda^{t}_{m'}}\cdots s_{n+\inv(m')-f(j_{m'}) + 2}.
\]
Now we push this interval through the $m' - 1$ remaining interval braids and get
\begin{equation}\label{eq: push before substract}
s_{m-1+\lambda^{t}_{m'}+m'-1}\cdots s_{n+\inv(m')-f(j_{m'}) + 2+m'-1}=
s_{n-k-1+\lambda^{t}_{m'}}\cdots s_{n+\inv(m')-f(j_{m'}) + m'+1}.
\end{equation}
This agrees with the interval braid in $\beta^{(1)}_{m+1}$.
\end{itemize}
\end{proof}

\begin{proof}[Proof of Theorem \ref{thm:Richardson_to_juggling}]

We apply $n-k$ destabilization steps in Lemma \ref{lem: inductive destabilization} and obtain a $k$-stranded braid at $m=n-k+1$, which is described by the following three pieces. First,
\[
\beta_{m}^{(1)}=\left(s_{\lambda_1^{t}+n-k-1}\cdots s_{2n-f(j_1) +\inv(1) - n+2}\right)\cdots \left(s_{\lambda_{n-k}^{t}+n-k-1}\cdots s_{2n-k-f(j_{n-k}) +\inv(n-k)+1}\right).
\]
Since we did not relabel the strands yet, we need to shift all the subindices of the Coxeter generators down by $n-k$ and obtain
\[
\left(s_{\lambda_1^{t}-1}\cdots s_{k-f(j_1) +\inv(1) +1+1}\right)\cdots \left(s_{\lambda_{n-k}^{t}-1}\cdots s_{k-f(j_{n-k}) +\inv(n-k)+n-k +1}\right).
\]
 The second part $\beta_{m}^{(2)}$ is a permutation braid obtained as a part of the wiring diagram for $u^{-1}w_0$ connecting $1, \dots, k$ on the right with $2n - f(i_1) + 1, \dots, 2n - f(i_k) +1$ on the left. Finally, the third part is $\Delta_k$. By Corollary \ref{cor: intervals juggling} and Proposition \ref{prop: J2}, this implies that we have obtained $J_k(f)$.
\end{proof}

\subsection{Le Braid}\label{ssec:LeBraid} In this section, we define the Le braid $D_k(\Le)$ associated to a Le diagram $\Le$ and compare it to $J_k(f)$, where $f$ is the bounded affine permutation associated to $\Le$. See Theorem \ref{thm:Lebraid_is_jugglingbraid} below. The Le braid will be a concatenation of two braids in $k$-strands: $D_{k}(\Le) = D_{k}^{(2)}(\Le)D_k^{(1)}(\Le)$, in such a way that $D_{k}^{(1)}(\Le)$ is positive and (typically) not reduced, while $D_{k}^{(2)}(\Le)$ is negative and reduced. 

Let us start with the braid $D_{k}^{(2)}(\Le)$, since it is easier to define. In order to do this, recall the wiring diagram of $\Le$: there are exactly $k$ wires starting on the left border of the Young diagram $\lambda$. We obtain a $k$-stranded braid by:
\begin{itemize}
    \item Reading these wires in the northeast to southwest direction. Note that this direction is opposite to that in Section~\ref{ssec:Comb}.
    \item Delcaring all crossings to be \emph{negative}. 
\end{itemize}

The obtained braid is by definition the braid $D^{(2)}(\Le)$. Clearly, it is reduced. Note that $D^{(2)}(\Le)$ is the inverse of the braid obtained by joining $i_s$ to $f(i_s)$, $s = 1, \dots, k$. See Figure \ref{fig: D2} for an example. 

\begin{figure}[h!]
\begin{tikzpicture}[scale=0.6]
\draw (0,0)--(9,0)--(9,4)--(6,4)--(6,6)--(3,6)--(3,8)--(0,8)--(0,0);

\draw(0,1)--(9,1);
\draw(0,2)--(9,2);
\draw(0,3)--(9,3);
\draw(0,4)--(6,4);
\draw(0,5)--(6,5);
\draw(0,6)--(3,6);
\draw(0,7)--(3,7);

\draw(1,0)--(1,8);
\draw(2,0)--(2,8);
\draw(3,0)--(3,6);
\draw(4,0)--(4,6);
\draw(5,0)--(5,6);
\draw(6,0)--(6,4);
\draw(7,0)--(7,4);
\draw(8,0)--(8,4);

\draw(0.5,7.5) node {$\bullet$};
\draw(1.5,7.5) node {$\bullet$};
\draw(2.5,7.5) node {$\bullet$};

\draw(1.5, 6.5) node{$\bullet$};
\draw(2.5, 6.5) node{$\bullet$};

\draw(3.5, 5.5) node{$\bullet$};
\draw(5.5, 5.5) node{$\bullet$};

\draw(1.5, 4.5) node{$\bullet$};
\draw(4.5, 4.5) node{$\bullet$};
\draw(5.5, 4.5) node{$\bullet$};

\draw(4.5, 3.5) node{$\bullet$};
\draw(6.5, 3.5) node{$\bullet$};
\draw(7.5, 3.5) node{$\bullet$};
\draw(8.5, 3.5) node{$\bullet$};

\draw(6.5, 2.5) node{$\bullet$};
\draw(7.5, 2.5) node{$\bullet$};
\draw(8.5, 2.5) node{$\bullet$};

\draw(7.5, 1.5) node{$\bullet$};

\draw(4.5, 0.5) node{$\bullet$};

\draw[<-, color=blue] (0,0.5) to (4, 0.5) to[out=0, in=270] (4.5, 1) to (4.5, 3) to[out=90, in=180] (5, 3.5) to (6, 3.5) to[out=0, in=270] (6.5, 4);

\draw[<-, color=blue](0,1.5) to (7, 1.5) to[out=0, in=270] (7.5, 2) to[out=90, in=180] (8, 2.5) to[out=0, in=270] (8.5, 3) to [out=90, in=180] (9, 3.5); 

\draw[<-, color=blue](0, 2.5) to (6, 2.5) to[out=0, in=270] (6.5, 3) to[out=90, in=180] (7, 3.5) to[out=0, in=270] (7.5, 4);

\draw[<-, color=blue] (0, 3.5) to (4, 3.5) to[out=0, in=270] (4.5, 4) to[out=90, in=180] (5, 4.5) to[out=0, in=260] (5.5, 5) to[out=90, in=180] (6, 5.5);

\draw[<-,color=blue] (0, 4.5) to (1, 4.5) to[out=0, in=270] (1.5, 5) to (1.5, 6) to[out=90, in=180] (2, 6.5) to[out=0, in=270] (2.5, 7) to[out=90, in=180] (3, 7.5);

\draw[<-,color=blue] (0, 5.5) to (3, 5.5) to [out=0, in=270] (3.5, 6);

\draw[<-,color=blue] (0, 6.5) to (1, 6.5) to[out=0, in=270] (1.5, 7) to [out=90, in=180] (2, 7.5) to [out=0, in=270] (2.5, 8);

\draw[<-,color=blue] (0, 7.5) to[out=0, in=270] (0.5, 8);
\end{tikzpicture}
\caption{The braid $D^{(2)}_{k}(\Le)$. All of the crossings are negative.}
\label{fig: D2}
\end{figure}


\begin{lemma}\label{lem: D2 vs J2}
Let $\Le$ be a Le diagram and $f$ its associated bounded affine permutation. Then, the $k$-stranded braids $D^{(2)}_{k}(\Le)$ and $\Delta_{k}^{-1}\Jb(f)$ are braid equivalent. 
\end{lemma}
\begin{proof}
   The statement is equivalent to showing that the positive braids $\Delta_{k}(f)$ and $\Jb(f)(D^{(2)}(\Le))^{-1}$ are braid equivalent. This is equivalent to showing that in the positive braid $\Jb(f)(D^{(2)}(\Le))^{-1}$ each pair of strings cross exactly once. For this, we picture this braid as follows:

   \begin{center}
    	\begin{tikzpicture}[scale = 0.6]
  \draw node at (5.8, 0.5) {$i_1$};
  \draw node at (5.8,1.5) {$i_2$};
  \draw node at (5.8,2.5) {$\vdots$};
  \draw node at (5.8,3.5)  {$i_k$};
  \draw (6.5, 0) -- (6.5, 4) -- (9.2, 4) -- (9.2, 0) -- cycle;

  \draw node at (10.2,0.5) {$f(i_{a_1})$};
  \draw node at (10.2,1.5)  {$f(i_{a_2})$};
  \draw node at (10.2,2.5)  {$\vdots$};
  \draw node at (10.2,3.5)  {$f(i_{a_k})$};

  \draw (6, 0.5) -- (6.5, 0.5);
  \draw (6, 1.5) -- (6.5, 1.5);
  \draw (6, 3.5) -- (6.5, 3.5);

    \draw (9.2, 0.5) -- (9.5, 0.5);
  \draw (9.2, 1.5) -- (9.5, 1.5);
  \draw (9.2, 3.5) -- (9.5, 3.5);

  \draw node at (14.5, 0.5) {$i_k$};
  \draw node at (14.8,1.5) {$i_{k-1}$};
  \draw node at (14.5,2.5) {$\vdots$};
  \draw node at (14.5,3.5)  {$i_1$};

    \draw (10.9, 0.5) -- (11.2, 0.5);
  \draw (10.9, 1.5) -- (11.2, 1.5);
  \draw (10.9, 3.5) -- (11.2, 3.5);

    \draw (14, 0.5) -- (14.3, 0.5);
  \draw (14, 1.5) -- (14.3, 1.5);
  \draw (14, 3.5) -- (14.3, 3.5);

  \draw (11.2, 0) -- (11.2, 4) -- (14, 4) -- (14, 0) -- cycle;

  \draw node at (12.6, 2) {$\Jb(f)$};

  \draw node at (7.85,2) {$D^{(2)}(\Le)^{-1}$};


  \end{tikzpicture}
\end{center}

Here $f(i_{a_k}), \dots, f(i_{a_1})$ in the middle column are ordered in decreasing order when reading top to bottom, the braid $D^{(2)}(\Le)^{-1}$ joins $i_{s}$ to $f(i_{s})$, and the braid $\Jb(f)$ joins $f(i_{s})$ to $i_s$.

Let $1 \leq s < t \leq k$. If $f(i_s) < f(i_t)$, the strands whose left labels are $i_s$ and $i_t$ will not meet in the $D^{(2)}(\Le)^{-1}$ part of the braid, but they will meet exactly once in the $\Jb(f)$ part of it. If, on the contrary, $f(i_t) < f(i_s)$, then the same strands will meet exactly once in the $D^{(2)}(\Le)^{-1}$ part of the braid, and they will not meet in the $\Jb(f)$ part of it. The result follows. 
\end{proof}

Let us now define $D^{(1)}(\Le)$. We need an auxiliary diagram, close in spirit to the wiring diagram, that essentially keeps track of all the crossings between non-special circumferences in the juggling diagram of the associated bounded affine permutation. 

\begin{definition}\label{def:bounded wiring}
The \emph{bounded wiring diagram} of a Le diagram $\Le$ is defined as follows. For each non-empty column of $\Le$, draw a wire from the top of the column to the lowest dot on it, always passing through the left of every dot of the column. Right after the lowest dot on the column take a right U-turn (going under that lowest dot), and then proceed with the usual wiring rules. This wire connects the top of the starting column with either the rightmost edge of a row, or the top of a different column. If the latter situation occurs, then further connect the wire with the wire starting at the top of this second column.\hfill$\Box$
\end{definition}

\noindent See Figure \ref{fig:bounded_wiring_diagram} for an instance of such a bounded wiring diagram from Definition \ref{def:bounded wiring}.\\

\begin{figure}[h!]
\begin{tikzpicture}[scale=0.6]
\draw (0,0)--(9,0)--(9,4)--(6,4)--(6,6)--(3,6)--(3,8)--(0,8)--(0,0);

\draw(0,1)--(9,1);
\draw(0,2)--(9,2);
\draw(0,3)--(9,3);
\draw(0,4)--(6,4);
\draw(0,5)--(6,5);
\draw(0,6)--(3,6);
\draw(0,7)--(3,7);

\draw(1,0)--(1,8);
\draw(2,0)--(2,8);
\draw(3,0)--(3,6);
\draw(4,0)--(4,6);
\draw(5,0)--(5,6);
\draw(6,0)--(6,4);
\draw(7,0)--(7,4);
\draw(8,0)--(8,4);

\draw(0.5,7.5) node {$\bullet$};
\draw(1.5,7.5) node {$\bullet$};
\draw(2.5,7.5) node {$\bullet$};

\draw(1.5, 6.5) node{$\bullet$};
\draw(2.5, 6.5) node{$\bullet$};

\draw(3.5, 5.5) node{$\bullet$};
\draw(5.5, 5.5) node{$\bullet$};

\draw(1.5, 4.5) node{$\bullet$};
\draw(4.5, 4.5) node{$\bullet$};
\draw(5.5, 4.5) node{$\bullet$};

\draw(4.5, 3.5) node{$\bullet$};
\draw(6.5, 3.5) node{$\bullet$};
\draw(7.5, 3.5) node{$\bullet$};
\draw(8.5, 3.5) node{$\bullet$};

\draw(6.5, 2.5) node{$\bullet$};
\draw(7.5, 2.5) node{$\bullet$};
\draw(8.5, 2.5) node{$\bullet$};

\draw(7.5, 1.5) node{$\bullet$};

\draw(4.5, 0.5) node{$\bullet$};

\draw[red] (0.25, 8) to (0.25, 7.25) to[out=325, in=225] (0.75, 7.25) to[out=90, in=250] (0.75, 7.5) to[out=20, in=205](1.15, 7.5) to[out=25, in=205] (1.25, 8);
\draw[red] (1.25,8) to (1.25,4.25) to[out=325, in=225] (1.75,4.25) to [out=90, in=250] (1.75, 4.5) to[out=10, in=190](4.15, 4.5) to[out=80, in=260](4.25,6);
\draw[red] (4.25, 6) to[out=275, in=90] (4.25, 0.25) to[out=325, in=225] (4.75, 0.25) to[out=90,in=250](4.75, 0.5) to[out=5,in=180](9,0.5);

\draw[blue] (2.25, 8) to (2.25, 6.25) to[out=325, in=225] (2.75, 6.25) to[out=90, in=250] (2.75, 6.5) to (3,6.5); 

\draw[cyan] (3.25, 6) to (3.25, 5.25) to[out=325, in=225] (3.75, 5.25) to[out=90, in=250](3.75, 5.5) to[out=10, in=175] (5.15, 5.5) to[out=25, in=205] (5.25, 6);
\draw[cyan] (5.25, 6) to (5.25, 4.25) to[out=325, in=225] (5.75, 4.25) to[out=90, in=250] (5.75, 4.5) to [out=5, in=180] (6, 4.5);

\draw[purple] (6.25, 4) to (6.25, 2.25) to[out=325, in=225] (6.75, 2.25) to[out=90, in=250] (6.75, 2.5) to[out=5, in=185] (7.15, 2.5) to[out=25, in=260] (7.5, 3.25) to[out=70,in=190] (7.75, 3.5) to[out=5, in=175](8.15, 3.5) to[out=25, in=205] (8.25, 4);
\draw[purple] (8.25, 4) to (8.25, 2.25) to[out=325, in=225] (8.75, 2.25) to[out=90, in=250] (8.75, 2.5) to (9,2.5);

\draw[teal] (7.25, 4) to (7.25, 1.25) to[out=325, in=225] (7.75, 1.25) to[out=90, in=250] (7.75, 1.5) to [out=10, in=180] (9, 1.5);

\draw(0.5, 8.3) node {$1$};
\draw(2.5, 8.3) node {$2$};
\draw(3.2, 7.5) node{$3$};
\draw(3.5, 6.3) node{$4$};
\draw(6.2, 5.5) node{$5$};
\draw(6.5, 4.3) node{$6$};
\draw(7.5, 4.3) node{$7$};
\draw(9.2, 3.5) node{$8$};
\end{tikzpicture}
\caption{The labeled diagram $\widetilde{\wire{\Le}}$.}\label{fig:bounded_wiring_diagram}
\end{figure}

Definition \ref{def:bounded wiring} creates a (typically) \emph{non-reduced} braid with $k'$ strands, $k' \leq k$. Indeed, note that we have an injective map from the set of wires to the rows of $\Le$. In fact, it follows from identities \eqref{eq: w from f} and \eqref{eq: u from f} that these wires correspond to chains of non-special circumferences in the juggling braid of the associated bounded affine permutation $f$. Moreover, the crossings in the bounded wiring diagram correspond precisely to the crossings between these non-special circumferences. 

Let us enhance the bounded wiring diagram by enumerating both the right ends of the rows which do not have a wire ending on them and the top ends of those columns where a wire starts, as follows. Reading the steps of the diagram from northwest to southeast (both vertical \emph{and} horizontal steps), we enumerate the following two types of steps in the order in which these steps are found:
\begin{itemize}
    \item[-] A horizontal step at the top of a column which is the initial step of a wire. (In particular, there is no wire starting in a column to its left that goes up to the top of that column.)\\
    
    \item[-] A vertical step at the right of a row which is not the endpoint of a wire. 
\end{itemize}

\noindent In this manner, we have labeled exactly $k$ steps. Let us denote by $\wire{\Le}$ the bounded wiring diagram of $\Le$ and by $\widetilde{\wire{\Le}}$ its labeled version.  We are ready to define the braid $D^{(1)}(\Le)$.

\begin{definition}
The $k$-stranded braid $D_{k}^{(1)}(\Le)$ of a $\Le$ diagram with at most $k$ rows is defined inductively on the columns of $\Le$ as follows.
\begin{enumerate}
\item If $\Le$ is the empty Young diagram, $D_{k}^{(1)}(\Le)$ is the trivial braid on $k$ strands.\\
\item Assume $D_{k}^{(1)}(\Le)$ has been defined and let $\Le'$ be a Le diagram obtained from $\Le$ by attaching a column of height $h$ to the left of $\Le$. Then:

\begin{itemize}
    \item[(i)] If this column has no dots, then we define $D^{(1)}_{k}(\Le') := D^{(1)}_{k}(\Le)$.\\

    \item[(ii)] Else, we first draw the bounded wiring diagram $\wire{\Le'}$ of $\Le'$. The wire associated to the first column of $\Le'$, i.e.~to the column that does not belong to $\Le$, ends in a labeled step of $\widetilde{\wire{\Le}}$; say with label $t$. 
    Then we define
\[
D^{(1)}_{k}(\Le') := (\sigma_{h-1}\cdots \sigma_{k-t+1})\cdot D^{(1)}_{k}(\Le). 
\]
\hfill$\Box$

See Figure \ref{fig: adding columns} for an example. 
\end{itemize}
\end{enumerate}
\end{definition}

\begin{figure}[h!]
\begin{tikzpicture}[scale=0.6]

\draw(0,4)--(0,8);

\draw node at (0.2, 7.5) {$1$};
\draw node at (0.2, 6.5) {$2$};
\draw node at (0.2, 5.5) {$3$};
\draw node at (0.2, 4.5) {$4$};

\draw (1,4.5)--(2,4.5);
\draw (1,5.5)--(2,5.5);
\draw (1,6.5)--(2,6.5);
\draw (1,7.5)--(2,7.5);

\draw node at (1.5, 3) {$D_k^{(1)}(\Le)$};

\draw node at (4, 6) {\huge $\rightsquigarrow$};

\draw(7,4)--(7,8);
\draw[dashed] (6,4)--(7,4)--(7,8)--(6,8)--cycle;
\draw[dashed] (6,5)--(7,5); \draw[dashed](6,6)--(7,6); \draw[dashed] (6,7)--(7,7);

\draw node at (6.5, 7.5) {$\bullet$};
\draw node at (6.5, 6.5) {$\bullet$};
\draw node at (6.5, 4.5) {$\bullet$};
\draw node at (7.2, 7.5) {$1$};
\draw node at (7.2, 6.5) {$2$};
\draw node at (7.2, 5.5) {$3$};
\draw node at (7.2, 4.5) {$4$};

\draw[blue] (6.25, 8) to (6.25, 4.3) to[out=270, in=270)] (6.75, 4.3) to[out=90, in=180] (7, 4.5);

\draw (10,4.5) to[out=0, in=180] (11,5.5);
\draw (10, 5.5) to[out=0, in=180] (11, 6.5);
\draw (10,6.5) to[out=0, in=180] (11,7.5);
\draw (10,7.5) to[out=0, in=180] (11,4.5);
\draw (10, 4.5) -- (9, 4.5);
\draw (10, 5.5) -- (9, 5.5);
\draw (10,6.5)--(9,6.5);
\draw (10,7.5)--(9,7.5);
\draw node at (10, 3) {$D_k^{(1)}(\Le')$};

\draw[dashed] (12,2)--(12,8);

\draw(13,4)--(13,8);
\draw (13,4)--(14,4)--(14,8)--(13,8)--cycle;
\draw (13,5)--(14,5); \draw(13,6)--(14,6); \draw (13,7)--(14,7);

\draw node at (13.5, 7.5) {$\bullet$};
\draw node at (13.5, 6.5) {$\bullet$};
\draw node at (13.5, 4.5) {$\bullet$};
\draw node at (13.5, 8.3) {$1$};
\draw node at (14.2, 7.5) {$2$};
\draw node at (14.2, 6.5) {$3$};
\draw node at (14.2, 5.5) {$4$};

\draw[blue] (13.25, 8) to (13.25, 4.3) to[out=270, in=270)] (13.75, 4.3) to[out=90, in=180] (14, 4.5);

\draw node at (17.5, 6) {\huge $\rightsquigarrow$};

\draw (15,4.5) to[out=0,in=180] (16, 5.5);
\draw (15,5.5) to[out=0,in=180] (16, 6.5);
\draw (15,6.5) to[out=0,in=180] (16, 7.5);
\draw (15,7.5) to[out=0,in=180] (16, 4.5);
\draw node at (15.5, 3) {$D_k^{(1)}(\Le)$};

\draw(20,4)--(20,8);
\draw (20,4)--(21,4)--(21,8)--(20,8)--cycle;
\draw (20,5)--(21,5); \draw(20,6)--(21,6); \draw (20,7)--(21,7);

\draw node at (20.5, 7.5) {$\bullet$};
\draw node at (20.5, 6.5) {$\bullet$};
\draw node at (20.5, 4.5) {$\bullet$};
\draw node at (20.5, 8.2) {$1$};
\draw node at (21.2, 7.5) {$2$};
\draw node at (21.2, 6.5) {$3$};
\draw node at (21.2, 5.5) {$4$};

\draw[blue] (20.25, 8) to (20.25, 4.3) to[out=270, in=270)] (20.75, 4.3) to[out=90, in=180] (21, 4.5);

\draw[dashed] (19,4)--(19,8); \draw[dashed] (19,4)--(20,4); \draw[dashed] (19,5)--(20,5); \draw[dashed] (19,6)--(20,6); \draw[dashed] (19,7)--(20,7); \draw[dashed] (19,8)--(20,8);

\draw node at (19.5, 7.5) {$\bullet$};
\draw node at (19.5, 4.5) {$\bullet$};

\draw[red] (19.25, 8) to (19.25, 4.3) to[out=270, in=270)] (19.75, 4.3) to[out=90, in=180] (20, 4.5) to[out=0, in=270] (20.5, 5) to (20.5, 6) to[out=90, in=180] (21, 6.5);

\draw (23,4.5) to[out=0,in=180] (24,5.5);
\draw (23,5.5) to[out=0,in=180] (24,6.5);
\draw (23,6.5) to[out=0,in=180] (24,7.5);
\draw (23,7.5) to[out=0,in=180] (24,4.5);
\draw (24, 7.5)--(25, 7.5);
\draw (24,6.5) to[out=0,in=180] (25,4.5);
\draw (24,5.5) to[out=0,in=180] (25,6.5);
\draw (24,4.5) to[out=0,in=180] (25,5.5);
\draw node at (24, 3) {$D_k^{(1)}(\Le')$};
\end{tikzpicture}
\caption{Adding a column to the left of a Le diagram, and its effect on the braid $D_k^{(1)}(\Le)$. Here $k = 4$. On the left-hand side of the figure, we add a column to a Le diagram associated to the empty partition. Since we are adding a column of height $h=4$ and the endpoint of the new strand is $t = 4$, we have to multiply by the interval braid $\sigma_{h-1} \dots \sigma_{k-t+1} = \sigma_3\sigma_2\sigma_1$. On the right-hand side of the figure, we add a column to a nonempty Le diagram. On this side, $h = 4$ and $t = 3$, so we have $D_k^{(1)}(\Le') = (\sigma_3\sigma_2)D_k^{(1)}(\Le)$. }
\label{fig: adding columns}
\end{figure}

\noindent The comparison to $J_k(f)$ needs the analogue of Lemma \ref{lem: D2 vs J2}, which reads as follows:

\begin{lemma}\label{lem: D1 vs J1}
Let $\Le$ be a Le diagram and $f$ its associated bounded affine permutation. Then the $k$-stranded braids $D^{(1)}_{k}(\Le)$ and $\Ja(f)$ are braid equivalent. 
\end{lemma}
\begin{proof}
 Let us work by induction on the number of columns of $\Le$. If $\Le$ has no columns, i.e. if it is associated to the empty Young diagram, then $D^{(1)}_{k}(\Le)$ is the trivial braid. In that case, all the arcs in the juggling diagram for $f$ are special, which implies that $\Ja(f)$ is the trivial braid as well. Let us now assume the result for the Le diagram $\Le$, with associated bounded affine permutation $f$, and let $\Le'$ be a Le diagram (with permutation $f'$) obtained from $\Le$ by adjoining an extra column to the left. We let $n$ be the width of the Le diagram $\Le$, so that $n+1$ is the width of $\Le'$.
 
 \noindent If this added column is empty then $D^{(1)}_{k}(\Le) = D^{(1)}_{k}(\Le')$  by definition. Since $f'(1) = 1$ and $f'(i) = f(i-1) + 1$ for $i > 1$, the juggling braids for $f$ and $f'$ coincide, and thus $\Ja(f) = \Ja(f')$.
 
\noindent If this added column is non-empty, let us study first how the bounded affine permutation $f'$ is obtained from the bounded affine permutation $f$. For this, let $h$ be the height of the added column. Note that, if $i < k-h+1$ then $f'(i) = n+1+i = f(i) + 1$. On the other hand, if $i > k-h+1$ then $f'(i) = f(i-1)+1$, this follows from \eqref{eq: u from f}, see e.g. Figure \ref{fig: wiring diagram}. Finally, the juggling diagram of $f'$ is then obtained from that of $f$ by inserting a new non-special arc, which is precisely the arc joining $f'(k-h+1)$ to $k-h+1$. Recall now that $t$ is the label in $\widetilde{\wire{\Le}}$ of the strand starting in this new column of $\Le'$. Then, $k-t+1$ is the position of the strand in $J^{(1)}(f)$ that we join to this new arc. The end position of this strand is $k-(k-h+1) + 1 = h$. Thus, $J^{(1)}(f') = \sigma_{h-1}\cdots \sigma_{k-t+1}J^{(1)}(f)$ and the result follows.
\end{proof}

\noindent Lemmas \ref{lem: D2 vs J2} and \ref{lem: D1 vs J1} thus imply the following result:

\begin{thm}\label{thm:Lebraid_is_jugglingbraid}
Let $\Le$ be a Le diagram with corresponding bounded affine permutation $f$. Then, the braids $\Delta_k^{2}D_{k}(\Le)$ and $\Delta_kJ_{k}(f)$ are braid equivalent. 
\end{thm}
\begin{proof}
By Lemma \ref{lem: D1 vs J1}, $\Delta_{k}^{2}D_{k}(\Le)=
\Delta_{k}^{2}D_k^{(2)}(\Le)D_{k}^{(1)}(\Le)
=\Delta_{k}^{2}D_{k}^{(2)}(\Le)J_{k}^{(1)}(f)$. By Lemma \ref{lem: D2 vs J2}, this also equals $\Delta_{k}^{2}D_{k}^{(2)}(\Le)J_{k}^{(1)}(f) = \Delta_k^{2}\Delta_k^{-1}J^{(2)}_{k}(f)J_{k}^{(1)}(f) = \Delta_{k}J_k(f)$.
\end{proof}

\noindent This concludes our discussion of braids directly associated to Le-diagrams. A Legendrian link associated to $D_k(\Le)$ will be discussed in Section \ref{sec:Legendrian}.

    

\subsection{Matrix Braids}\label{ssec:MatrixBraids} 
Let us finally introduce cyclic rank matrices $r=(r_{ij})$, their associated matrix braids $M_k(r)\in\Br_k$, and conclude Theorem \ref{thm:main1}.(ii).


\begin{definition}\label{def:rankmatrix}
A cyclic rank matrix of type $(k,n)$ is an array $r=(r_{ij})$ indexed by $(i,j)\in\Z^2$ satisfying the following conditions:
\begin{itemize}
	\item[(i)] $r_{ij}= 0$ if $j<i$ and $r_{ij}=k$ if $i+n-1\leq j$.
	\item[(ii)] $r_{ij}-r_{(i+1)j}\in\{0,1\}$, $r_{ij}-r_{i(j-1)}\in\{0,1\}$, and $r_{ij}=r_{(i+n)(j+n)}$, for all $i,j\in\Z$.
	\item[(iii)] If $r_{(i+1)(j-1)}=r_{(i+1)j}=r_{i(j-1)}$ then $r_{ij}=r_{(i+1)(j-1)}$.
\end{itemize}
Note that we can restrict to the grid $j\in[1,n]$ 
due to the condition $r_{ij}=r_{(i+n)(j+n)}$, for all $i,j\in\Z$.\hfill$\Box$
\end{definition}

\noindent Given a cyclic rank matrix $r=(r_{ij})$, and each $i\in\Z$, there is a unique index $f(i)$ such that
$$r_{i\mbox{ }f(i)}=r_{(i+1)\mbox{ }f(i)}=r_{i \mbox{ }(f(i)-1)}=r_{(i+1)\mbox{ }(f(i)-1)}+1.$$
Then the map $f:\Z\lr\Z$ defined by $f(i)=j$ if and only if
$$r_{i\mbox{}j}=r_{(i+1)\mbox{}j}=r_{i \mbox{}(j-1)}=r_{(i+1)\mbox{}(j-1)}+1,$$
defines a bounded affine permutation. In fact, this establishes a bijection between cyclic rank matrices and bounded affine permutations, as explained in \cite[Section 3.3]{KLS}.

\begin{remark}\label{rmk: up turn}
Note that $r_{i\mbox{}(j+1)} = r_{i \mbox{}j} + 1 = r_{(i+1) \mbox{} j} + 1 = r_{(i+1) \mbox{}(j+1)} + 1$ can only happen if $i = j+1$ and $r_{i \mbox{} j} = r_{(i+1) \mbox{} j} = r_{(i+1) \mbox{}(j+1)} = 0$, see e.g. \cite[Corollary 3.12]{KLS}. 
\end{remark}

\begin{center}
	\begin{figure}[h!]
		\centering
		\includegraphics[scale=0.55]{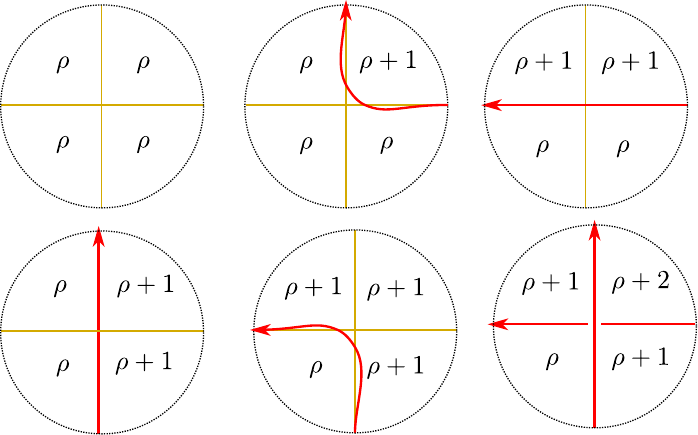}
		\caption{The local models for the matrix braid $M_k(r)$ associated to a cyclic rank matrix $r=(r_{ij})$, drawn near each four entries of the matrix. The value of a given entry $r_{ij}$ is denoted by $r_{ij}=\rho$ and the braid is depicted in red strands. The yellow lines are used to separate the matrix entries of $r$. Note that, by Remark \ref{rmk: up turn}, the middle model of the top row can only happen if $\rho = 0$.}
		\label{fig:MatrixBraid}
	\end{figure}
\end{center}

\begin{definition}\label{def:braidmatrix}
Let $r=(r_{ij})$ be a cyclic rank matrix of type $(k,n)$, $(i,j)\in\Z^2$. By definition, the infinite matrix braid $M^\infty_k(r)$ 
 is given by the tangle diagram obtained by drawing in $\R^2$ the six local tangles according to Figure \ref{fig:MatrixBraid}. By definition, the matrix braid $M_k(r)\in\cB_k$ is obtained from the infinite matrix braid $M^\infty_k(r)\in\cB_k$ by restricting its diagram to the grid $j\in[1,n]$. We define the matrix link $\La(r)\sse\R^3$ to be the smooth 0-framed closure of $M_k(r)$.\hfill$\Box$
\end{definition}

Thanks to Remark \ref{rmk: up turn}, the following gives a procedure for drawing the infinite matrix braid $M_{k}^{\infty}(r)$ for a cyclic rank matrix associated to a bounded affine permutation $f$. 

\begin{itemize}
    \item[-] For each $i$ such that $i \neq f(i)$, connect $(i, f(i))$ to $(i, i)$ using a horizontal line.
    \item[-] For each $i$ such that $i \neq f(i)$, connect $(i,i)$ to $(f^{-1}(i), i)$ using a vertical line. 
\end{itemize}

See Figure \ref{fig: rank matrix} for an example of a matrix braid $M_k(r)$. We remark that we are using matrix notation, so $(i,j)$ is the $ij$-th entry of a matrix: the coordinate $i$ increases down, and the coordinate $j$ increases to the right. 

\begin{remark}
In Definition \ref{def:braidmatrix}, we orient the strands so that they point northwest.\hfill$\Box$
\end{remark}

\begin{figure}[h]
		\begin{tikzpicture}[scale=0.6]
		\filldraw[lightgray] (4,6)--(4,7)--(5,7)--(5,6);
		\filldraw[lightgray] (6,5)--(6,6)--(7,6)--(7,5);
		\filldraw[lightgray] (8,4)--(8,5)--(9,5)--(9,4);
		\filldraw[lightgray] (5,3)--(5,4)--(6,4)--(6,3);
		\filldraw[lightgray] (7,2)--(7,3)--(8,3)--(8,2);
		\filldraw[lightgray] (9,1)--(9,2)--(10,2)--(10,1);
		\filldraw[lightgray] (10,0)--(10,1)--(11,1)--(11,0);

  \filldraw[lightgray] (12,0)--(12,-1)--(13,-1)--(13,0);

    \filldraw[lightgray] (14,-1)--(14,-2)--(15,-2)--(15,-1);

        \filldraw[lightgray] (11, -2) -- (12, -2) -- (12, -3) -- (11,-3);
		
		\draw[step=1] (0,-1) grid (10,7);
            \draw[step=1] (0, -3) grid (10, -1);

		\draw (-1,8)--(0,7);
		\draw (-0.7,7.3) node {$i$};
		\draw (-0.3,7.7) node {$j$};
		\draw (-0.5,6.5) node {$-3$};
		\draw (-0.5,5.5) node {$-2$};
		\draw (-0.5,4.5) node {$-1$};
		\draw (-0.5,3.5) node {$0$};
		\draw (-0.5,2.5) node {$1$};
		\draw (-0.5,1.5) node {$2$};
		\draw (-0.5,0.5) node {$3$};
		\draw (-0.5,-0.5) node {$4$};
            \draw (-0.5, -1.5)  node {$5$};
            \draw (-0.5, -2.5) node {$6$};

		\draw (0.5,7.5) node {$-3$};
		\draw (1.5,7.5) node {$-2$};
		\draw (2.5,7.5) node {$-1$};
		\draw (3.5,7.5) node {$0$};
		\draw (4.5,7.5) node {$1$};
		\draw (5.5,7.5) node {$2$};
		\draw (6.5,7.5) node {$3$};
		\draw (7.5,7.5) node {$4$};
		\draw (8.5,7.5) node {$5$};
		\draw (9.5,7.5) node {$6$};
		
		\draw (0.5,6.5) node {$1$};
		\draw (1.5,6.5) node {$2$};
		\draw (2.5,6.5) node {$3$};
		\draw (3.5,6.5) node {$4$};
		\draw (4.5,6.5) node {$4$};
		\draw (5.5,6.5) node {$4$};
		\draw (6.5,6.5) node {$4$};
		\draw (7.5,6.5) node {$4$};
		\draw (8.5,6.5) node {$4$};
		\draw (9.5,6.5) node {$4$};

		\draw (1.5,5.5) node {$1$};
		\draw (2.5,5.5) node {$2$};
		\draw (3.5,5.5) node {$3$};
		\draw (4.5,5.5) node {$4$};
		\draw (5.5,5.5) node {$4$};
		\draw (6.5,5.5) node {$4$};
		\draw (7.5,5.5) node {$4$};
		\draw (8.5,5.5) node {$4$};
		\draw (9.5,5.5) node {$4$};

		\draw (2.5,4.5) node {$1$};
		\draw (3.5,4.5) node {$2$};
		\draw (4.5,4.5) node {$3$};
		\draw (5.5,4.5) node {$3$};
		\draw (6.5,4.5) node {$4$};
		\draw (7.5,4.5) node {$4$};
		\draw (8.5,4.5) node {$4$};
		\draw (9.5,4.5) node {$4$};

		\draw (3.5,3.5) node {$1$};
		\draw (4.5,3.5) node {$2$};
		\draw (5.5,3.5) node {$2$};
		\draw (6.5,3.5) node {$3$};
		\draw (7.5,3.5) node {$3$};
		\draw (8.5,3.5) node {$4$};
		\draw (9.5,3.5) node {$4$};

		\draw (4.5,2.5) node {$1$};
		\draw (5.5,2.5) node {$2$};
		\draw (6.5,2.5) node {$3$};
		\draw (7.5,2.5) node {$3$};
		\draw (8.5,2.5) node {$4$};
		\draw (9.5,2.5) node {$4$};
		
		\draw (5.5,1.5) node {$1$};
		\draw (6.5,1.5) node {$2$};
		\draw (7.5,1.5) node {$3$};
		\draw (8.5,1.5) node {$4$};
		\draw (9.5,1.5) node {$4$};
		
		\draw (6.5,0.5) node {$1$};
		\draw (7.5,0.5) node {$2$};
		\draw (8.5,0.5) node {$3$};
		\draw (9.5,0.5) node {$4$};
		
		\draw (7.5,-0.5) node {$1$};
		\draw (8.5,-0.5) node {$2$};
		\draw (9.5,-0.5) node {$3$};

            \draw (8.5, -1.5) node {$1$};
            \draw (9.5, -1.5) node {$2$};

            \draw (9.5, -2.5) node {$1$};
		
		\draw [line width=2, color=blue] (9.8, -3) -- (9, -3) --(9,1) -- (5,1) -- (5,3) -- (3.8,3);

  \draw [line width=2, color=blue, dashed] (3.8, 3)--(3,3)--(3,7);

  \draw [line width=2, color=blue, dashed] (10, -3)--(9.8, -3);

        \draw [line width=2, color=magenta] (9.8, -2) -- (8, -2) -- (8, 4) -- (3.8, 4);

        \draw [line width=2, color=magenta, dashed] (3.8, 4)--(2, 4)--(2, 7);

        \draw [line width=2, color=magenta, dashed] (10, -2)--(9.8, -2);

        \draw [line width=2, color=red] (9.8, -1) -- (7, -1) -- (7,2) -- (4,2) -- (4,6) -- (3.8,6);

        \draw [line width=2, color=red, dashed] (3.8, 6)--(0,6)--(0,7);
        \draw [line width=2, color=red, dashed] (10, -1)--(9.8, -1);

        \draw [line width = 2, color=cyan] (9.8, 0) -- (6, 0) -- (6, 5) -- (3.8,5);

        \draw [line width=2, color=cyan, dashed] (3.8, 5)--(1, 5)--(1,7);

        \draw [line width=2, color=cyan, dashed] (10, -3) -- (10, 0)--(9.8, 0);

        \draw[dashed] (3.8,7) -- (3.8, -3);
        \draw[dashed] (9.8, 7) -- (9.8, -3);
		
		\end{tikzpicture}
	\caption{Cyclic rank matrix and matrix braid for the bounded affine permutation $f = [4, 6, 7, 9, 11, 8]$ of Example \ref{ex: running 1}. The dashed lines indicate that we restrict to $1 \leq j \leq 6$, and we have marked in gray the boxes $(i, f(i))$. 
    }
    \label{fig: rank matrix}
	\end{figure}

\noindent The conditions in Definition \ref{def:rankmatrix} imply that the matrix braid $M_k(r)\in\cB_k$ is a $k$-stranded tangle. These braids were introduced in \cite[Section 3.2]{STWZ}. The following result compares the braids $M_k(r)$ to the juggling braids $J_k(f)$.

\begin{thm}
	\label{thm: STZW}
	Let $r=(r_{ij})$ be a cyclic rank matrix of type $(k,n)$, $(i,j)\in\Z^2$, and $f$ its associated bounded affine permutation. The matrix braid $M_k(r)\in\cB_k$ is equivalent to the braid $J_k(f)\Delta_k \in\cB_k$.
\end{thm}

\begin{proof} 
Let us first restrict to the case $i \in \{f^{-1}(1), \dots, f^{-1}(n)\}$ and let us look at the horizontal segment separating the rows $i$ and $i+1$. If $f(i) = i$, then there is no horizontal segment of the braid separating these rows. Else, we have a horizontal segment of the braid connecting $(i, f(i))$ to $(i,i)$ (or to $(i,1)$, if $i \leq 0$), which corresponds to the arc $A_{f(i)}$ in the juggling diagram of $f$, connecting $f(i)$ to $i$. 

Assume now $i \not\in \{f^{-1}(1), \dots, f^{-1}(n)\}$. If $f(i) < 0$, and $i \neq f(i)$, then the horizontal segment of the infinite braid $M_k^{\infty}(r)$ separating the rows $i$ and $i+1$ is contained entirely outside (to the left) of the strip $1 \leq j \leq n$. The same holds if $f(i) > n$ and $i \not\in \{1, \dots, n\}$. It remains to see the case $i \in \{1, \dots, n\}$ and $f(i) > n$. Note that there are exactly $k$ such values of $i$. In this case, we will have a horizontal segment in the braid $M_k(r)$ connecting $(i, n)$ to $(i,i)$. This horizontal segment will intersect exactly once with any strand in the braid that contains a vertical segment separating $j$ and $j+1$ for $j > i$. Thus, up to braid moves pulling the strands to the right, the braid $M_k(r)$ is of the form $J_k(f)\Delta_k$, as needed. 

\end{proof}


\section{Legendrian Links and positroid data}\label{sec:Legendrian}


The goal of this section is to associate Legendrian links to positroid data and show that equivalent positroid data yield Legendrian isotopic links, up to trivially adding unlinked unknots. These Legendrian links are introduced in Subsection \ref{ssec:Leglinks}. Theorem \ref{thm:Leg_allsame} establishes the necessary Legendrian isotopies. It is proven in Subsection \ref{ssec:Leg_isotopic}, after the main technical lemma is proven in Subsection \ref{ssec:MarkovMoves}.

\subsection{Legendrian links associated to positroid data}\label{ssec:Leglinks} In Section \ref{sec:RichardsonJuggling}  we introduced the following braid words associated to positroid data:

\begin{enumerate}
    \item For $u,w\in S_n$ permutations and $w$ $k$-Grassmannian, the positive Richardson braid word 
    $R^{+}_{n}(u, w)=\beta(u^{-1}w_0)\beta(w^*)$,
    in Subsection \ref{ssec:RichardsonBraid}. It is a positive $n$-stranded braid word.\\

    \item For a bounded affine permutation $f$ of size $n$, the juggling braid $J_k(f)=J_k^{(2)}(f)J^{(1)}_k(f)$, in Subsection \ref{ssec:JugglingBraid}. It is a positive $k$-stranded braid word. By Corollary \ref{cor: intervals juggling}, $J_k^{(1)}$ is a product of interval braids, and Proposition \ref{prop: J2}  shows that $J^{(2)}_k(f)$ is a permutation braid.\\

    \item For a Le diagram $\Le$, the Le braid $D_k(\Le)=D_k^{(2)}(\Le)D_k^{(1)}(\Le)$, in Subsection \ref{ssec:LeBraid}. It is a $k$-stranded braid word. By construction, $D_k^{(1)}(\Le)$ is a positive braid word and $D_k^{(2)}(\Le)$ is a permutation braid with all its crossings being negative.\\

    \item For a  cyclic rank matrix $r$, the  matrix braid $M_k(r)$, in Subsection \ref{ssec:MatrixBraids}. It is a $k$-stranded positive braid word.
\end{enumerate}

\noindent Let us use \cite[Section 2.2]{CasalsNg} to associate a Legendrian link in $(\R^{3},\xi_{st})$ to a positive braid word. Recall that the front projection of a Legendrian link in $(\R^{3},\xi_{st})$ with the standard contact form $\xi_{st}=\ker\{dz-ydx\}$ is its projection onto the $(xz)$-plane $\R^2_{x,z}$. The front of a Legendrian link recovers the link, cf.~\cite[Section 3.2]{Geiges08}.

\begin{definition}[\cite{CasalsNg}]\label{defn: -1 closure} Let $\beta$ be a positive braid word. By definition, the $(-1)$-closure of $\beta$ is the Legendrian link  $\Lambda_\beta\sse(\bR^3,\xi_{st})$ whose front projection is drawn in Figure \ref{fig:-1 closure of a positive braid} (left).\hfill$\Box$
\end{definition}

\begin{figure}[H]
    \centering
    \begin{tikzpicture}[baseline=0,scale=0.8]
    \draw [dashed] (0,0.25) rectangle node [] {$\beta$} (4,2.25);
    \foreach \i in {1,3,4}
    {
    \draw [decoration={markings,mark=at position 0.5 with {\arrow{>}}},postaction={decorate}] (4,0.5*\i) to [out=0,in=180] (5,1+0.5*\i) to [out=180,in=0] (4,2+0.5*\i) -- (0,2+0.5*\i) to [out=180,in=0] (-1,1+0.5*\i) to [out=0,in=180] (0,0.5*\i);
    }
    \node at (2,3) [] {\footnotesize{$\vdots$}};
    \node at (-1,2) [] {\footnotesize{$\vdots$}};
    \node at (5,2) [] {\footnotesize{$\vdots$}};
    \end{tikzpicture}\hspace{2cm}
    \begin{tikzpicture}[baseline=-20,scale=0.8]
        \draw [dashed] (0,0.25) rectangle node [] {$\beta$} (4,2.25);
        \foreach \i in {1,3,4}
    {
    \draw [decoration={markings,mark=at position 0.5 with {\arrow{>}}},postaction={decorate}](-1,0.5*\i) -- (0,0.5*\i);
    \draw [decoration={markings,mark=at position 0.5 with {\arrow{>}}},postaction={decorate}](4,0.5*\i) -- (5,0.5*\i);
    }
    \draw [dotted] (-1,0) -- (-1,2.5);
    \draw [dotted] (5,0) -- (5,2.5);
    \node at (-0.5,1) [] {\footnotesize{$\vdots$}};
        \node at (4.5,1) [] {\footnotesize{$\vdots$}};
    \end{tikzpicture}
    \caption{(Left) $(-1)$-closure of a positive braid word $\beta$. (Right) A front in $S^1\times\R$  associated to $\beta$: its satellite along the standard Legendrian unknot yields the $(-1)$-closure on the left.}
    \label{fig:-1 closure of a positive braid}
\end{figure}

\noindent Let $\beta'$ be a positive braid word obtained from $\beta$ by applying braid moves. The Legendrian Reidemeister moves can then be used the show that $\Lambda_{\beta'}$ is Legendrian isotopic to $\Lambda_\beta$, cf.~\cite[Section 2.3]{Etnyre05} and \cite{Kalman}. In particular, the Legendrian isotopy type of $\Lambda_\beta$ is independent of the choice of braid word representing the element inside the positive monoid.

Definition \ref{defn: -1 closure} can now be applied to the positive Richardson and juggling braids, both of which are positive braids. This leads to the following definitions:

\begin{definition}[Richardson Legendrians]\label{def:Richardson_Leg}
Let $u, w \in S_n$ be two permutations such that $u \leq w$ in the Bruhat order and $w$ is $k$-Grassmannian. By definition, the Legendrian Richardson link $\Lambda^+(u, w) \subseteq (\R^{3},\xi_{st})$ is the $(-1)$-closure of $\Delta_{n}R^{+}_{n}(u, w)$.
\hfill$\Box$
\end{definition}

To ease notation, we often denote the Legendrian link $\Lambda^+(u, w) \subseteq (\R^{3},\xi_{st})$ in Definition \ref{def:Richardson_Leg} by $\Lambda(u, w) \subseteq (\R^{3},\xi_{st})$. See Subsection \ref{ssec:Lag_negative} for results making this abuse of notation justified.

\begin{definition}[Juggling Legendrians]\label{def:juggling_Leg}
Let $f:\Z\lr\Z$ be bounded affine permutation of size $n$. By definition, the Legendrian juggling link $\Lambda(f) \subseteq (\R^{3},\xi_{st})$ is the $(-1)$-closure of $\Delta_{k}J_k(f)$.\hfill$\Box$
\end{definition}

\begin{definition}[Matrix Legendrians]\label{def:matrix_Leg}
Let $r$ be a cyclic rank matrix of type $(k,n)$. By definition, the Legendrian matrix link $\Lambda(r) \subseteq (\R^{3},\xi_{st})$ is the $(-1)$-closure of $M_k(r)$.\hfill$\Box$
\end{definition}

\noindent Note that Theorem \ref{thm: STZW} implies that the matrix braid $M_k(r)\in\cB_k$ is directly equivalent to $\Delta_kJ_k(f)\in\cB_k$ and they are both positive and $k$-stranded.

\noindent By construction, the smooth links underlying the Legendrian links $\Lambda(u, w)$ and $\Lambda(f)$ in Definitions \ref{def:Richardson_Leg}, \ref{def:juggling_Leg} and \ref{def:matrix_Leg} above coincide with the homonymous links introduced in Section \ref{sec:RichardsonJuggling}. This justifies using the same notation. Since the Le braid $D_k(\Le)$ is not a positive braid,  Definition \ref{defn: -1 closure} cannot be used directly. An appropriate modification will suffice, as follows.

\begin{definition}[Le Legendrians]\label{def:Le_Leg}
Let $\Le$ be a Le diagram. By definition, the $(-1)$-closure of a positive braid  of the form $\Delta_k \beta(w_0D_k^{(2)}(\Le))D^{(1)}_{k}(\Le)$ is said to be a Legendrian Le link $\Lambda(\Le) \subseteq (\R^{3},\xi_{st})$.
\hfill$\Box$
\end{definition}

\noindent First, $\beta(w_0D_{k}^{(2)}(\Le))$ denotes a positive braid lift of  the permutation $w_0D_{k}^{(2)}(\Le)$, e.g.~ as in Definition \ref{def:Richardson}. Second, any two different choices of such lifts lead to equivalent braid words and it follows that the Legendrian links they define are Legendrian isotopic. Thus our use of the notation $\Lambda(\Le)$ for any such Legendrian links, which does not explicitly refer to the choice of lift. By construction, the smooth link underlying any such $\Lambda(\Le)$ in Definitions \ref{def:Le_Leg} coincides with the homonymous link introduced in Section \ref{sec:RichardsonJuggling}.

The main goal of the rest of this section is to prove the following result:

\begin{thm}\label{thm:Leg_allsame} Let $u, w \in S_n$ be two permutations, with $w$  $k$-Grassmannian, $f:\Z\lr\Z$ be bounded affine permutation of size $n$ and $\Le$ a Le diagram  all representing the same positroid data. Then the Legendrian links $\Lambda(u, w),\Lambda(f),\Lambda(\Le)$ and $\La(r)$ are all Legendrian isotopic in $(\R^{3},\xi_{st})$, up to adding unlinked max-tb Legendrian unknots.
\end{thm}

\noindent Theorem \ref{thm:Leg_allsame} is proven in Subsection \ref{ssec:Leg_isotopic} below, once we have established Lemma \ref{lem:Markov}. In fact, the proof shows that $\Lambda(f),\Lambda(\Le)\sse(\R^{3},\xi_{st})$ are Legendrian isotopic, without adding any max-tb Legendrian unknots. Here max-tb stands for maximal Thurston-Bennequin number, cf.~ \cite[Section 2.6]{Etnyre05} for details on the Thurston-Bennequin invariant.


\subsection{Variations on the Markov move}\label{ssec:MarkovMoves} The Legendrian links that we are comparing are associated to braids with different number of strands. For instance, the Richardson braids are $n$-stranded and the juggling braids are $k$-stranded for a positroid in $\Gr(k,n)$. In order to show that these Legendrian links are Legendrian isotopic, we must therefore apply a type of (de)stabilization. This is also the reason to work in $(\R^{3},\xi_{st})$ instead of the 1-jet space $(J^1S^1,\xi_{st})$. The following destabilization lemma,  which implies Lemma \ref{lem:Markovsmooth}, will suffice for our purposes:

\begin{center}
	\begin{figure}[h!]
		\centering
		\includegraphics[width=\textwidth]{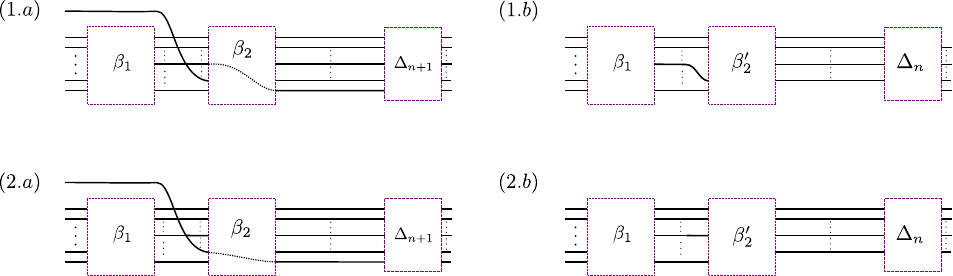}
		\caption{Lemma \ref{lem:Markov} states that two fronts $(1.a)$ and $(1.b)$, resp.~$(2.a)$ and $(2.b)$ have front homotopic $(-1)$-closures, resp.~front homotopic $(-1)$-closures up to an unlinked max-tb unknotted component. Note that these fronts, if to be understood cyclically in $S^1_\theta\times\R_z$, do {\it not} yield Legendrian (or even smoothly) isotopic links. This forces considering $(-1)$-closures so as to find the Legendrian isotopy in $(\R^3,\xi_{st})$.
  }
		\label{fig:Ingredients_FrontsMarkov_Statement}
	\end{figure}
\end{center}

\begin{lemma}[Legendrian destabilizations]\label{lem:Markov} Let $\beta_1\in\Br_n$ be a positive braid on $n$ strands in the Artin generators $\sigma_1,\ldots,\sigma_{n-1}$ and $\beta_2\in\Br_{n+1}$ a reduced positive lift of a permutation $w\in S_{n+1}$ on $(n+1)$ strands. Set $j:=w^{-1}(n+1)$, let $\tilde{\beta}_1 \in \Br_{n+1}$ the braid obtained from $\beta_1$ 
by shifting all the indices of the Artin generators up by $1$, and let $\beta'_2 \in \Br_n$ be the braid obtained from $\beta_2$ by removing the strand that ends at the bottom on the right of $\beta_2$. Then the following two statements hold:


\begin{itemize}
    \item[$(1)$] Consider the following two positive braid words:\\
    
    \begin{itemize}
        \item[(a)] $\eta_1 := \Delta_{n+1}\beta_2 (\sigma_{\ell}\cdot\ldots\cdot\sigma_{j+1}\sigma_j\cdot\ldots\cdot\sigma_2\sigma_1)\tilde{\beta}_1 \in \Br_{n+1}$, the $(n+1)$-stranded positive braid depicted in Figure \ref{fig:Ingredients_FrontsMarkov_Statement}.$(1.a)$,\\

        \item[(b)] $\eta_2:=\Delta_n\beta'_2(\sigma_{\ell - 1}\cdot\ldots\cdot\sigma_{j+1})\beta_1 \in \Br_n$, the $n$-stranded positive braid depicted in Figure \ref{fig:Ingredients_FrontsMarkov_Statement}.$(1.b)$.
        \\
    \end{itemize} 

\noindent Then the Legendrian link $\La_{\eta_1}$ is Legendrian isotopic to $\La_{\eta_2}$.\\

     \item[$(2)$] Consider the following two positive braid words:\\
    
    \begin{itemize}
        \item[(a)] $\eta_1:=\Delta_{n+1}\beta_2(\sigma_{j}\sigma_{j-1}\cdot\ldots\cdot\sigma_1)\tilde{\beta}_1 \in \Br_{n+1}$, the $(n+1)$-stranded positive braid depicted in Figure \ref{fig:Ingredients_FrontsMarkov_Statement}.$(2.a)$,\\

        \item[(b)] $\eta_2:=\Delta_n\beta'_2\beta_1\in \Br_n$, the $n$-stranded positive braid depicted in Figure \ref{fig:Ingredients_FrontsMarkov_Statement}.$(2.b)$.
        \\

    \end{itemize}
\noindent Then the Legendrian link $\La_{\eta_1}$ is Legendrian isotopic to the Legendrian link given by a max-tb Legendrian unknot unlinked union with $\La_{\eta_2}$.\\
\end{itemize}
\end{lemma}

\begin{center}
	\begin{figure}[h!]
		\centering
		\includegraphics[width=\textwidth]{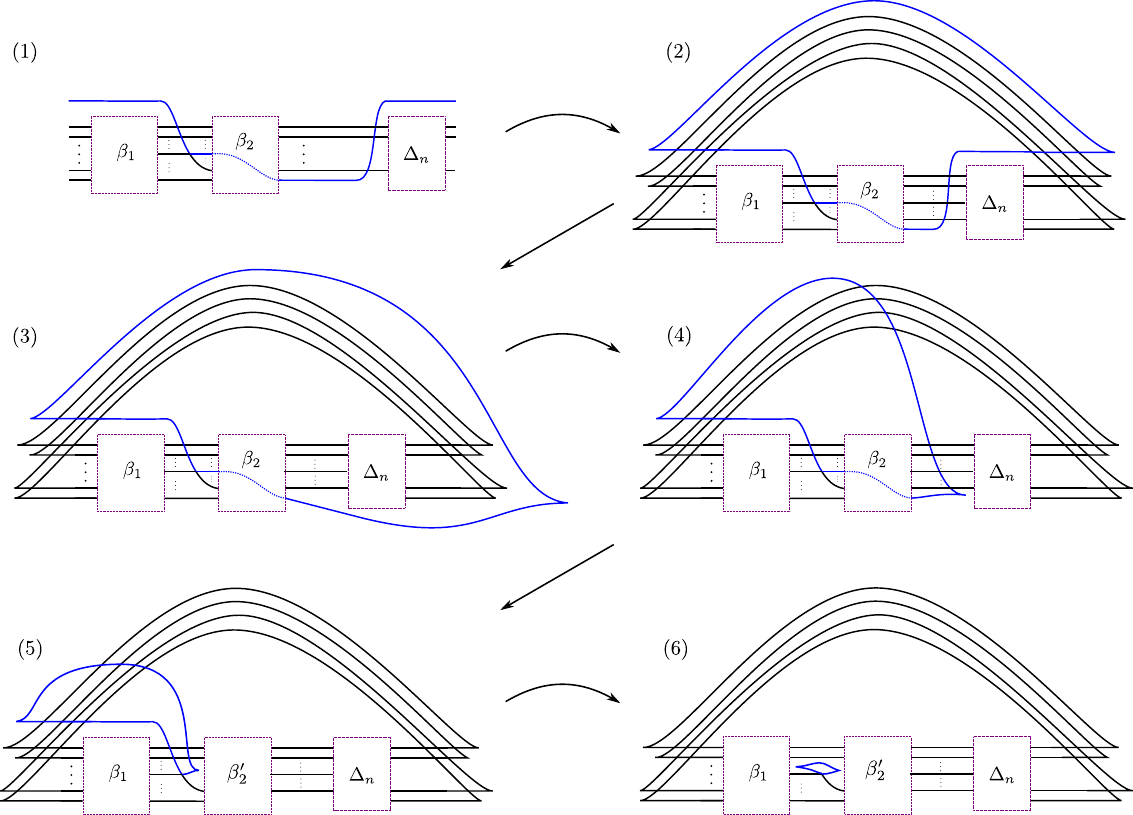}
		\caption{In (1) the starting piece of a front. From (2)-(6), the Legendrian isotopy that proves Lemma \ref{lem:Markov} for its $(-1)$-closure.}
		\label{fig:Ingredients_FrontsMarkov2_General}
	\end{figure}
\end{center}

\begin{proof}
The proof is contained in Figure \ref{fig:Ingredients_FrontsMarkov2_General}, for item $(1)$, and \ref{fig:Ingredients_FrontsMarkov2_General_Case2}, for item $(2)$, which we now explain. First item $(1)$. Choose a braid word for the half-twist $\Delta_{n+1}$ which is of the form $\Delta_n(\sigma_n\sigma_{n-1}\ldots \sigma_2\sigma_1)$. The resulting braid is then as in Figure \ref{fig:Ingredients_FrontsMarkov2_General}.(1) and its $(-1)$-closure is depicted in Figure \ref{fig:Ingredients_FrontsMarkov2_General}.(2). 

The fronts in Figure \ref{fig:Ingredients_FrontsMarkov2_General}.(2)-(6) are then front homotopic, i.e.~realized by Legendrian isotopies, as follows. Front (2) to (3) consists of a series of Reidemeister III moves followed by a sequence of Reidemeister II moves that pull the top (blue) strand to the right of the front past $\Delta_n$ and $n$ right cusps. Front (3) to (4) is the reverse of that sequence applied to the piece of the (blue) strand above the rightmost (blue) cusp: first a sequence of Reidemeister II moves and then a sequence of Reidemeister III moves. Front (4) to (5) is a sequence of Reiedemeister III and Reidemeister II moves that pull the right blue cusp pointing down up to the center region. The sequence of Reidemeister III moves necessary to realize this Legendrian isotopy from (4) to (5) exists because $\beta_2$ is (the lift of) a permutation braid. In particular, the blue strand inside the $\beta_2$-box in Front (4) always goes above the other strands. Then the isotopy from (4) to (5) is as follows: first pull the blue strand exiting the right blue cusp from above across part of the $\beta_2$-box via Reidemeister III moves; then use Reidemeister II moves to pull the right blue cusp through the $\beta_2$-box so as to reach Front (5). Finally, Front (5) to (6) is a sequence of Reidemeister II moves that pulls the top leftmost blue cusp to the region containing the right blue cusp. Front (6) is homotopic to the $(-1)$-closure of $\eta_2$ via a Reidemeister I move applied to the crossing between $\beta_1$ and $\beta_2$, which proves item (1) of the lemma.

For item (2), proceed as in item (1) by choosing a braid word for the half-twist $\Delta_{n+1}$ of the form $$\Delta_n(\sigma_n\sigma_{n-1}\cdot \ldots \cdot \sigma_2\sigma_1).$$
The resulting braid is drawn in Figure \ref{fig:Ingredients_FrontsMarkov2_General_Case2}.(1) and its $(-1)$-closure is in Figure \ref{fig:Ingredients_FrontsMarkov2_General_Case2}.(2). From Front (2) to Front (3) we perform a sequence of Reidemeister III and then Reidemeister II moves moving the right piece of the blue strand (under the cusps). This is the same first step as for item (1). From Front (3) to Front (4) we perform a movie similar to the previous one but to the left piece, using the left cusps. The only difference is that we must pull the strand through a piece of the $\beta_2$-box, as indicated by the red arrows in Figure \ref{fig:Ingredients_FrontsMarkov2_General_Case2}.(3). This is indeed possible because $\beta_2$ is a permutation braid. Once at Front (4), the blue component is a max-tb Legendrian unknot which is unlinked from the other components of the Legendrian link, thus item (2) follows.

\end{proof}

\begin{center}
	\begin{figure}[h!]
		\centering
		\includegraphics[width=\textwidth]{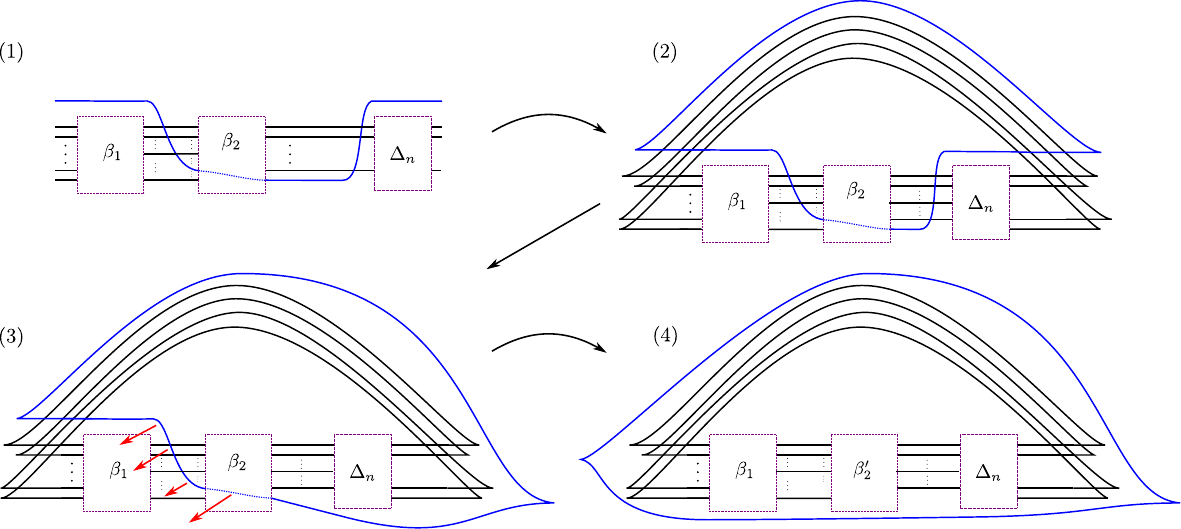}
		\caption{In (1) the starting piece of a front. From (2)-(4), the Legendrian isotopy that proves Lemma \ref{lem:Markov} for its $(-1)$-closure.}
		\label{fig:Ingredients_FrontsMarkov2_General_Case2}
	\end{figure}
\end{center}


\begin{remark}
Lemma \ref{lem:Markov} does not hold for general positive braid words $\beta_2\in\Br_{n+1}$. 
Fortunately, the hypothesis of $\beta_2$ being a permutation braid is met for the positroid Richardson braids.\hfill$\Box$
\end{remark}


\subsection{Proof of Theorem \ref{thm:Leg_allsame}}\label{ssec:Leg_isotopic} Let us conclude Theorem \ref{thm:Leg_allsame} from the results established so far. Let us first show that $\Lambda(u, w)$ and $\Lambda(f)$ are Legendrian isotopic in $(\R^{3},\xi_{st})$, up to adding unlinked max-tb Legendrian unknots. For that we follow the proof of Theorem \ref{thm:Richardson_to_juggling} in Subsection \ref{ssec:destabilization}. Thanks to Lemma \ref{lem:Markov}, we claim that the same argument gives a Legendrian isotopy, not just a smooth one. Indeed,  it suffices to note the following two facts:
\begin{itemize}
    \item[(i)] All braids being used in the proof of Theorem \ref{thm:Richardson_to_juggling} are positive braids.  Therefore, at any stage, we can consider the $(-1)$-closure and have the argument be about Legendrian links.\\

    \item[(ii)] The only result that is used in the proof of Theorem \ref{thm:Richardson_to_juggling} is Lemma \ref{lem:Markovsmooth}. Since Lemma \ref{lem:Markov} is precisely a Legendrian realization of Lemma \ref{lem:Markovsmooth}, we can apply the argument in the proof of Theorem \ref{thm:Richardson_to_juggling} to Legendrian links and use instead Lemma \ref{lem:Markovsmooth} each time that Lemma \ref{lem:Markovsmooth} is invoked in that smooth proof.\\
\end{itemize}
\noindent This concludes the desired statement about $\Lambda(u, w)$ and $\Lambda(f)$. The Legendrian links $\Lambda(f)$ and $\Lambda(r)$ are Legendrian isotopic because they are $(-1)$-closures of equivalent positive $k$-stranded braids, thanks to Theorem~\ref{thm: STZW}.

Finally, let us now show that $\Lambda(f)$ and $\Lambda(\Le)$ are Legendrian isotopic. Sections \ref{ssec:JugglingBraid} and \ref{ssec:LeBraid}  provided the decompositions $J_k(f)=J^{(2)}_k(f)J_k^{(1)}(f)$ and $D_k(\Le)=D_k^{(2)}(\Le)D_k^{(1)}(\Le)$. Here, $J_k^{(1)}(f)$ and $D_k^{(1)}(\Le)$ are $k$-stranded positive braids and $J_k^{(2)}(f)$ and $D_k^{(2)}(\Le)$ are reduced, but the latter has negative crossings. We claim that the two positive braid words $\Delta_kJ^{(2)}_k(f)J_k^{(1)}(f)$ and $\Delta_{k}\beta(w_0D_{k}^{(2)}(\Le))D_k^{(1)}(\Le)$ are equivalent through positive braid words. For ease of notation, we denote $J_i:=J_k^{(i)}(f)$ and $D_i:=D_k^{(i)}(\Le)$, $i=1,2$, and write $\beta_1\stackrel{+}{\sim}\beta_2$ if two positive braid words are equivalent through positive braid words.

Let us prove the claim. First, Lemma \ref{lem: D1 vs J1}  and its proof imply that $J_1\stackrel{+}{\sim}D_1$. Second, we now need $J_2\stackrel{+}{\sim}\beta(w_0D_2)$. Lemma \ref{lem: D2 vs J2} and its proof show that $J_2D_2^{-1}\stackrel{+}{\sim}\Delta_k$, where $D_2^{-1}$ is a positive braid word. By construction, $\Delta_k\stackrel{+}{\sim}\beta(w_0D_2)D_2^{-1}$ and thus $J_2D_2^{-1}\stackrel{+}{\sim}\beta(w_0D_2)D_2^{-1}$. By ~\cite[Prop.~II.4.7]{Dehornoy00}, the positive braid monoid is right-cancellative and thus $J_2D_2^{-1}\stackrel{+}{\sim}\beta(w_0D_2)D_2^{-1}$ implies $J_2\stackrel{+}{\sim}\beta(w_0D_2)$. Therefore, we have $J_1\stackrel{+}{\sim}D_1$ and $J_2\stackrel{+}{\sim}\beta(w_0D_2)$. Combined, these imply $\Delta_kJ_2J_1\stackrel{+}{\sim}\Delta_k\beta(w_0D_2)D_1$. Hence, their $(-1)$-closures $\Lambda(f)$ and $\Lambda(\Le)$ are Legendrian isotopic.\hfill$\Box$

\subsection{Lagrangian projections and negative crossings}\label{ssec:Lag_negative} This subsection is not logically needed for the rest of the results in this manuscript, but we include this brief discussion for completeness. Both in Section \ref{sec:RichardsonJuggling} above and the literature, see e.g.~\cite[Section 3]{GL}, braids with negative crossings are discussed in relation to positroids. In the former instance, both the Richardson braid word, in Definition \ref{def:Richardson} in Subsection \ref{ssec:RichardsonBraid}, and the Le braid $D_k(\Le)$, in Subsection \ref{ssec:LeBraid}, typically have negative crossings. This raises the question of whether there are Legendrian links in $(\R^{3},\xi_{st})$ that can be naturally associated to them. In particular, this would allow us to describe certain types of varieties and their algebraic combinatorics, associated to such braid words (with some negative crossings), intrinsically in terms of Legendrian links and their invariants, cf.~\cite{CGGS,CasalsNg}.

\noindent The answer is affirmative in this context, as the Lagrangian projection allows for the introduction of negative crossings in the following case. It is possible to introduce a cancelling pair $vv^{-1}$ where $v$ is a reduced positive braid word for a permutation braid. This technique is known as dipping in the literature. For dipping, also known as splashing in \cite[Section 3.2]{Dip1}, we refer to \cite[Section 2.4]{Dip2}. We refer to \cite[Section 2.2]{CasalsNg} for the definition and details on the Lagrangian $(-1)$-closure of a braid word.  In conclusion,  we can use the following lemma to associate Legendrian links to both the Richardson and Le braids, which are described by braid words with some negative crossings.

\begin{lemma}\label{lem:neg_cross} Let $\gamma\in\Br_n^+$ be a positive braid word and $v\in S_n$ a permutation. Then:
\begin{enumerate}
    \item There exists a Legendrian link $\La(v,\gamma)\sse (\R^{3},\xi_{st})$ whose Lagrangian projection is Hamiltonian isotopic to the Lagrangian $(-1)$-closure of the braid word $\Delta_n^{2}\beta(v)^{-1}\gamma$, for a choice of positive braid word $\Delta_n$ for the half-twist.\\

    \item $\La(v,\gamma)$ is Legendrian isotopic to the Legendrian $(-1)$-closure of $\Delta_n\beta(w_0v^{-1})\gamma$.
\end{enumerate}
\end{lemma}

\begin{proof}
By \cite[Prop.~2.7]{CasalsNg} the positive braid word $\Delta_n\beta(w_0v^{-1})\gamma$ is admissible, as defined in \emph{loc.~cit.}. In particular, there exists a Legendrian link $\La'(v,\gamma)\sse (\R^{3},\xi_{st})$ whose Lagrangian projection is the Lagrangian $(-1)$-closure of the braid word $\Delta_n\beta(w_0v^{-1})\gamma$. Note that \emph{ibid.}~also implies that $\La'(v,\gamma)$ is Legendrian isotopic to the Legendrian $(-1)$-closure of $\Delta_n\beta(w_0v^{-1})\gamma$. By dipping according to the permutation $v$, introducing $\beta(v)\beta(v^{-1})$ to the Lagrangian diagram exactly between $\beta(w_0v^{-1})$ and $\gamma$, there exists a Legendrian isotopy from $\La'(v,\gamma)$ to a Legendrian link $\La(v,\gamma)$ whose Lagrangian projection is the Lagrangian $(-1)$-closure of the braid word $\Delta_n\beta(w_0v^{-1})\beta(v)\beta(v^{-1})\gamma$. This braid is a positive braid word for $\Delta_n^{2}\beta(v)^{-1}\gamma$. Note that we can dip according to $\beta(v)^{-1}\beta(v)$, instead of $\beta(v)\beta(v)^{-1}$. Indeed, this is because the differences in height between the strands in the front diagram provided by \cite[Prop.~2.7]{CasalsNg}, away from a neighborhood of the crossings, are strictly increasing as we move to the right, instead of left, as in the Ng resolution \cite{Ng03}.\end{proof}



\noindent Lemma \ref{lem:neg_cross} can be applied to Richardson and Le braids, as follows:
\begin{itemize}
    \item[(a)] $\gamma=w$ and $v=u$, with $u,w\in S_n$, $w$ a $k$-Grassmannian permutation, and $u
    \leq w$ in the Bruhat order. Then Lemma \ref{lem:neg_cross} implies that $\La(u,w)$ is a Legendrian representative for the Richardson braids in Definition \ref{def:Richardson}, even if they contained negative crossings. Lemma \ref{lem:neg_cross}.(2) implies that $\La(u,w)$ is Legendrian isotopic to $\La^+(u,w)$ as introduced in Definition \ref{def:Richardson_Leg}.\\

    \item[(b)] For a Le diagram $\Le$, we can choose $\gamma=D_k^{(1)}(\Le)$ and $v=(D_k^{(2)}(\Le))^{-1}$. Then Lemma \ref{lem:neg_cross} constructs a Legendrian link $\La((D_k^{(2)}(\Le))^{-1},D_k^{(1)}(\Le))$ for the Le braid $D_k^{(2)}(\Le)D_k^{(1)}(\Le)$, even if it has negative crossings. By Lemma \ref{lem:neg_cross}.(2), $\La((D_k^{(2)}(\Le))^{-1},D_k^{(1)}(\Le))$ is Legendrian isotopic to $\La(\Le)$ as introduced in Definition \ref{def:Le_Leg}.
\end{itemize}

\noindent In particular, the above gives a symplectic geometric enhancement of the smooth links presented in \cite[Section 3]{GL}: they can be naturally defined as Legendrian links, from which positroid strata can be extracted by studying the Legendrian contact dg-algebra, see e.g.~\cite{CGGS,CasalsNg,CasalsWeng22}.



\section{Braid Varieties, Richardson varieties and Brick Manifolds}\label{sec:brick} In this section, we study braid varieties associated to positroid braids. In particular, we prove Theorems \ref{thm:rich vs juggling intro} and \ref{thm: intro brick}. In Subsection \ref{ssec:Richardson_braidvar}, we show that open Richardson varieties in type $A$, i.e.~$G=\GL(n,\C)$, are braid varieties. Subsection \ref{ssec:compactify} shows that brick manifolds $\brick(\beta)$, for any choice of braid word $\beta\in\cB$, provide {\it different} smooth projective compactifications of the {\it same} braid variety $X(\beta)$. Finally, Subsection \ref{ssec:equiv_hom} uses Subsection \ref{ssec:compactify} to present a description of the homology of braid varieties in terms of brick manifolds and also establishes the relation to Khovanov-Rozansky homology.



\subsection{Preliminaries on braid varieties}\label{ssec:braid_prelim} Braid varieties are defined as follows:

\begin{definition}[\cite{CGGS}]\label{def:braid_var}
Let $\beta$ be a positive braid word $\beta\in \cB^+_n$, $\beta=\sigma_{i_1}\cdots \sigma_{i_\ell}$, and $\pi\in \GL(n,\bC)$ a permutation matrix. By definition, the {\em braid variety} associated to $\beta$ and $\pi$ is
$$
X(\beta;\pi):=\left\{(z_1,\ldots,z_\ell)\ :\ B_\beta(z_1,\ldots,z_{\ell})\pi\ \text{is upper-triangular}\right\}\sse \bC^{\ell},
$$
where the matrix $B_{\beta}(z_1,\ldots,z_{\ell})\in \GL(n,\bC[z_1,\ldots,z_\ell])$ is defined to be the matrix product
$$
B_{\beta}(z_1,\ldots,z_{\ell}):=B_{i_1}(z_1)\cdots B_{i_\ell}(z_{\ell}),
$$
and the braid matrices $B_i(z)\in \GL(n,\bC[z])$ are defined by:
\[
(B_i(z))_{jk} := \begin{cases} 1 & j=k \text{ and } j\neq i,i+1 \\
1 & (j,k) = (i,i+1) \text{ or } (i+1,i) \\
z & j=k=i+1 \\
0 & \text{otherwise;}
\end{cases},\quad\mbox{i.e. }
B_i(z):=\left(\begin{matrix}
1 & \cdots  & & & \cdots & 0\\
\vdots & \ddots & & & & \vdots\\
0 & \cdots & 0 & 1 & \cdots & 0\\
0 & \cdots & 1 & z & \cdots & 0\\
\vdots &  & & &\ddots & \vdots\\
0 & \cdots & & & \cdots & 1\\
\end{matrix}\right).
\]
Here the only non-trivial $(2\times 2)$-block of $B_i(z)$ is at the $i$th and $(i+1)$st rows.\hfill$\Box$
\end{definition}

Braid varieties were introduced in \cite{CGGS,Mellit}, where part of their geometry was studied. In particular, we proved in \cite{CGGS} that $X(\beta_1;\pi)\cong X(\beta_2;\pi)$ if 
$\beta_1$ and $\beta_2$ are related by Reidemeister III moves or braid commutation; hence the name {\it braid} varieties. In this article, the permutation (matrix) $\pi$ will often be $\pi=w_{0,n}=[n, n-1, \ldots, 1]\in S_n$, and we sometimes abbreviate $X(\beta)$ for $X(\beta;w_{0,n})$.


\subsection{Richardson and Positroid varieties} We introduce Richardson and positroid varieties, following \cite{KLS}, see also \cite{SpeyerSurvey}. Let us consider the flag variety $\Fl_{n}$ associated to the group $G := \GL(n,\bC)$. That is, $\Fl_{n} = G/B$ where $B$ is the Borel subgroup of upper triangular matrices. We denote by $\cF^{A}$ the flag associated to a matrix $A \in G$. Namely, the $i$-th space in $\C^n$ for the flag $\cF^{A}$ is spanned by the first $i$ columns of the matrix $A$. Let us denote by $\cF^{\std} := (0 \subseteq \langle e_{1} \rangle \subseteq \langle e_{1}, e_{2}\rangle \subseteq \cdots \subseteq \langle e_{1}, \dots, e_{n-1}) \subseteq \bC^{n})$ the standard flag, and by $\cF^{\ant} := (0 \subseteq \langle e_{n}\rangle \subseteq \langle e_{n}, e_{n-1}\rangle \subseteq \cdots \subseteq \langle e_{n}, \dots, e_{2} \rangle \subseteq \bC^n)$ the anti-standard flag. Note that we have an action of $G$ on $G/B = \Fl_n$ by left multiplication. By restricting, we also have an action of $B$ on $\Fl_n$. Similarly, if $B_{-}$ is the opposite Borel subgroup of lower triangular matrices, we also have an action of $B_{-}$ on $\Fl_n$ by left multiplication.

Let $w \in S_n$ be a permutation. The \emph{Schubert cell} $X^{\circ}_{w}$ associated to $w\in S_n$ is defined to be the $B$-orbit of the flag $wB/B$, where $w$ is understood as a permutation matrix:
$$
\schub_{w} := BwB/B =  \{\cF \in \Fl_{n} \mid \dim(F^{\std}_{p} \cap F_{q}) = \#\{i \leq q \mid w(i) \leq p\} \; \text{for all} \; p, q\in[1,n]\}.
$$

The \emph{opposite Schubert cell} of $w \in S_n$ is defined to be the $B_{-}$-orbit of the flag $wB/B$:
\[
\schub\!^{w} := B_{-}wB/B.
\]
Similarly to $\schub_{w}$, the opposite Schubert cell $\schub\!^{w}$ admits a explicit description as follows. Let $w_{0}$ be the longest element of $S_{n}$. Then:
$$
\schub\!^{w} = \{\cF \in \Fl_{n} \mid \dim(F_{p}\cap F_{q}^{\ant}) = \#\{i \leq q \mid w_{0}w(i) \leq p\} \; \text{for all} \; p, q = 1, \dots, n\}.
$$

\begin{definition}
Let $u, w \in S_n$ be two permutations. The \emph{open Richardson variety} $\rich{w}{u} \subseteq \Fl_{n}$ is the intersection:
\[
\rich{w}{u} := \schub_{w} \cap \schub\!^{u}.
\]
\end{definition}

\noindent It is known that $\rich{w}{u}$ is nonempty if and only if $u \leq w$ in the Bruhat order, cf.~\cite{SpeyerSurvey}, in which case it is a smooth affine variety of dimension $\ell(w) - \ell(u)$. In particular, $\rich{w}{w}$ is a single point.

Now fix $k \leq n$ and consider the Grassmannian $\Gr(k,n)$ of $k$-planes in $\C^{n}$. In fact, $\Gr(k,n) = G/P$, where $P$ is the subgroup of $G = \GL(n,\C)$ consisting of block-upper-triangular matrices with blocks of sizes $k$ and $n-k$. We have the associated subgroup $W_P = S_k \times S_{n-k} \subseteq S_n$. Again, we have an action of $B$ on $\Gr(n,k)$ and, given a permutation $w \in S_n$, we have an associated Schubert cell in the Grassmannian:
\[
C_{w} := BwP/P \subseteq \Gr(k,n).
\]
In fact, $C_{w} = C_{w'}$ if the cosets $wW_{P}$ and $w'W_{P}$ coincide.

Note that we have a natural projection $\pi: \Fl_n \to \Gr(k,n)$. In terms of flags, $\pi(\cF) = \cF_{k}$ is the $k$-th subspace of the flag. This map is $B$-equivariant and thus $\pi(\schub_{w}) \subseteq C_{w}$. The following result underscores the importance of $k$-Grassmannian permutations. 

\begin{lemma}[Section 5, \cite{KLS}]\label{lem: projection schubert}
The map $\pi|_{\schub_{w}}: \,\, \schub_{w} \to C_w$ is an isomorphism of algebraic varieties if and only if $w$ is a $k$-Grassmannian permutation. 
\end{lemma}

In particular, if $w$ is a $k$-Grassmannian permutation and $u \leq w$ then we have:
\[
\Fl_{n} \supseteq \rich{w}{u} \cong \pi(\rich{w}{u}) \subseteq \Gr(k,n).
\]

\begin{definition}
    Let $u, w \in S_n$, with $w$ a $k$-Grassmannian permutation and $u \leq w$. The \emph{open positroid variety} $\Pi_{u, w} \subseteq \Gr(k,n)$ is:
    \[
    \Pi_{u, w} := \pi(\rich{w}{u}). 
    \]
\end{definition}

By definition, the positroid variety $\Pi_{u, w}$ is isomorphic to the Richardson variety $\rich{w}{u}$. Since positroids can be defined inside the Grassmannian, they are specially nice amongst other Richardson varieties, see e.g. \cite{GL, KLS, KLS2, SpeyerSurvey}.

\subsection{Open Richardson varieties as braid varieties}\label{ssec:Richardson_braidvar}

The following result is the main theorem of this subsection. It states that open Richardson varieties can be described as braid varieties.

\begin{thm}\label{thm:richardson}
	Let $\beta_{1},\beta_{2}\in\cB^+_n$ be two reduced positive braid words, and $w_1,w_2\in S_n$ be their Coxeter projections. 

 \begin{itemize}
	\item[(a)] The map $$\iota:\bC^{\ell(\beta_1)}\times\bC^{\ell(\beta_2)}\lr\Fl_n,\qquad (z_1,z_2) \mapsto \cF^{B_{\beta_1}^{-1}(z_1)},\quad (z_1,z_2)\in\bC^{\ell(\beta_1)}\times\bC^{\ell(\beta_2)},$$
	
	\noindent restricts to an isomorphism
	$$X(\beta_1\beta_2; w_0) \lr \iota(X(\beta_1\beta_2; w_0))\cong\schub_{w_{1}^{-1}} \cap \schub\!^{w_2w_0}$$
	of affine algebraic varieties.\\
 
 \item[(b)] The map
 \[
 \jmath: \bC^{\ell(\beta_1)} \times \bC^{\ell(\beta_2)} \lr \Fl_n, \qquad (z_1, z_2) \mapsto \cF^{w_0B_{\beta_1}^{-1}(z_1)}, \quad (z_1, z_2) \in \bC^{\ell(\beta_1)} \times \bC^{\ell(\beta_2)}
 \]
 restricts to an isomorphism
  \[
 X(\beta_1\beta_2; w_0) \lr \jmath(X(\beta_1\beta_2; w_0)) \cong \schub\!^{w_0w_1^{-1}}\cap \schub_{w_2^*}
 \]
 of affine algebraic varieties. 
 \end{itemize}
\end{thm}

\begin{cor} \label{cor:richardson}
	Let $u, w \in S_{n}$ be such that $u \leq w$ in Bruhat order, and $\beta(w),\beta(u^{-1}w_0), \beta(w^*)\in\Br_n$ positive lifts of $w,u^{-1}w_0, w^*$. Then restrictions give the following isomorphisms of affine algebraic varieties
 \[
	     \iota: X(\beta(w)\beta(u^{-1}w_0); w_{0}) \overset\sim\to \; \rich{w^{-1}}{u^{-1}}, \qquad 
 \jmath: X(\beta(u^{-1}w_0)\beta(w^*); w_0) \overset\sim\to \; \rich{w}{u}.
	\]
\end{cor}

\noindent Theorem \ref{thm:richardson}, through Corollary \ref{cor:richardson}, implies Theorem \ref{thm:rich vs juggling intro}, as discussed in Subsection \ref{sect:rich vs juggling} below. Subsections \ref{sssec:lemmas_rich} and \ref{sssec:thm_rich} now prove Theorem \ref{thm:richardson}.

\begin{remark}
\label{rem:different isomorphisms}
We note that interpreting open Richardson varieties as braid varieties allows for a clear description of certain isomorphisms between them.
In particular, 
the composition $\iota \circ \jmath^{-1}$ associated with $\beta_1 = \beta(u^{-1}w_0), \beta_2 = \beta(w^*)$ gives an isomorphism
\[
\rich{w}{u} \overset\sim\to \rich{w_0 u}{w_0 w}.
\]
Also, we have an explicit isomorphism 
\[
\psi: X(\beta(w)\beta(u^{-1}w_0); w_{0}) \cong X(\beta(u^{-1}w_0)\beta(w^*); w_0)
\]
given by a ``cyclic rotation'' of the braid word. The description of $\psi$ and the proof that it is an isomorphism can be found in \cite[Lemma 3.10]{cgglss}. \footnote{The paper \cite{cgglss} uses different conventions to define braid varieties, but the same arguments apply in our setting.} 
The composition 
$\jmath \circ \psi \circ \iota^{-1}$, with $\iota, \jmath$ as in Corollary \ref{cor:richardson}, gives an isomorphism
\[
\rich{w^{-1}}{u^{-1}} \overset\sim\to \rich{w}{u}.
\]

In fact, the existence of these isomorphisms implies that we could have chosen slightly different braids throughout the paper so that their varieties would be isomorphic to open Richardson and positroid varieties, at the cost of changing such isomorphisms. Namely, our positive Richardson braid is $R^+_n(u, w) = \beta(u^{-1}w_0)\beta(w^*)$, and we use the isomorphism $\jmath$ from Corollary \ref{cor:richardson} to realize open Richardson varieties as braid varieties. Instead, we could have used the braid $\beta(w)\beta(u^{-1}w_0)$ and its braid variety throughout the paper and applied the isomorphism $\jmath \circ \psi$ to realize open Richardson varieties as braid varieties. All the other positroid braids would have changed accordingly.\hfill$\Box$
\end{remark}

\subsubsection{Technical lemmas for Theorem \ref{thm:richardson}}\label{sssec:lemmas_rich} The open Schubert cell $\schub_{w}$ is isomorphic to an affine space of dimension $\ell(w)$, the length of $w$, which can be described as follows. Consider $w$ as a permutation matrix and let $E_{lm}$ be the $(lm)$-elementary matrix, so that the $(s,t)$ entry of $E_{lm}$ is the product $(\delta_{l,s})(\delta_{m,t})$ of Kronecker deltas. Inside the space of $(n\times n)$-matrices $M_n(\bC)\cong\bC^{n^2}$, consider the unique affine subspace $M_w$ which contains $w\in M_n(\bC)$ and is spanned by all the matrices $E_{w(j),i}$ such that the pair $(i, j)$ satisfies $(i, j) \in \inv(w)$.\footnote{Recall that an \emph{inversion} of $w$ is a pair $(i, j)$ where $i < j$ and $w(i) > w(j)$, and $\inv(w)$ denotes the set of inversions of $w$; note that $\ell(w) = \#\inv(w)$.} It can be proven, see e.g. \cite[Proposition 1.12]{SpeyerSurvey}, that the map
$$\iota:\GL(n,\bC)\lr \Fl_n,\quad A \mapsto \cF^{A},$$
restricts to an isomorphism between $M_{w}$ and the Schubert cell $\schub_{w}$, i.e. $\iota(M_{w})\cong\schub_{w}$. Let us now relate this to braid matrices. Indeed, braid matrices serve as a parametrization of the affine spaces $M_{w}$, and thus of the corresponding Schubert cells. This is the content of the following lemma, for a proof see e.g. \cite[Proposition 5.1.5]{Mellit} and \cite[Section 2]{CGGS}.

\begin{lemma}\label{lemma:cells1}
	Let $w \in S_{n}$ be a permutation and $\beta = \sigma_{i_1}\cdots \sigma_{i_{\ell}}\in\cB_n$ a choice of reduced positive lift for $w$. Then the map
	$$\bC^{\ell(w)} \to M_{w}, \qquad (z_1, \dots, z_{\ell}) \mapsto B_{i_1}^{-1}(z_1)\cdots B_{i_{\ell}}^{-1}(z_{\ell})
	$$
	\noindent is an isomorphism of affine algebraic varieties.\hfill$\Box$
\end{lemma}

\noindent The opposite open Schubert cell $\schub\!^{w}$ is also an affine space, as we now show.

\begin{lemma}\label{lemma:cells2}
Let $w \in S_n$. The opposite Schubert cell $\schub\!^{w}$ is an affine space of dimension $\ell(ww_0) = \ell(w_0) - \ell(w)$. Moreover, if $\beta = \sigma_{j_1}\cdots \sigma_{j_s}$ is a reduced positive lift of $ww_0$ then the map
\[
\C^{\ell(ww_0)} \to \schub\!^{w}, \qquad (t_1, \dots, t_s) \mapsto \cF^{B_{\beta}(t)w_0}
\]
is an isomorphism of affine algebraic varieties.
\end{lemma}
\begin{proof}
This follows from Lemma \ref{lemma:cells1}, as follows. By definition,
\[
\schub\!^{w} = B_{-}wB/B = (w_{0}Bw_{0})wB/B = w_0\!\schub_{w_0w}
\]
so that $\schub\!^{w}$ is an affine space of dimension $\ell(w_0w) = \ell(ww_0)$. Now, $w_0w = (ww_0)^*$, and thus the word $\sigma_{i_1}\cdots\sigma_{i_s}$, where 
$\sigma_{i_k} = \sigma_{j_k}^* = \sigma_{n - j_k}$, for $1 \leq k \leq s$, is a reduced braid word for $\beta(w_0w)$.
Then, according to Lemma \ref{lemma:cells1}, we have a parametrization of $\schub\!^{w}$ by flags associated to matrices of the form:
\[
w_0B_{i_1}^{-1}(t_1)\cdots B_{i_s}^{-1}(t_s), \qquad (t_1, \dots, t_s) \in \C^{s}.
\]
It remains to notice that $B_{i}^{-1}(t) = w_0B_{n-i}(-t)w_0$, that 
is proved by a direct verification.
\end{proof}

\noindent Lemmas \ref{lemma:cells1} and \ref{lemma:cells2} suffice to prove Theorem \ref{thm:richardson}. That is, we now show that if $\beta = \beta_1\beta_2$, with $\beta_1$ and $\beta_2$ both being reduced words, then the braid variety of $\beta$ is isomorphic to the intersection of the corresponding Schubert and opposite Schubert cells.

\subsubsection{Proof of Theorem \ref{thm:richardson}}\label{sssec:thm_rich}
Let us first verify that, in fact, Statements (a) and (b) are equivalent. Indeed, let $A \in \GL(n, \C)$. Then,
\[
\cF^{A} \in \rich{w_1^{-1}}{w_2w_0} = ((w_0Bw_0)(w_2w_0)B/B)\cap(Bw_1^{-1}B/B) 
\]
if and only if
\[
\cF^{w_0A} \in (Bw_2^*B/B)\cap ((w_0Bw_0)(w_0w_1^{-1})B/B) = \rich{w_2^*}{w_0w_1^{-1}},
\]
and the equivalence of Statements (a) and (b) follows. 

Now we show (a), from which (b) will follow by the above equivalence. First, let us verify that the image $\iota(X(\beta_1\beta_2; w_0))$ is indeed in the required intersection, i.e. that for each $z_1 \in \mathbb{C}^{\ell(\beta_1)}$, the flag $\cF^{B_{\beta_1}^{-1}(z_1)}$ belongs to both  cells $\schub_{w_{1}^{-1}}$ and $\schub\!^{w_2w_0}$. By Lemma \ref{lemma:cells1}, and since $\ateb_1 = \beta((w_1)^{-1})$, the matrix $B^{-1}_{\beta_1}(z_1)$ belongs to the affine subspace $M_{w_{1}^{-1}}$, and thus $\cF^{B_{\beta_1}^{-1}(z_1)} \in \schub_{w_{1}^{-1}}$, as needed. For the inclusion $\cF^{B_{\beta_1}^{-1}(z_1)} \in \schub\!^{w_2w_0}$, we note that for each $z_2 \in \mathbb{C}^{l(\beta_2)}$ such that $(z_1, z_2) \in X(\beta_1 \beta_2, w_0)$, we have the identity
	$$
	B_{\beta_{1}}(z_1)B_{\beta_2}(z_2)w_{0} = U
	$$
	\noindent for some upper triangular matrix $U$. This implies that	
	$$
	B_{\beta_1}^{-1}(z_1) = B_{\beta_{2}}(z_2)w_{0}U^{-1},
	$$
	
	\noindent and, since $U^{-1}$ is upper triangular, we conclude that $\cF^{B_{\beta_1}^{-1}(z_1)} = \cF^{ B_{\beta_{2}}(z_2)w_{0}U^{-1}} = \cF^{B_{\beta_{2}}(z_2)w_{0}}$. Then Lemma \ref{lemma:cells2}, together with the observation that $\beta_2$ is a reduced braid word for $(w_2w_0)w_0$, shows that the flag $\cF^{B_{\beta_{2}}(z_2)w_0}$ belongs to $\schub\!^{w_2w_0}$, as needed.  

 Second, in order to show that $\iota$ restricts to a bijection, consider a flag $\cF \in \schub_{w_{1}^{-1}}\cap \schub\!^{w_2w_0}$. By Lemma \ref{lemma:cells1}, we can find a unique element  $z_1 \in \bC^{\ell_{1}}$ such that $\cF = \cF^{B_{\beta_1}^{-1}(z_1)}$. Finally, there exists a unique element $z_2\in \bC^{\ell_{2}}$ such that $(z_1,z_2) \in X(\beta_1\beta_2; w_0)$. Indeed, since $\cF \in \schub\!^{w_2w_0}$, Lemma \ref{lemma:cells2} implies that there exists a unique $z_2 \in \bC^{\ell_{2}}$ such that $\cF = \cF^{B_{\beta_{2}}(z_2)w_{0}}$. So $B_{\beta_1}^{-1}(z_1) = B_{\beta_2}(z_2)w_0U$ for an upper triangular matrix $U$ and the result follows. \hfill$\Box$

\subsubsection{Proof of Theorem \ref{thm:rich vs juggling intro}}\label{sect:rich vs juggling} We are now in position to prove Theorem \ref{thm:rich vs juggling intro} from the introduction. Recall that a positroid pair $(u,w)$ yields the open positroid variety $\Pi_{u, w} \subseteq \Gr(k,n)$ in the Grassmannian. By \cite[Theorem 5.9]{KLS}, there exists an algebraic isomorphism $\Pi_{u, w} \cong\, \rich{w}{u}$ between the positroid stratum for $(u,v)$ and the Richardson variety $\rich{w}{u}$. By Corollary \ref{cor:richardson}, we have the algebraic isomorphism 
$\rich{w}{u} \cong X(\beta(u^{-1}w_0)\beta(w^*))$ given by the restriction of $\jmath$,
where, as usual, $\beta(u^{-1}w_0)$, 
$\beta(w^*)$
are positive braid lifts of their corresponding arguments. The composition of these two isomorphisms proves Part (i).\\

For Part (ii), we use Section \ref{sec:Legendrian}. In particular, that both $X(R_n^+(u,w))$ and $X(J_{k}(f); w_{0,k}) \times (\mathbb{C}^{*})^{d}$ are Legendrian invariants associated to the Legendrian links $\La(u,w)$ and $\La(f)$, respectively. For that, we use \cite[Section 5.1]{CasalsNg}, which explicitly describes the Legendrian contact dg-algebra $A_{\Lambda(\beta)}$ for a Legendrian $(-1)$-closure $\Lambda(\beta)$ in terms of braid matrices. It follows from \emph{loc.~cit.}, similarly to \cite[Section 2.6]{CGGS}, that there is an algebraic isomorphism $X(\beta)\cong\mbox{Spec } H^0(A_{\Lambda(\beta\Delta)})$, where we choose one marked point per strand. By \cite[Theorem 3.4]{Chekanov}, the quasi-isomorphism type of $A_{\Lambda(\beta)}$ is a Legendrian isotopy invariant. Therefore $X(\beta)\cong X(\beta')$ if $\La(\beta)$ is Legendrian isotopic to $\La(\beta')$.

\noindent Let $\overline{\La(f)}\sse(\R^3,\xi_{st})$ be the Legendrian link given by the unlinked union of $\La(f)$ with $m$ unlinked max-tb Legendrian unknots, each of the unknots with one unique marked point. Then
$$\mbox{Spec } H^0(A_{\overline{\La(f)}})\cong \mbox{Spec } H^0(A_{\La(f)})\times (\C^*)^{m}.$$
Indeed, $H^0(A_{\Lambda(e)})\cong \C[t,t^{-1}]$ for the 1-stranded braid $e$ with one marked point $t\in\C^*$, which corresponds to a max-tb Legendrian unknot with one marked point, and $\mbox{Spec } H^0(A_{\La_1\cup\La_2})\cong \mbox{Spec } H^0(A_{\La_1})\times \mbox{Spec } H^0(A_{\La_2})$ is a Cartesian product if the components $\La_1,\La_2$ of the link $\La_1\cup\La_2$  are unlinked from each other. 
By Theorem \ref{thm:Leg_allsame}, $\La(u,w)$ is Legendrian isotopic to $\overline{\La(f)}\sse(\R^3,\xi_{st})$ for some $m\in\N$. The number $m$ is determined by the number of destabilizations 
as in Lemma \ref{lem:Markov}.(2) in Subsection \ref{ssec:MarkovMoves}, equivalently the number of destabilizations using Lemma \ref{lem:Markovsmooth}.(2) in Subsection \ref{ssec:destabilization}. Following the proof of Theorem \ref{thm:Richardson_to_juggling} gives $m=n-k-\varphi$, where $\varphi$ is the number of fixed points of $f$. Therefore, we obtain  the following sequence of isomorphisms:
\begin{align*}
X(R_n^+(u,w))&\cong\mbox{Spec } H^0(A_{\La(u,w)})\cong\mbox{Spec } H^0(A_{\overline{\La(f)}})\cong\\
&\cong\mbox{Spec } H^0(A_{\La(f)})\times (\C^*)^{n-k-\varphi}
\cong X(J_{k}(f)) \times (\mathbb{C}^{*})^{n-k-\varphi}.
\end{align*}
This proves Part (ii) and thus finishes the argument for Theorem \ref{thm:rich vs juggling intro}.\hfill$\Box$

\color{black}


\subsection{Brick manifolds and compactifications of braid varieties}\label{ssec:compactify}
This subsection discusses smooth compactifications of braid varieties. Let $\beta=\sigma_{i_{1}}\cdots \sigma_{i_{\ell}}\in\cB_n^+$ be a positive braid word. We first define the brick manifold $\brick(\beta)$ associated to $\beta$, following \cite{Escobar}. These brick varieties $\brick(\beta)$ will provide natural smooth compactifications of our braid varieties $X(\beta)$.

By definition, the Bott-Samelson variety $\BS(\beta)$ associated to $\beta$ is the moduli space of collections of flags $(\cF^{0}, \dots, \cF^{\ell})\in\Fl_n$ such that $\cF^{0}$ is the standard flag and either $\cF^{j} = \cF^{j+1}$, or the two contiguous flags $\cF^{j},\cF^{j+1}$ differ precisely in the $i_{j+1}$-subspace. This projective variety $\BS(\beta)$ contains a natural subvariety, called the open Bott-Samelson variety $\OBS(\beta)$ in \cite{CGGS}, defined by the additional condition that two contiguous flags must be different, i.e. $\cF^{j} \neq \cF^{j+1}$ for every $j\in[0,\ell - 1]$.

\noindent The Bott-Samelson variety $\BS(\beta)$ admits a natural projection map
$$m_{\beta}: \BS(\beta) \lr \Fl_n,\quad m_{\beta}(\cF^{0}, \dots, \cF^{\ell}):= \cF^{\ell},$$
onto the last, rightmost, flag.

\begin{definition}
Let $\beta\in\cB_n^+$ be a positive braid word. By definition, the {\it brick variety} $\brick(\beta)$ associated to $\beta$ is
$$\brick(\beta) := m_{\beta}^{-1}(\delta(\beta)\cF^{\std}),$$
where $\delta(\beta)\in S_n$ denotes the Demazure product of $\beta$. The associated {\it open brick variety} is defined as
$$\brick^{\circ}(\beta) := m_{\beta}^{-1}(\delta(\beta)\cF^{\std})\cap \OBS(\beta).$$
\hfill$\Box$
\end{definition}

\noindent 
Here the Demazure product $\delta(\beta)\in S_n$ is the (unique) maximal permutation with respect to the Bruhat order such that $\beta$ contains its positive braid lift. The brick manifold $\brick(\beta)$, unlike the braid variety $X(\beta)$, significantly depends on the braid word $\beta\in\cB$, and not only on the braid $[\beta]\in\Br$. 

\noindent Let us prove Theorem \ref{thm: intro brick} in the introduction. Given $\beta = \sigma_{i_{1}}\cdots\sigma_{i_{\ell}}\in\cB_n$, denote its opposite braid word by $\ateb:= \sigma_{i_{\ell}}\cdots \sigma_{i_{1}}$. Braid varieties relate to brick varieties, up to this mirroring, as follows:

\begin{thm}\label{prop:brick} Let $\beta = \sigma_{i_{1}}\cdots\sigma_{i_{\ell}}\in\cB_n$ be a positive braid word, and consider the truncations $\ateb_{j} := \sigma_{i_{l}}\cdots \sigma_{i_{l-j+1}}$, $j\in[1,\ell]$.
The following holds:
	
	\begin{itemize}
		\item[$(i)$] The algebraic map
		$$\Theta:\bC^{\ell}\lr\Fl_n^{\ell+1},\quad (z_1, \dots, z_{\ell}) \mapsto (\cF^{\std}, \cF^{1}, \dots, \cF^{\ell}),$$
		where $\cF^{j}$ is the flag associated to the matrix $B_{\ateb_{j}}^{-1}(z_{\ell - j + 1}, \dots, z_{\ell})$, restricts to an isomorphism
		
		$$\Theta:	X(\ateb; \delta(\beta)) \stackrel{\cong}{\lr} \brick^{\circ}(\beta),$$
		of affine varieties. In particular, the braid variety $X(\ateb; \delta(\beta))$ is smooth.\\
		
		\item[$(ii)$] 
The complement to $X(\ateb; \delta(\beta))$ in $\brick(\beta)$ is a normal crossing divisor. Its components correspond to all possible ways to remove a letter from $\beta$ while preserving its Demazure product.
	\end{itemize}
\end{thm}
\begin{proof}
For Part (i), we first verify $\Theta(X(\ateb; \delta(\beta)))\sse\brick^{\circ}(\beta)$. For that, note that $$B_{\ateb_{j+1}}^{-1}(z_{\ell - j}, \dots, z_{\ell}) = B_{\ateb_{j}}^{-1}(z_{\ell - j + 1}, \dots, z_{\ell})B_{i_{j+1}}^{-1}(z_{\ell-j}),$$
and thus the two flags $\cF^{j}$ and $\cF^{j+1}$ are indeed in position $i_{j+1}$, as required. In order to check that
$$\cF^{B^{-1}_{\ateb}(z_1, \dots, z_{\ell})} = \delta(\beta)\cF^{\std},$$
we observe that we have $B_{\ateb}(z_1, \dots, z_{\ell})\delta(\beta) = U$, where $U$ is an upper triangular matrix, and hence $B^{-1}_{\ateb}(z_1, \dots, z_{\ell}) = \delta(\beta)U^{-1}$, from which this conclusion follows. The fact that the map $\Theta$ restricts to an isomorphism follows from the statement of Lemma \ref{lemma:easy} below. The smoothness claim follows from \cite[Theorem 3.3]{Escobar}.

\noindent For Part (ii), we proceed as follows. For a subset $I \subseteq [1, \ell]$, let $\brick(\beta)^{\circ}_{I} \subseteq \brick(\beta)$ be defined by the conditions that $\cF^{i-1} \neq \cF^{i}$ if and only if $i \in I$, where $\cF^{0} = \cF^{\std}$. For example, $\brick(\beta)^{\circ}_{[1, \ell]} = \brick(\beta)^{\circ}$. Now let $\brick(\beta)_{I} := \overline{\brick(\beta)_{I}^{\circ}}$, which is similarly defined by the condition that $\cF^{i-1} = \cF^{i}$ if $i \not\in I$. Note that $\brick(\beta)_{I} \subseteq \brick(\beta)_{J}$ if $I \subseteq J$, and that $\brick(\beta)_{I}^{\circ}$ is nonempty if and only if $\delta(\beta_{I}) = \delta(\beta)$, where $\beta_{I}$ is the subword of $\beta$ indexed by $I$. Moreover, in this case we have natural isomorphisms $$\brick(\beta_{I})^{\circ} \overset\sim\rightarrow \brick(\beta)_{I}^{\circ}, \qquad \brick(\beta_{I}) \overset\sim\rightarrow \brick(\beta)_{I}.$$
Now we have the following the composition
$$
\brick(\beta) = \brick(\beta)^{\circ} \sqcup \bigcup_{I \subsetneq [1,\ell]}\brick(\beta)_{I} = X(\ateb; \delta(\beta))\sqcup\bigcup_{\substack{I \subsetneq [1, \ell] \\ \delta(\beta_{I}) = \delta(\beta)}}\brick(\beta_{I}).
$$
If $\delta(\beta_{I}) = \delta(\beta)$, then $\brick(\beta_{I})$ is a smooth variety of dimension $|I| - \ell(\delta(\beta))$, see \cite[Theorem 3.3]{Escobar}. Therefore, in this case $$\brick(\beta_{I}) = \brick(\beta)_{I} = \bigcap_{j \not\in I}\brick(\beta)_{[1,\ell]\setminus\{j\}} = \bigcap_{j \not\in I}\brick(\beta_{[1,\ell]\setminus\{j\}}).$$
Hence $\brick(\beta_{I})$ is a complete intersection, and the divisors $\brick(\beta_{[1,\ell]\setminus\{j\}}) \subseteq \brick(\beta)$ intersect transversely, if the intersection is non-empty.
\end{proof}

\begin{lemma}\label{lemma:easy}
	Let us consider an invertible matrix $A \in \GL_{n}$ and $i\in[1,n-1]$. Then, the map
	$$\bC \lr \Fl,\quad z \longmapsto \cF^{AB^{-1}_{i}(z)},$$
	yields an isomorphism from $\bC$ to the set of all flags that are in position $i$ with respect to $\cF^{A}$.
\end{lemma}
\begin{proof}
	Let $\cF := \cF^{A}$, $\cF'$ another flag that is in position $i$ with respect to $\cF$ and consider the projection $\cF_{i}' \lr \cF_{i+1}/\cF_{i-1}$. The latter space has basis $\{a_{i}, a_{i+1}\}$, where $a_{i}, a_{i+1}$ are the $i$th and $(i+1)$st columns of the matrix $A$. Since the image is a one-dimensional subspace that cannot coincide with $\bC a_{i}$, it is of the form $\bC(a_{i+1}+za_{i})$ for a unique $z \in \bC$. The result follows.
\end{proof}

\begin{remark}
Note that Corollary \ref{cor:richardson}  also follows from Theorem \ref{prop:brick} by results of \cite{Escobar}, since an open Richardson variety is an open brick variety for the word considered in Corollary \ref{cor:richardson}. Resolutions of Richardson varieties via fibers of the Bott-Samelson map first appeared in \cite{Balan}, reformulating constructions in \cite{Brion}; the latter were also studied in \cite{KLS2}.\hfill$\Box$
\end{remark}

\begin{remark}
In \cite[Theorem 2.34]{CGGS} we compared braid varieties with the \emph{diagonal} open Bott-Samelson varieties. These are subvarieties of $\OBS(\beta)$ defined by an additional condition $\cF^0 = \cF^l$, but this flag (which is simultaneously the first and the last one) is not fixed. Open brick varieties are defined by a different condition, as one fixes both the first and the last flags. This difference explains the additional appearance of $G$ and $B$ in \cite[Theorem 2.34]{CGGS}.\footnote{Note that the Borel subgroup $B$ is denoted by $\mathcal{B}$ in \cite[Theorem 2.34]{CGGS}.}. 
\hfill$\Box$
\end{remark}

\noindent Since the brick manifold $\brick(\beta)$ depends on the braid word $\beta\in\cB$, and the braid variety $X(\beta)$ does not, Theorem \ref{prop:brick}.(ii) allows us to construct many different compactifications for a given braid variety. These compactification are all smooth and projective and the  compactifying divisor is a smooth normal crossing divisor.

\begin{example}
Let us consider the equivalent braid words
$$\beta_1=\sigma_1\sigma_2\sigma_1\sigma_2\sigma_1,\quad \beta_2=\sigma_1\sigma_2\sigma_2\sigma_1\sigma_2.$$
In both cases, the braid varieties are algebraic tori
$$X(\ateb_1;w_0)\cong X(\ateb_2;w_0)\cong (\bC^*)^2.$$
The variety $\brick(\beta_1)$ has $X(\ateb_1;w_0)$ as an open stratum, 
5 codimension-1 strata (all isomorphic to $\bC^*$), and 5 
codimension-2 strata, which are points. In fact, $\brick(\beta_1)$ is a toric degree 5 del Pezzo surface, i.e. the toric variety associated to the pentagon, and these various strata correspond to toric 
orbits. In contrast, $X(\sigma_1\sigma_2^3;w_0)$ is empty and there are only 4 codimension-1 strata and 4 codimension-2 strata in $\brick(\beta_2)$. In fact, $\brick(\beta_2)\cong\PP^1\times \PP^1$, which is a different toric variety.\hfill$\Box$
\end{example}

Theorem \ref{prop:brick}  also brings forth a connection between braid varieties and the subword complexes introduced in \cite{KM}. For that, note that the open brick variety $\brick^{\circ}(\beta)$ is the higher-dimensional stratum in the stratification of the brick variety $\brick(\beta)$ given in \cite[Theorem 24]{Escobar}, and we claim that all other strata of this stratification can also be realized as braid varieties, as follows.

Let $\beta'$ be a subword of $\beta$ such that the Demazure product of $\beta'$ coincides with that of $\beta$. Then, $X(\ateb'; \delta(\beta))$ is a strata of $\brick(\beta)$, given by the conditions that $\cF^{j} = \cF^{j+1}$ whenever $i_{j+1} \not\in \beta'$. This stratification is dual to the \emph{subword complex} $(\beta, \delta(\beta)),$ as defined in \cite{KM}. Subword complexes are defined for arbitrary pairs $(\beta, \pi),$ where $\pi$ is an element of a finite Coxeter group and $\beta$ is a word in simple generators; the latter can also be seen as a positive braid word in the corresponding braid group. In \cite{KM} it is proven that a subword complex is spherical if and only if $\delta(\beta) = \pi$. Thus, brick manifolds bijectively correspond to spherical subword complexes and they are stratified by braid varieties, with the adjacency of strata described by the dual complexes.

\begin{example} Let us choose $n=2$ and $\beta=\s_1^3$. The braid variety $X(\s_1^3;s_1)$ is a smooth surface in $\bC^3$ defined by the equation
$$X(\s_1^3;s_1) = \{(z_1,z_2,z_3)\in\bC^3: z_1 + z_3(1 + z_1 z_2) = 0\}.$$
If $1 + z_1 z_2 = 0,$ we get $z_1 = 0,$ and thus come to a contradiction. Therefore, $1 + z_1 z_2 \neq 0.$ Then $z_3 = - \frac{z_1}{1 + z_1 z_2}$, and so $X(\s_1^3;s_1)$ is isomorphic to the complement
$$Y = \{(z_1, z_2)\in\bC^2:1 + z_1 z_2 \neq 0\}$$
of a smooth hyperbola in $\bC^2$. The braid variety $X(\s_1^2;s_1)$ is isomorphic to $\bC^*$ and $X(\s_1;s_1)$ is a point. The corresponding brick manifold is $\brick(\s_1^3)\cong\PP^1\times \PP^1$, and these different braid varieties stratify it as follows. Consider the homogeneous coordinates $(x,y)\in\PP^1\times\PP^1$, and denote $[0:1]$ by $0$ and $[1:0]$ by $\infty$. Note that the homogenized hyperbola $\bar{C}$ contains the
points $(0,\infty)$ and $(\infty,0)$. The stratification of the brick variety $\PP^1\times \PP^1$ given by the braid varieties has the following seven strata:
	
	\begin{enumerate}
		\item[-] One 2-dimensional stratum $Y$, which is the complement of the smooth hyperbola in $\bC^2$.\\
		
		\item[-] Three 1-dimensional strata, each isomorphic to $\bC^*$. Two such strata are given by
		$$S_1:=\{(x,y)\in\PP^1\times \PP^1: x=\infty,\quad y\neq 0,\infty\},\quad S_2:=\{(x,y)\in\PP^1\times \PP^1: x\neq0,\infty,\quad y=\infty\},$$
		and the third one $S_3:=\{(z_1, z_2)\in\bC^2:1 + z_1 z_2 = 0\}$ is the affine hyperbola itself.\\
		
		\item[-] Three 0-dimensional strata, each a point: $p_1:=(\infty,0),p_2:=(0,\infty)$ and $p_3:=(\infty,\infty)$.
	\end{enumerate}

 \noindent The subvarieties $\overline{S_1}:=S_1\cup \{p_1,p_3\}$, $\overline{S_2}:=S_2\cup \{p_2,p_3\}$ and $\overline{S_3}:=S_3\cup \{p_1,p_2\}$ are all smooth projective lines $\P^1\sse \brick(\s_1^3)$, transversely intersecting each other. The compactifying divisor of $X(\s_1^3;s_1)\sse\brick(\s_1^3)$ is the union $\overline{S_1}\cup\overline{S_2}\cup \overline{S_3}$. Its dual complex, which records the  intersections of these strata, coincides with the spherical subword complex for $\beta=\s_1^3$.\hfill$\Box$
\end{example}

\subsection{Equivariant homology and Lefschetz property for braid varieties}\label{ssec:equiv_hom} Theorem \ref{prop:brick}.(ii) also allows us to compute the equivariant Borel-Moore homology of braid varieties, as follows. Given an algebraic variety $X$ with an action of a torus $T$, we denote by $H^{\BM}_*(X)$ (resp. $H^{\BM,T}_*(X)$) the  Borel-Moore homology (resp. $T$-equivariant Borel-Moore homology) of $X$. Similarly, we denote by $H^{*}_{c}(X)$ (resp. $H^{*}_{T,c}(X)$) the compactly supported cohomology (resp. $T$-equivariant compactly supported cohomology) of $X$. 


	
\noindent By Theorem \ref{prop:brick}.(ii), we can express the weight filtration on Borel-Moore homology of the braid variety $X(\ateb;w_0)$ in terms of homologies of brick manifolds 
$\brick(\beta_I)$ as follows. Consider the big complexes
$$C_{\bullet}:=\bigoplus_{I}H^{\BM}_*(\brick(\beta_I)),\quad
C_{\bullet}^T:=\bigoplus_{I}H_*^{\BM,T}(\brick(\beta_I))
$$
where the summation is over the subsets $I \subseteq [1, \ell]$ such that $\delta(\beta_I)=\delta(\beta)$, and the differential is given by inclusions $\brick(\beta_I) \hookrightarrow \brick(\beta_J)$
for $I\subset J,|J|=|I|+1$. The torus $T=(\mathbb{C}^*)^{n-1}$ has a natural action on braid varieties as in \cite[Section 2.2]{CGGS}, and the brick compactification is $T$-equivariant.
One can also consider the complexes 
$$C^{\bullet}:=\bigoplus_{I}H^*(\brick(\beta_I)),\quad
C^{\bullet}_T:=\bigoplus_{I}H^*_T(\brick(\beta_I)),
$$
where all inclusions of brick manifolds correspond to maps in cohomology (resp. equivariant cohomology) in the opposite direction.


\begin{prop}\label{prop:homology_braid}
In the notation above, we have the isomorphisms
$$	H_*(C_{\bullet})=\mathrm{gr}^{W}H^{\BM}_*(X(\ateb;w_0)),\quad 
H_*(C_{\bullet}^T)=\mathrm{gr}^{W}H_*^{\BM,T}(X(\ateb;w_0)),
$$
and
$$
H_*(C^{\bullet})=\mathrm{gr}^{W}H_{c}^*(X(\ateb;w_0)),\quad 
H_*(C^{\bullet}_T)=\mathrm{gr}^{W}H^*_{c,T}(X(\ateb;w_0)),
$$
where $\mathrm{gr}^{W}$  is the associated graded for the weight filtration.\\
\end{prop}

\begin{proof}
By Theorem \ref{prop:brick}.(ii), the homology of the complex $C_{\bullet}$ (resp. $C_{\bullet}^T$) is the $E_2$-page of the spectral sequence computing the Borel-Moore homology (resp. $T$-equivariant Borel-Moore homology) of $X(\ateb;w_0)$. Since all brick manifolds are smooth and projective, the weight filtrations in their homology agree with the homological gradings. By \cite{D,D2}, the higher differentials preserve weights and the spectral sequence collapses at the $E_2$ page. Therefore we obtain the equality
	$$
	H_*(C_{\bullet})=\mathrm{gr}^{W}H^{\BM}_*(X(\ateb;w_0)),\quad H_*(C_{\bullet}^T)=\mathrm{gr}^{W}H_*^{\BM,T}(X(\ateb;w_0)).
	$$
The proofs for complexes $C^{\bullet},C^{\bullet}_T$ are similar.    
\end{proof}

In order to describe the complexes $C^{\bullet}_T$ in more detail, we need notation from the theory of Soergel bimodules, see \cite{soergelbook} for details and references. Let 
$$
R:=\C[x_1,\ldots,x_n]=H^{*}_T(\mathrm{pt}).
$$
We need the following $R$-$R$- bimodules. For $1\le i\le n-1$, we write
\begin{equation}
\label{eq: Bi}
\SB_{i}:=\frac{\C[x_1,\ldots,x_n,x'_1,\ldots,x'_n]}{x_{i}+x_{i+1}=x'_{i}+x'_{i+1},x_{i}x_{i+1}=x'_{i}x'_{i+1},x_m=x'_m\quad (m\neq i,i+1)}.
\end{equation}
Note that we have a natural map $b_{i}:\SB_{i}\to R$ defined by 
$$
b_{i}(f(x_1,\ldots,x_n,x'_1,\ldots,x'_n)=f(x_1,\ldots,x_n,x_1,\ldots,x_n).
$$
Given a braid word $\beta=s_{i_1}\ldots s_{i_r}$, we define the {\it Bott-Samelson bimodule}
$$
\SB_{\beta}=\SB_{i_1}\bigotimes_{R}\cdots \bigotimes_{R} \SB_{i_r}.
$$
Alternatively, one can write $\SB_{\beta}$ as a quotient of the polynomial ring $\C\left[x^{(i)}_j,0\le \ell \le r,1\le j\le n\right]$ by the relations of the form \eqref{eq: Bi} between $(x^{(\ell)}_{j})_{j \in [1,n]}$ and $(x^{(\ell+1)}_{j})_{j \in [1,n]}$ for all $\ell \in [0, r - 1]$. 

\begin{prop}
\label{prop: homology bs} In the notation above, the following hold:
\begin{itemize}
    \item[(a)] We have the isomorphism $H^*_T(\BS(\beta))\simeq \SB_{\beta}$.\\
\item[(b)] If $\beta'$ is a subword of $\beta$  then we have a commutative diagram
$$
\begin{tikzcd}
H^*_T(\BS(\beta)) \arrow{r}{\simeq} \arrow{d}& \SB_{\beta} \arrow{d}\\
H^*_T(\BS(\beta')) \arrow{r}{\simeq}  & \SB_{\beta'},\\
\end{tikzcd}
$$
where the vertical arrows correspond to the inclusion $\BS(\beta')\hookrightarrow \BS(\beta)$ and to the composition of maps $b_i$ respectively.
\end{itemize}
\end{prop}

\begin{proof}
These facts are well-established, cf.~\cite{soergelbook}, but we review the key geometric ideas for completeness. The variables $x^{(\ell)}_{j}$ correspond to the first Chern classes of the tautological line bundles $\cL^{(\ell)}_{j}=\cF^{\ell}_{j}/\cF^{\ell}_{j-1}$ on the Bott-Samelson variety $\BS(\beta)$. Since $\BS(\beta)$ is a tower of $\mathbb{P}^1$-bundles, it can be proven that
$x^{(\ell)}_{j}$ generate the equivariant cohomology. If two flags $\cF$ and $\cF'$ are in position $s_i$, and $\cL_i=\cF_i/\cF_{i-1},\ \cL'_i=\cF'_i/\cF'_{i-1}$, then the rank two bundle 
$$
\cF_{i+1}/\cF_{i-1}=\cF'_{i+1}/\cF'_{i-1}
$$
is filtered both by $\cL_i,\cL_{i+1}$ and by $\cL'_i,\cL'_{i+1}$, hence the first Chern classes
$x_i=c_1(\cL_i),\ x'_i=c_1(\cL'_i)$ satisfy the relations \eqref{eq: Bi}. This implies Part (a).

For Part (b), consider the inclusion of the diagonal 
$\BS(1) \hookrightarrow \BS(s_i)$ and note that the corresponding map in equivariant cohomology agrees with $b_i:\SB_i\to R$.
\end{proof}

To describe the homology of $
\brick(\beta)$, we need to introduce the category of Soergel bimodules $\SBim_n$. It is defined as the smallest full subcategory of $R$-$R$ bimodules which contains $\SB_i$ and $R$ and is closed under taking tensor products, direct sums, direct summands and grading shifts. 

Any Soergel bimodule decomposes as a direct sum of {\em indecomposable} Soergel bimodules, which are indecomposable direct summands in the Bott-Samelson bimodules. The main structural Theorem of Soergel \cite{Soergel07} states that in fact indecomposable Soergel bimodules $\SB_{\sigma}$ (up to isomorphism and grading shifts) are indexed by the elements $\sigma\in S_n$. In particular, simple reflections $s_i$ correspond to $\SB_i$ while the element $w_0\in S_n$ corresponds to Soergel bimodule
$$
\SB_{w_0}=\frac{\C[x_1,\ldots,x_n,x'_1,\ldots,x'_n]}{e_i(x_1,\ldots,x_n)=e_i(x'_1,\ldots,x'_n),\ i=1,\ldots,n}.
$$
Here $e_i$ are elementary symmetric functions. 
We refer to \cite{soergelbook} for more details. 

\begin{definition}
Given a Soergel bimodule $\SB$, we denote the graded multiplicity of $\SB_{\sigma}$ in $\SB$ by $m_{\sigma}(q)$:
$$
\SB=\bigoplus_{\sigma\in S_n}m_{\sigma}(q)\SB_{\sigma}.
$$
Here $q$ corresponds to the grading shift. Furthermore, we define 
$$
\pi_{w_0}(\SB)=m_{w_0}(q)\SB_{w_0},
$$
this is the $\SB_{w_0}$-component of $\SB$.
\end{definition}

\begin{lemma}
\label{lem: homology brick}
Assume that $\delta(\beta)=w_0$, and let 
$$
\overline{\SB_{\beta}}=\pi_{w_0}(\SB_{\beta})/\left(x^{(0)}_{j}=x^{(r)}_{n+1-j},\quad 1\le j\le n\right).
$$
Then the following hold:
\begin{itemize}
    \item[(a)] We have an isomorphism $H^*_T(\brick(\beta))\simeq\overline{\SB_{\beta}}$.\\
\item[(b)] If $\beta'$ is a subword of $\beta$ with $\delta(\beta')=\delta(\beta)=w_0$ then we have a commutative diagram
$$
\begin{tikzcd}
H^*_T(\brick(\beta)) \arrow{r}{\simeq} \arrow{d}& \overline{\SB_{\beta}} \arrow{d}\\
H^*_T(\brick(\beta')) \arrow{r}{\simeq}  & \overline{\SB_{\beta'}}.
\end{tikzcd}
$$
\end{itemize}
\end{lemma}

\begin{proof}
Part (a) follows from the results of \cite{Harterich,Shchigolev}. 
In particular, by \cite[Corollary 6.5(b)]{Harterich} and \cite[Corollary 4.12]{Shchigolev}, the restriction map in equivariant cohomology
$$
H^*_T(\BS(\beta))\to H^*_T(\brick(\beta))
$$
is surjective. 
Part (b) is immediate from Proposition \ref{prop: homology bs}(b).
\end{proof}

Next, we associate {\em Rouquier complexes} \cite{Rouquier06} of Soergel bimodules to braids as follows. Define
$$
T_i=[\SB_i\xrightarrow{b_i} R],\quad T_i^{-1}=[R\xrightarrow{b_i^*} \SB_i]
$$
    where the map $b_i$ is as above and $b^*_i$ sends $1$ to $x_i-x'_{i+1}$. It is known \cite{Rouquier06,soergelbook} that $T_i,T^{-1}_i$ satisfy braid relations up to homotopy:
$$
T_i\otimes_R T_i^{-1}\simeq R, \quad T_{i+1}\otimes_R T_i\otimes_R T_{i+1}\simeq T_{i}\otimes_R T_{i+1}\otimes_R T_i,\quad T_i\otimes_R T_j\simeq T_j\otimes_R T_i\ (|i-j|\ge 2).
$$
This means that to any braid $\gamma$ one can associate a complex $T_{\gamma}$ by multiplying $T_i,T_i^{-1}$. We can use Rouquier complexes to give an alternative characterization of $\overline{\SB_{\beta}}$, as follows. For a Soergel bimodule $\SB$, let us denote

$$\overline{\SB}:=\pi_{w_0}(\SB)/\left(x^{(0)}_{j}=x^{(r)}_{n+1-j},\quad 1\le j\le n\right).$$

\begin{lemma}
\label{lem: hom from w_0}
Let $\SB$ be a Soergel bimodule. Then
$$\Hom(T_{w_0},\SB)\simeq \overline{\SB}.$$
\end{lemma}

\begin{proof}
Observe that $\Hom(T_{w_0},-)$ is an additive functor and it is enough to compute it on indecomposable objects $\SB_{\sigma}$. By the main result of \cite{LW}, cf.~also \cite[Corollary A.6]{GHMN}, for $u,v\in S_n$ one has $\Hom(T_u,T_v)=0$ unless $u\preceq v$ in Bruhat order. In particular, $\Hom(T_{w_0},T_v)=0$ unless $v=w_0$ where $\Hom(T_{w_0},T_{w_0})\simeq R$. 

\noindent Note that we can present an indecomposable bimodule $\SB_{\sigma}$ as a twisted complex built from $T_v$ for $v\preceq \sigma$, where there is exactly one copy of $T_{\sigma}$, cf.~\cite{LW,soergelbook}. This implies $\Hom(T_{w_0},\SB_{\sigma})=0$ for $\sigma\neq w_0$ and $\Hom(T_{w_0},\SB_{w_0})\simeq \Hom(T_{w_0},T_{w_0})\simeq R \simeq \overline{\SB_{w_0}}$. 
\end{proof}

\noindent By combining the above results, we can relate the homology of braid varieties to 
Khovanov-Rozansky link homology \cite{KhS,KR}.
The following result gives an alternative proof of 
\cite[Corollary 4]{trinh2021hecke}.

\begin{thm}
Let $\beta$ be a positive braid and suppose that $\delta(\beta)=w_0$. Then we have an isomorphism
$$\mathrm{gr}^{W}H_{c,T}^{*}(X(\ateb;w_0))\simeq \HHH^{a=n}(\beta\Delta)\simeq \HHH^{a=0}(\beta\Delta^{-1}),
  $$
where $\HHH^{a=0}$, resp.~ $\HHH^{a=n}$, is the Khovanov-Rozansky homology in lowest, resp.~highest, $a$-degree.
\end{thm}

\begin{proof}
We prove this by combining the above results. By Proposition \ref{prop:homology_braid} we get
$$
\mathrm{gr}^{W}H_{c,T}^{*}(X(\ateb;w_0))\simeq H^*(C^{\bullet}_{T})
$$
where 
$$
C^{\bullet}_{T}=\bigoplus_{I} H^*_T(\brick(\beta_{I})),
$$
the sum runs over subsets $I$ such that $\delta(\beta_I)=w_0$ and the differential is given by the restriction maps in equivariant cohomology corresponding to the inclusions $\brick(\beta_I) \hookrightarrow \brick(\beta_J)$. 

\noindent By Lemma \ref{lem: homology brick} we can replace each 
$H^*_T(\brick(\beta_{I}))$ by $\overline{\SB_{\beta_I}}$ and the maps between these by the maps $b_i$. By Lemma \ref{lem: hom from w_0} we get  $\overline{\SB_{\beta_I}}\simeq \Hom(T_{w_0},\SB_{\beta_I})$,
and $\Hom(T_{w_0},\SB_{\beta_I})=0$ whenever $\delta(\beta_I)\neq w_0$. 

To sum up, we can replace the complex $C^{\bullet}_{T}$ by $\bigoplus_{I}\Hom(T_{w_0},\SB_{\beta_I})$ where the sum runs over all possible subsets $I$, and the restriction maps are induced by the maps $b_i$. This latter complex is precisely
$$
\bigoplus_{I}\Hom(T_{w_0},\SB_{\beta_I})=\Hom\left(T_{w_0},[\SB_{i_1}\xrightarrow{b_{i_1}} R]\otimes_{R} \cdots \otimes_{R} 
[\SB_{i_r}\xrightarrow{b_{i_r}} R]\right)=
$$
$$
\Hom(T_{w_0},T_{i_1}\otimes_R \cdots \otimes_R 
T_{i_r})=\Hom(T_{w_0},T_{\beta})=\Hom(R, T_{\beta}T_{w_0}^{-1})=\HHH^{a=0}(\beta\Delta^{-1}).
$$
Finally, by the main result of \cite{GHMN} we conclude $\HHH^{a=0}(\beta\Delta^{-1})\simeq \HHH^{a=n}(\beta\Delta)
$.
\end{proof}

Finally, \cite{Mellit} studied the \emph{curious Lefschetz property}, as defined by \cite{HRV}, for cohomology rings of character varieties. In \cite{Mellit}, a stratification of a certain vector bundle over a given character variety is shown to be equal to a stratification by vector bundles over braid varieties. In particular, \cite{Mellit} proves the curious Lefschetz property for each braid variety. Therefore, Corollary \ref{cor:richardson} implies:

\begin{cor} Let $u, w \in S_{n}$ be such that $u \leq w$ in Bruhat order.
Then the open Richardson variety  $\rich{w}{u}$ for $G:=\mbox{SL}(n,\C)$ 
satisfies the curious Lefschetz property.\hfill$\Box$
\end{cor}

\noindent This result was first conjectured by T. Lam and D. Speyer \cite[Section 1.5.1]{LS}, see also further discussion in a recent paper by P. Galashin and T. Lam \cite{GL}.

\section{Concluding remarks on cluster structures and Legendrian links}
\label{sec:cluster}

Let us conclude this article with a few comments and conjectures on cluster structures. The reader is referred to \cite{FominZelevinsky_ClusterI,FominZelevinsky_ClusterII} and \cite{FockGoncharov_ModuliLocSys,FockGoncharovII} for the necessary preliminaries on cluster structures, and see also \cite{FG,Fraser, GHKK, LS, Mills, Muller, MS,Qin}. See \cite[Section 2.8]{CasalsWeng22} and \cite[Section 6]{FWZ} 
for a discussion on cluster structures on spaces, and we refer to \cite{CasalsHonghao,CG23}  and references therein for recent developments.\\

\noindent Let $\La\sse(\R^3,\xi_{st})$ be a Legendrian link and $T\sse\La$ a set of marked points, with at least one marked point per component. From now onward, unless otherwise specified, we assume that $\La$ is isotopic to the Legendrian lift $\La_\beta$ of a $(-1)$-closure of a positive braid $\beta$, which we always assume has $w_0$ as a suffix. Let us denote by $\mathfrak{M}(\La,T)$ either of the two smooth affine varieties:

\begin{enumerate}
    \item $\mathfrak{M}(\La,T)=\mbox{Spec}(H^0(\SA(\La,T)))$, where $\SA(\La,T)$ denotes the (commutative) Legendrian contact dg-algebra associated to $(\La,T)$. This dg-algebra was defined in \cite{Chekanov}, see \cite{EtnyreNg18} for a survey.\\

  \item $\mathfrak{M}(\La,T)$ is the moduli stack of pseudo-perfect objects of microlocal rank one in $\mbox{Sh}_\La(\R^2,T)$ with a trivialization of the microlocal functor on $\La\setminus T$. See \cite[Section 2.7]{CasalsWeng22} for more details.
\end{enumerate}

\noindent These varieties are isomorphic, cf.~\cite[Section 4]{CasalsWeng22} and \cite[Prop.~5.2]{CasalsNg}, or \cite{NRSSZ}. If $T$ is chosen so that every strand of the positive braid word $\beta$ associated to $\La=\La_\beta$ has exactly one marked point, then $\mathfrak{M}(\La_\beta,T)$ is isomorphic to 
a 
braid variety, 
cf.~\cite[Section 2]{CGGS}. 
In \cite{cgglss} we proved that the algebra of regular functions $\C[\mathfrak{M}(\La,T)]$ is a cluster algebra satisfying many interesting properties, such as local acyclicity and the existence of a Donaldson-Thomas transformation, see also \cite{CasalsWeng22,glsbs}. This fact was conjectured in the initial version of this manuscript, prior to ibid.

Let us consider a contact isotopy $\{\phi_t\}_{t\in[0,1]}$, so that each $\phi_t\in\mbox{Cont}^c(\R^3,\xi_{st})$ is a compactly supported contactomorphism and $\phi_0=\mbox{Id}$. Note that any Legendrian isotopy extends to a contact isotopy \cite{Geiges08}. Then the contact isotopy $\{\phi_t\}_{t\in[0,1]}$ induces an algebraic isomorphism of affine varieties
$$\Phi_t:\mathfrak{M}(\La,T)\lr\mathfrak{M}(\phi_t(\La),\phi_t(T)).$$
Indeed, this is proved in \cite{Chekanov} for the Floer-theoretic approach with the dg-algebra, 
and holds by \cite{GKS_Quantization} in the sheaf-theoretic approach. We conjecture that these are not just isomorphisms of affine varieties, but also quasi-cluster isomorphisms. See ~\cite{Fraser} for details on quasi-cluster structures and their morphisms. The precise conjecture reads as follows:

\begin{conj}[Legendrian invariance of cluster structures]
Let $\La_0$ and $\La_1$ be two Legendrian links in $(\R^3,\xi_{st})$ and let $\{\phi_t\}$ be a contact isotopy $\phi_t\in\Cont^c(\R^3,\xi_{st})$, ${t\in[0,1]}$, such that $\phi_1(\La_0)=\La_1$. Consider a set of marked points $T_0\sse\La_0$, with at least one marked point per component and write $T_1:=\phi_1(T)$. Then the induced algebraic morphism
$$\Phi_t^*:\C[\mathfrak{M}(\La_1,T_1)]\lr\C[\fM(\La_0,T_0)]$$
is a quasi-cluster isomorphism, where both target and domain are endowed with the cluster structures constructed in \cite{cgglss}. That is, the quasi-cluster isomorphism type of the cluster algebra $\C[\mathfrak{M}(\La,T)]$ is a Legendrian isotopy invariant of a Legendrian link $\La\sse(\R^3,\xi_{st})$ with a set of marked points $T\sse\La$.
\end{conj}

Following the symplectic geometry behind the construction of the cluster structures in \cite{cgglss,CasalsWeng22}, it is natural to expect a generalization of our previous work and the above conjecture to any Legendrian link $\La\sse(\R^3,\xi_{st})$, even if it is not the $(-1)$-closure of a positive braid. A first challenge is structural: the spaces $\fM(\La,T)$ are in general just $D^-$-stacks 
and there is no definition known to us of what it means for a derived stack to admit a cluster structure. More generally, both the dg-category $\mbox{Sh}_\La(\R^2,T)$ of sheaves above and the stable tame isomorphism type of $\SA(\La,T)$ are Legendrian invariant. The geometry indicates that they should admit a sort of ``cluster structure'', but this is an algebraic structure which is yet to be specified. One challenge is that we do not know what it means for a (finite type) dg-category to admit a cluster structure\footnote{A desirable feature would be that the ``ring of regular functions'' of its derived stack of pseudo-perfect objects -- e.g.~supposing it is isomorphic to an affine scheme -- is a cluster algebra, again in the appropriate sense.} or for a dg-algebra to be a ``cluster'' dg-algebra. The characteristic algebra, as introduced in \cite[Section 3]{Ng03}, or the graded augmentation variety, modulo dg-homotopies, might be reasonable objects to study from this viewpoint. In either case, based on the geometry, we conjecture that such generalizations of cluster structures exist and are Legendrian invariant, up to quasi-cluster isomorphisms.

\bibliographystyle{plain}
\bibliography{main}

\end{document}